\documentclass{amsart}[12pt]
\usepackage{anysize}
\marginsize{3,5cm}{2,5cm}{2,5cm}{2,5cm}
\usepackage{amssymb}
\usepackage{amsthm}
\usepackage{amsmath}
\usepackage{graphicx}
\usepackage{latexsym,amsfonts}
\usepackage{xcolor}
\definecolor{RED}{rgb}{1,0,0}
\usepackage{enumerate}
\usepackage{float}
\usepackage{caption}
\usepackage{subcaption}
\usepackage{geometry}
\usepackage{anyfontsize}
\usepackage{t1enc}

\usepackage{enumitem}
\setlist[itemize,1]{label=$\circ$}

\usepackage{lipsum} 

\newtheorem{thm}{Theorem}[section]
\newtheorem{cor}[thm]{Corollary}
\newtheorem{lem}[thm]{Lemma}
\newtheorem{prop}[thm]{Proposition}
\theoremstyle{remark}
\newtheorem{rem}[thm]{Remark}
\theoremstyle{definition}
\newtheorem{defs}[thm]{Definition}

\theoremstyle{plain}
\newcommand{\thistheoremname}{}
\newtheorem{genericthm}[thm]{\thistheoremname}

  \newtheorem*{genericthm*}{\thistheoremname}
\newenvironment{namedthm*}[1]
  {\renewcommand{\thistheoremname}{#1}%
   \begin{genericthm*}}
  {\end{genericthm*}}

\newcommand{\mycomment}[1]{}

\author[C.\ Fougeron]{Charles Fougeron}
\address{LAGA - Université Sorbonne Paris Nord, 99 Av. Jean Baptiste Clément,
93430 Villetaneuse, France}
\email{charles.fougeron@math.cnrs.fr}

\author[S.\ Schmidhuber]{Sophie Schmidhuber}
\address{Institut f\"ur Mathematik, Universit\"at Z\"urich, Winterthurerstrasse 190,
CH-8057 Z\"urich, Switzerland}
\email{sophie.schmidhuber@math.uzh.ch}

\author[C.\ Ulcigrai]{Corinna Ulcigrai}
\address{Institut f\"ur Mathematik, Universit\"at Z\"urich, Winterthurerstrasse 190,
CH-8057 Z\"urich, Switzerland}
\email{corinna.ulcigrai@math.uzh.ch}

\title[Dynamical decomposition of GIETs]{Dynamical decomposition of generalized interval exchange transformations}

\begin{document}
\begin{abstract} We develop a renormalization scheme which extends the classical Rauzy-Veech induction used to study interval exchange tranformations (IETs) and allows to study generalized interval exchange transformation (GIETs) $T: [0,1) \to [0,1)$ with possibly more than one quasiminimal component (i.e.~not infinite-complete, or, equivalently, not semi-conjugated to a minimal IET). The renormalization is defined for more general maps that we call \textit{interval exchange transformations with gaps} (g-GIETs), namely~partially defined GIETs which appear naturally as the first return map of $C^r$-flows on two-dimensional manifolds to \textit{any} transversal segment. We exploit this renormalization scheme  to find a decomposition of $[0,1)$ into finite unions of intervals which either contain no recurrent orbits, or contain only recurrent orbits which are closed, or contain a unique quasiminimal.
 This provides an alternative approach to the decomposition results for foliations and flows on surfaces by Levitt, Gutierrez and Gardiner from the $1980s$.

\end{abstract}

\maketitle

\section{Introduction}
In this paper, we provide a contribution to the study of flows on surfaces as well as their Poincar{\'e} maps, namely \emph{generalized interval exchange transformations} (GIETs). We introduce a new  approach to some of the classical decomposition and structure results established by Levitt, Gutierrez and Gardiner in the 1980s, which allow to decompose a surface with a flow into subsurfaces that each contain at most one \emph{quasiminimal} component. Our  approach, which provides both alternative proofs as well as new decomposition results at the level of the Poincar{\'e} map, is based on a new extension of Rauzy-Veech induction, a powerful tool which was introduced in the $1980s$ and exploited in the last $40$ years to study interval exchange maps and orientable measured foliations. Section \S~ 1.1 provides a historical overview over some key results in the study of flows on one hand and interval exchange transformations on the other, in Section \S~\ref{sec:resultsintro} we present an overview of our results and in Section \S~ 1.3 we present an example decomposition of GIET called the \textit{Disco map}.

\subsection{A historical overview}\label{sec:history}
\subsubsection{Flows on surfaces} \label{sec:recurrenceintro}
In 1886, Poincaré became interested in the study of recurrent trajectories both in the plane and on the torus \cite{poincare}. In a statement whose proof was completed by Bendixson in 1901, Poincaré showed that there exist no \emph{non-trivially recurrent} trajectories in the plane, namely~trajectories which are \emph{recurrent} (i.e.~come back infinitely often arbitraritly close to the initial point) but \emph{non-trivial} in the sense that  they are not periodic, or equivalently not closed. However, such trajectories may exist on two-dimensional manifolds: the simplest example is given by the irrational flow on the torus. The closure of such a non-trivially recurrent trajectory is called a \textit{quasiminimal}. Poincaré distinguished between two types of non-trivially recurrent trajectories: those which are \emph{everywhere dense} in a subsurface and those which are \emph{nowhere dense}. He called the latter trajectories \emph{exceptional}, their closure is locally homeomorphic to the product of the Cantor set and a segment. Famous examples of flows with exceptional trajectories were given by Denjoy \cite{Denjoy} in 1932 and Cherry \cite{CherryFlow} in 1937 by modifying a countable number of orbits of the irrational flow on the torus (see also sections \S~ \ref{sec:Denjoyflow} and \S~ \ref{sec:Cherryflow}).

By the end of the 1930s Maier \cite{Maier} proved that every trajectory contained in a quasiminimal is dense therein. Together with a result by Cherry \cite{Cherry} this then implies that an orientable surface of genus $g$ can have at most $g$ distinct quasiminimals. Maier's work inspired the idea of a decomposition of a surface with a flow into '\emph{simpler}' components, i.e a set of curves that divide our original surface into subsurfaces, each of which contains either no quasiminimal or exactly one quasiminimal (see Section \S~\ref{sec:decompositionthm}). G. Levitt proved the existence of such decomposition for flows with saddle equilibria in 1978 \cite{Levitt}. For a wider class of flows the decomposition theorem was established by C. Gardiner \cite{gardiner} in 1985.

\subsubsection{Interval exchange transformations and Rauzy-Veech induction}\label{sec:IETsintro}
In the 1980s, Keane, Rauzy and Veech introduced new tools to study
\emph{interval exchange
transformations} (IETs for short), i.e.~piecewise isometries which appear in connection  to the study of billiards in rational polygons and  Teichmüller dynamics. These maps are intimately related to the study of orientable \emph{measured foliations}, namely foliations on surfaces (with saddle-type singularities with even number of prongs) which carry a non-atomic invariant\footnote{Invariance of the transverse measure refers here to \emph{holonomy invariance}: if a transverse curve is homotoped in such a way that the endpoints slide on trajectories of the flow and the image does not hit saddles, its  transverse measure does not change.}  \emph{transverse measure}. Indeed, the \emph{first return map} of such a flow on a transverse section is, in suitable coordinates, an interval exchange transformation.\footnote{The invariance of the transverse measure by holonomy implies indeed that the Poincar{\'e} map on a transverse segment, restricted to any continuity interval, is an isometry of the induced measure; since the measure is not atomic, up to reparametrization, it can be assumed to be the Lebesgue measure, so that the Poincar{\'e} map is a piecewise translation.}

It was conjectured by Keane, and proved by Masur \cite{Masur} and Veech \cite{Veech} independently, that IETs are 
typically\footnote{Typically is meant here in the measure theoretical sense, i.e.~for Lebesgue-almost every choice of the lenghts of the IETs continuity intervals, as long as the associated permutation is \emph{irreducible}.} uniquely ergodic.
The proof given by Veech uses a renormalization scheme nowadays called \emph{Rauzy-Veech induction}, since it appeared for the first time in the seminal work of Rauzy \cite{Rauzy} and was later developed by Veech himself (in \cite{Veech} and further papers). The renormalization scheme associates to an IET \emph{without connections}, which correspond to a Poincar{\'e} section of a flow without \emph{saddle-connections} (i.e without trajectories which join a saddle singularity to another saddle singularity), the IETs obtained by considering its restriction to smaller and smaller subintervals. The assumption on the absence of connections is important for the induction to be well-defined, but for IETs it also implies that the IET is minimal \cite{Keane1975}, which in particular implies that there is exactly one quasiminimal and no periodic orbits.

Rauzy-Veech induction also allows to associate to any IET $T$ with no connections an invariant\footnote{This combinatorial rotation number is a \emph{topological conjugacy invariant}, i.e if two IETs $T_0$ and $T_1$ are \emph{conjugated}, i.e.~$T_1=h\circ T_0\circ h^{-1}$ for some homeomorphism $h$  of $[0,1)$, then $T_0$ and $T_1$ have the same combinatorial rotation number.} known as \emph{combinatorial rotation number}\footnote{In the literature what we here call \emph{combinatorial rotation number} is sometimes called a 'IET rotation number' of simply 'rotation number'. It is an infinite path on a graph called Rauzy-class which describe all the possible permutations of IETs obtained by iterating the renormalization.}, which essentially consists of a sequence of permutations which describe the combinatorics of the successive induced maps produced by the renormalization scheme and thus encode the \emph{order} of orbits of $T$ inside $[0,1)$.


\
\subsubsection{Rauzy-Veech induction in the study of generalized and affine IETs}\label{sec:RV_GIETs} A natural generalization of IETs, which arises taking the Poincaré section of a general (rather than \emph{measured}) foliations are the so-called \emph{generalized interval exchange transformations} (GIETs for short). While IETs are piecewise isometries, GIETs are piecewise diffeomorphisms (see Section \S~\ref{sec:IETsintro}).
A  special  class of GIETs that is somehow easier to study, although many dynamical questions are widely open, are \emph{affine interval exchange transformations} (or AIETs in short) which are GIETs whose branches are affine maps. They
arise as the Poincaré section of transversally affine flows as well as linear flows on \emph{dilation} surfaces\footnote{Dilation surfaces are obtained by glueing sides of polygons using maps $z \to \lambda z + v$, where $\lambda \in \mathbb{R}_{>0}$ and $ v \in \mathbb{C}$ (see also \cite{ghazouani}) and generalize the notion of a \emph{translation surface}. On a dilation surface, as on a translation surface, one can define for every direction $\theta\in S^1$ a linear flow by moving with unit speed along lines in that  direction.}.

Rauzy-Veech induction has proven to be a useful tool in the study of AIETs and GIETs.
The definition of Rauzy-Veech induction can be extended, essentially verbatim, to GIETs with no connections and hence allows to define a meaninful notion of a combinatorial rotation number for those GIETs which are \emph{minimal}.
The induction has been used to prove the existence of (uniquely-ergodic) AIETs with a \emph{wandering interval}, i.e an interval whose forward images are all disjoint.
In particular, greatly extending  the ideas used to build  previous isolated examples\footnote{The first example of an AIET with a wandering interval was built by Levitt in \cite{levitt_feuilletages_1982}, but it was not uniquely ergodic. The first example of a uniquely ergodic AIET with a wandering interval was  found by Camelier and Gutierrez in \cite{CAMELIER_GUTIERREZ_1997} through the use of Rauzy-Veech induction for AIETs and later studied in detail by M. Cobo \cite{cobo_piece-wise_2002}.  Building on the example by Camelier and Gutierrez \cite{CAMELIER_GUTIERREZ_1997}, X. Bressaud, P. Hubert, and A. Maass  constructed in  \cite{bressaud_persistence_2010} several families of AIETs with wandering intervals. Recent examples in the Arnoux-Yoccoz family were also built in \cite{cobo_wandering_2018}.},
 Marmi, Moussa and Yoccoz proved  in \cite{MarmiMoussaYoccoz}, using very subtle and technical properties of Rauzy-Veech induction, that AIETs with wandering intervals exist for all combinatorial rotation numbers\footnote{More precisely, Marmi-Moussa-Yoccoz show that for almost every IET $T$ (in the sense of Lebesgue almost every choice of lenghts, assuming that the permutation is irreducible) and almost every choice of the \emph{slopes} of the affine branches, there exists an AIET with a wandering interval  semi-conjugated to $T_0$ and with the prescribed slopes.}. \\

Marmi, Moussa and Yoccoz also used Rauzy-Veech induction  from the early 2000s to study GIETs with the aim of  generalizing part of the theory of circle diffeomorphisms to higher genus. Their profound body of work {\cite{MMY_Cohomological, MMY_Linearization, MY_Holder}} has in particular allowed one to describe the local $\mathcal{C}^r$-conjugacy class\footnote{Marmi, Moussa and Yoccoz showed in particular in \cite{MMY_Linearization} that,  under suitable regularity assumptions, among GIETs which are a small (simple) perturbation of a IET, those which can be differentiably \emph{linearized} (in the sense that they are \emph{differentiably} conjugated to an IET, see, for example, \cite{Ulcigrai-Ghazouani} for definitions and results) form a submanifold of finite codimension.} of typical IETs for $r\geq 2$ (see also the work \cite{ghazouaniinventiones} by S. Ghazouani for a partial result in $r=1$ and the work \cite{Ulcigrai-Ghazouani} of Ghazouani with the last author   for a \emph{rigidity} results on GIETs achieved exploiting Rauzy-Veech induction).

\subsubsection{Beyond infinite completeness} All the results mentioned in the previous Section \S~\ref{sec:RV_GIETs} concern
AIETs and GIETs that, by the nature of the problems and constructions (in particular since they  are semi-conjugated or conjugated to a minimal IET) 
are \emph{dynamically indecomposable}, in the sense that they contain exactly \emph{one quasiminimal} (and no periodic orbits). In the language of Rauzy-Veech induction and combinatorial rotation numbers, the standing technical assumption in the works by Marmi, Moussa and Yoccoz on GIETs is that the combinatorial rotation number is infinite complete (see Definition~\ref{def:infcomplete})\footnote{Since rotation numbers of IETs with no connections are always infinite-complete, as shown  by Marmi, Moussa and Yoccoz in \cite{MarmiMoussaYoccoz} (see the lecture notes \cite{Poincaré-Yoccoz} for a proof), the invariance of combinatorial rotation numbers under semi-conjugacy implies that infinite completeness is a necessary condition  for a GIET to be conjugate or semi-conjugated to a minimal IET. It can be shown that it also sufficient, see Theorem \ref{thm:PoincareYoccoz}).}.

\smallskip
Let us stress though that given a GIET $T$, in general the rotation number is not infinite-complete and $T$ may have periodic orbits or more than one quasiminimal.  There are indications that the most common dynamical behaviour  for a GIET is \emph{Morse-Smale}, namely that one should expect to see finitely many attracting or repelling periodic orbits\footnote{Instances  of this phenomenon are easier to study for AIETs, or in the flow language, for transversally affine flows. In the latter setting,  see for example the work by Liousse, who proved in \cite{liousse} that generically (in the Baire category sense) transversely affine flows have Morse Smale dynamics. It has been conjectured by Ghazouani in the setting of AIETs (see \cite{ghazouani}) that this should also be the expected behaviour from the measure-theoretical point of view, i.e.~that for almost every dilation surface the vertical flow is Morse Smale.}
and that therefore non-infinite complete combinatorial rotation numbers are rare (like irrational rotation numbers are rare in the setting of circle diffeomorphisms); other dynamical behaviours, in particular the existence of Cantor-like attractors, are also possible (see e.g.~the work \cite{cascades} and the examples described in \S~\ref{sec:2chambers} and \S~\ref{sec:ex_quasiminimal}).

\subsection{Overview of the main results}\label{sec:resultsintro}
The main aim of our work is to show that the Rauzy-Veech renormalization, which is classically applied to study GIETs which have infinite complete rotation number and hence consist of only one dynamically independent component, can be suitably modified and extended to study GIETs with several quasiminimals or several dynamical components and provides rich information also in this setting; it allows us to \emph{locate} and \emph{separate} dynamically independent components and thus produce a \emph{dynamical decomposition}.

Far-reaching generalizations of Rauzy-Veech induction to more general maps (such as systems of linear isometries) have been also developed and exploited  by various authors mainly to identify in the parameter space of these maps those which are minimal (see for example the work by Avila, Hubert and Skripchenko \cite{Avila_2016} on the Novikov conjecture and the references therein). The aim behind these generalizations is often to identify in  a parameter space the dynamical systems which are minimal. Our goal, on the contrary, is to produce a dynamical decomposition and list and describe all  possible dynamical behaviours.

\subsubsection{Some definitions.}
Before we can state our results, we need a few definitions. A g-GIET $T$ is a map $T: I^t \subset [0,1) \to I^b \subset [0,1)$, where $I^t$ and $I^b$ are a collection of disjoint right-open subintervals of $[0,1)$, called the \textit{top intervals} and \textit{bottom intervals} respectively (as illustrated in Figure \ref{fig:g-GIET}), such that $T$ restricted to each top interval is an orientation preserving diffeomorphism onto a bottom interval (see Definition \ref{def:gGIETs} and Figure~\ref{fig:gGIETgraph}). The connected components of the complement of $I^t$ (resp. $I^b$) are called the top (resp. bottom) \textit{gaps}.

\begin{figure}[h]
\centering
\includegraphics[width=0.55\textwidth]{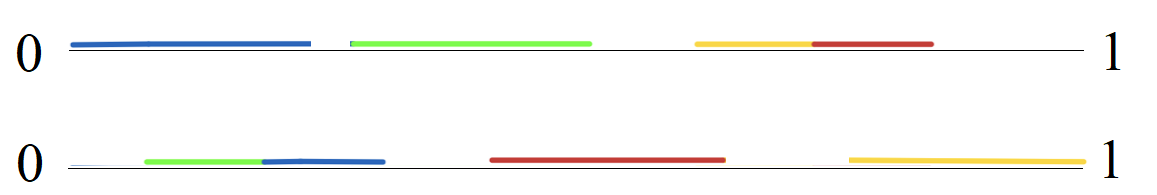}
\caption{A g-GIET on $4$ intervals.
 The top line represents the intervals $I^t$ (and the $3$ gaps), the bottom line the intervals $I^b$ (also with $3$ gaps).
Intervals with the same color are mapped to each other via an orientation preserving diffeomorphism.}\label{fig:g-GIET}
\end{figure}

Given a g-GIET $T: I^t \to I^b \subset [0,1)$, we say that $p \in I^t$ is \textit{transient} if there exists a finite orbit segment  \{$T^{-m} (p), \dots, T^{-1} (p), \, p , \,  T(p), \dots, T^{l}(p)$\},
for some $l,m \in \mathbb{N}$,  on which $T$ and $T^{-1}$ are defined (meaning that all these points belong to both the domain of $T$ and $T^{-1}$), but $T^{l+1}(p)$ lies in a top gap and $T^{-(m+1)}(p)$ lies in a bottom gap of $T$ (informally, we say that the orbit of $p$ \emph{escapes} out the domain of $T$ both in the future and in the past, for a formal definition and more examples, see \S~\ref{sec:typeorbits}).

If every orbit in $I^t$ is transient, we say that $I^t$ is a \textit{domain of transition}. If no $p \in I^t$ is transient, we say that $I^t$ is a \textit{domain of recurrence}. A \textit{periodic orbit at an endpoint} of a g-GIET $T:I^t \to I^b$ is a periodic orbit of the (continuous extension of the) restriction of $T$  to a continuity interval,  which passes through the right or left endpoint of this interval, see also Definition \ref{def:orbitatendpoint}.

\subsubsection{Decomposition into domains of recurrence}
We can now state our main result.
\begin{thm}[{\bf Decomposition theorem}]\label{thm:decomposition} Given any GIET $T:[0,1) \to [0,1)$, there exists a partition of $[0,1)$ into a finite disjoint union
\begin{align*}
    [0,1) = D \sqcup \left( \ \bigsqcup_{i=1}^k R_i \right)
\end{align*}
where   $D$ and $R_i$, for $1\leq i\leq k$, are finite unions of  right-open intervals, $D$ is a domain of transition and each $R_i$ is a domain of recurrence such that the g-GIET defined by restricting $T$ to $R_i$
is described by exactly one of the  two following cases:
\begin{enumerate}
    \item there exists a periodic orbit (possibly at an endpoint) and all recurrent orbits are periodic of the same period. \textbf{(periodic domain)};

    \vspace{1.3mm}
    \item there exists a unique quasiminimal and all recurrent orbits are non-trivially recurrent \textbf{(quasiminimal domain)}.
\end{enumerate}
\end{thm}
\noindent An example of this decomposition is given  in \S~\ref{sec:ex_decomposition}. Notice that although we start from a GIET with no gaps, the restriction of $T$ (formally defined in \S~\ref{sec:domains}) is a g-GIET (an example is shown in Figure ~\ref{fig:restriction}). The decomposition is in general not unique, see the remark in \S~\ref{sec:notunique}.

\subsubsection{Structure of domain of recurrence.}
The decomposition theorem is accompanied by a \emph{structure theorem}, which describes the structure of the $R_i$ given in Theorem 1.1. Namely, we introduce the notion of a \emph{tower representation over a g-GIET} (for which we refer the reader to the definition given in Section\S~\ref{sec:towers}) and show that each of the domains of recurrence given by Theorem 1.1 is a tower representation over a g-GIET that satisfies the following:
\begin{thm}[{\bf Structure Theorem} ]\label{thm:structure}
Each domain of recurrence $R_i$ for $i \in \{1, \dots, k\}$ given in Theorem \ref{thm:decomposition} is a tower representation of $T$ over a g-GIET $T_i: E_i^t \to E_i^b$. If $d_i$ denotes the number of intervals of $T_i$, then:
\begin{enumerate}[label=\roman*)]
\item if $R_i$ is in case 1, then $d_i = 1$ and $T_i$ is a $C^r$-diffeomorphism on its domain with at least one fixed point (possibly at an endpoint) and whose recurrent orbits are all fixed points;
\item if $R_i$ is in case 2, then $d_i >1$, the combinatorial rotation number of $T_i$ is infinite complete and $T_i$ is semi-conjugated to a minimal IET with the same combinatorial rotation number.
\end{enumerate}
\end{thm}

While decomposition and structure theorems in this spirit have already been proven for $C^r$-flows on orientable surfaces by Gutierrez, Gardiner and Levitt (see ~\S \ref{sec:decompositionthm}), our proof relies essentially on Rauzy-Veech induction at the level of the Poincaré map and consists of purely combinatorial arguments. It consists of building an explicit renormalization scheme which allows to find the decomposition given in Theorem~\ref{thm:decomposition} and \ref{thm:structure} and thus exhibits a new, powerful application of the Rauzy-Veech induction when studying GIETs which are not infinite complete.

\subsubsection{Bounds on periodic domains} As a corollary of our argument we receive the following bounds on periodic domains:

\smallskip

\begin{thm}\label{thm:boundonperiodicdomains} Let $T:[0,1) \to [0,1)$ be a GIET with $d$ intervals. Let $p$ be the number of periodic domains and $q$ be the number of quasiminmal domains of $T$ described by Theorem 1.1 and 1.2. Then
\begin{align*}
q < \left\lfloor \frac{d-p}{2}\right\rfloor
\end{align*}
\end{thm}
\noindent Note that $q$ corresponds to the number of quasiminimals of $T$. In the special case when $T$ is an AIET, every periodic domain contains either infinitely many or a unique periodic orbit (which corresponds to having either a \textit{flat} or \textit{affine} cylinder, see also \cite{ghazouani}). Therefore we get the following:

\begin{cor} Let $T:[0,1) \to [0,1)$ be an affine interval exchange map (AIET) with $d$ intervals. Then the number of periodic orbits is either infinite or at most
$d-2q$, where $q$ is the number of
quasiminimals of $T$.
\end{cor}

\subsubsection{Bounds on ergodic measures} Finally, our renormalization scheme also yields a way to {bound}  the number of (non-atomic) ergodic measures of a GIET $T$. By Theorem \ref{thm:decomposition}, we can decompose $T$ into $p$ periodic domains and $q$ quasiminimal domains. 
Notice that $T$ restricted to a periodic domain has as many ergodic measures as there are periodic orbits (i.e possibly infinite), given by the sum of Dirac measures on the points in each orbits. Let us therefore consider only \emph{non-atomic}\footnote{Let us recall that a (Borel) measure $\mu$ on $[0,1]$ is called \emph{non-atomic} if for $\mu(\{x\})=0$ for any $0\leq x\leq 1$.} invariant measures.

\begin{cor}\label{cor:ergodic_bound} Given a GIET $T:[0,1) \to [0,1)$ with $q$ quasiminimal domains, the number $\mathcal{E}_T$ of non-atomic ergodic invariant probability measures for $T$ is finite and bounded above by
$$
\mathcal{E}_T \leq \sum_{i=1}^q \left[ \frac{d_i}{2}\right]
$$
where $[x]$ denotes the fractional part of $x$ and  $d_i$, for $1\leq i\leq q$,  denote the number of continuity intervals of the IETs $T_i$ given by the Structure Theorem~\ref{thm:structure}. 
\end{cor}
\noindent Furthermore, the output of the renormalization scheme actually allows theoretically to \emph{compute}  $\mathcal{E}_T$ from the permutations of the IETs $T_i$ (see \S~\ref{sec:ergodic_bound}).   



\smallskip
\subsection{Organization of the paper and overview.} \label{sec:outlook}
In Section \S~\ref{sec:flowsintro}, which is written in the language of flows on surfaces,  we informally introduce some of the classical examples of quasiminimals in genus one and in higher genus. This section is meant to provide motivation and intuition to the reader for the objects that we describe combinatorially, as well as a historical account of some related results in the literature.

The basic object that we introduce in this paper and which is essential for the dynamical decomposition of a GIET are \textit{interval exchange transformations with gaps} (g-GIETs), whose formal definition and first properties are stated in Section \S~\ref{sec:gGIETsec}. The key tool for finding the dynamical decomposition of a GIET will be Rauzy-Veech induction, which we generalize to the setting of g-GIETs in Section \S~\ref{sec:RVsection}. The induction and related constructions as well a the notion of combinatorial rotation number are defined in \S~\ref{sec:RV g-GIETs}. When applied to a g-GIET with periodic orbits, or quasiminimals, Rauzy-Veech induction '\emph{accumulates}' towards the left-most endpoint of a periodic orbit, or a quasiminimal. Section \S~\ref{sec:renormandrecurrencesec} is therefore dedicated to identifying these \emph{renormalization limits} and, for each one of them, to find a suitable region of recurrence which contains the corresponding periodic orbit or quasiminimal and which satisfies the assumption of Theorem \ref{thm:decomposition} and \ref{thm:structure}. In Section \S~\ref{sec:decompositionproof}, we then define a renormalization scheme which successively \emph{removes} the regions of recurrence described in the previous section from the domain of any given GIET in order to \emph{separate} the dynamics and to be able to restart the renormalization scheme on the subset of the domain that remains. From this renormalization scheme we then deduce the proofs of Theorem \ref{thm:decomposition}, \ref{thm:structure} and \ref{thm:boundonperiodicdomains}.



\section{Flows with Quasiminimals}\label{sec:flowsintro} In this section, we give a brief, informal overview over the theory of flows on surfaces with a special focus on quasiminimal sets. The motivation for this chapter is on one hand to give the reader some intuition as well as a historical account of some classical examples of quasiminimals in genus one (see Denjoy- and Cherry flow in Section \S~\ref{sec:quasiminimals in genus 1}) as well as in higher genus (see the two chamber flow in Section \S~\ref{sec:DenjoyCherryHigherg}). Furthermore, we want to show how the previous work done on the decompostion of flows on surfaces throughout the 20th century, in particular the work by Levitt (\cite{levitt_feuilletages_1982}), Gardiner (\cite{gardiner}) and Gutierrez (\cite{gutierrez}), is related to our results. The reader who prefers can skip this section and move directly to the definitions in \S~\ref{sec:gGIETmotivation}.
\subsection{Recurrence and quasiminimals} We define the notion of \textit{non-trivial recurrence} and \textit{quasiminimals} for flows on surfaces.

\subsubsection{Recurrence} Given a continuous flow on a (compact, connected, orientable, topological)
surface with finitely many fixed points, the $\omega$-limit (respectively $\alpha$-limit) set of a trajectory $l$ is defined as the set of accumulation points of the positive (negative) semi-trajectory
\footnote{where the positive (negative) semi-trajectory of $p$ is defined as $\{f^t(p): t \geq 0 \}$ $(\{f^t(p): t \leq 0 \})$.}
$l^+(p)(l^-(p))$ starting at any point $p \in l$. A trajectory for which both the $\omega$-limit set and the $\alpha$-limit consist of a fixed point is called a \textit{saddle connection}. A trajectory for which either the $\omega$-limit set or the $\alpha$-limit set (but not both) consist of a fixed point is called a \textit{separatrix}. The trajectory $l$ is said to be \textit{$\omega$-recurrent} (resp.~\textit{$\alpha$-recurrent}) if it is contained in its $\omega$-limit (resp. $\alpha$-limit) set and \textit{recurrent} if it is both $\omega$- and $\alpha$- recurrent. We split recurrent trajectories into two types:

\begin{defs}A closed trajectory or a fixed point is called \textit{trivially recurrent}. All other recurrent trajectories are called \textit{non-trivially recurrent}. The topological closure of a recurrent trajectory is called a \textit{quasiminimal}.
\end{defs}

\subsubsection{First return map} Let $f^t$ be a flow on a surface $S$. A classical tool to analyze the trajectories of $f^t$, dating back to the seminal work of Poincaré in dynamics, is to consider the \textit{first return map} (also called Poincar{\'e} map) of $f^t$ on a transversal\footnote{An arc $I \subset S$ is called a \textit{transversal segment} if it does not contain any singularity of the foliation in its interior and for every point $p \in \Sigma$
it is \emph{transverse} to the leaves of the foliation.}
segment $I \subset S$. Given a transversal segment $I$, for any point $p \in I$ whose forward trajectory $l^+(p)$ intersects the transversal $I$ again, we define $T(p)$ as the \emph{first} point of intersection of $l^+(p)$.  
Notice that here $T$ is not
necessarily everywhere defined;
 we call  \textit{first return map} of $T$ to $I $ the map $T:I^t\to I$ where $I^t\subset I$ denotes the maximal domain of definition of $T$ in $I$.

\subsection{Some examples of quasiminimals on the torus}\label{sec:quasiminimals in genus 1} The simplest example of a flow with a quasiminimal set is provided by the irrational linear flow\footnote{Given a direction $\theta \in S^1$, the \textit{directional flow} $(f^t_{\theta})_{t\in \mathbb{R}}$ on the torus $\mathbb{T}^2$ is the projection to $\mathbb{T}^2=\mathbb{R}^2/\mathbb{Z}^2$ of the linear flow on  $\mathbb{R}^2$ which assigns to $(x,y)\in [0,1)^2$ the point reached in time $t$ flowing with unit speed along the line in direction $\theta$, see Figure~\ref{fig:flowstorus}, left.}
on the torus (see Figure \ref{fig:flowstorus}, left).
This 
flow has a Poincar{\'e} first return map which is an irrational rotation\footnote{Recall that an irrational rotation $R_\alpha$ with rotation number $\alpha\in \mathbb{R}$ is a map of $\mathbb{R}\backslash \mathbb{Z}$ given by $R_\alpha(x)=  x+\alpha \mod 1$. The rotation is called \emph{irrational} if $\alpha\notin \mathbb{Q}$; the map $R_\alpha$ can also be seen as an IET of $d=2$ intervals.}, and therefore is minimal and has a unique quasiminimal that is equal to the entire torus.  

At the end of last century, Poincaré [199] conjectured that there exists a $C^{\infty}$-flow on the torus without singularities and with a unique quasiminimal that is \emph{not} the entire torus, also called an \textit{exceptional set}. This conjecture has lead to the construction of two celebrated classes of examples of flows on the torus, namely \textit{Denjoy flows} and \textit{Cherry flows}.

\subsubsection{Denjoy flows}\label{sec:Denjoyflow}  Denjoy proved in \cite{Denjoy} that the Poincaré conjecture is false if the flow  is sufficiently regular\footnote{Denjoy showed that  if a circle diffeomorphism $f:S^1\to S^1$ is $C^{1+BV}$, namely $f$ is $C^{1}$ and the derivative has bounded variation, then $f$ has no wandering intervals. This implies in particular that no flow on the torus of class $\mathcal{C}^{1+BV}$ and without singularities can have a quasiminimal which is not the whole torus.} (in particular if it is of class $C^2$), but he constructed in \cite{Denjoy} a flow of class $\mathcal{C}^1$ on the torus with an exceptional set. This example, known as the Denjoy flow, has neither singularities nor closed trajectories and has a unique quasiminimal $\Omega$ whose closure is \textit{Cantor-like}, i.e.~locally homeomorphic to the direct product of a segment and the Cantor set.
Informally, the Denjoy flow is obtained \emph{blowing up} an orbit of the irrational flow on the torus, namely replacing it with an open \emph{band} (i.e.~topologically a strip) of shrinking size,  as depicted in Figure 2, and then filling the band again with orbits going in the same direction as the original orbit (for the actual  construction and definition of the flow, see \cite{Denjoy}). The quasiminimal $\Omega$ is then the closure of the complement of the band and for all $p \in T^2$, $\alpha(l(p)) = \omega(l(p)) = \Omega$.   The first return map of such a flow to any transversal is \textit{semi-conjugated} to a minimal interval exchange transformation with $d=2$ or $d=3$ intervals (more precisely, to an irrational circle diffeomorphism where the transversal is closed).

\begin{figure}
\centering
\includegraphics[width=1.0\textwidth]{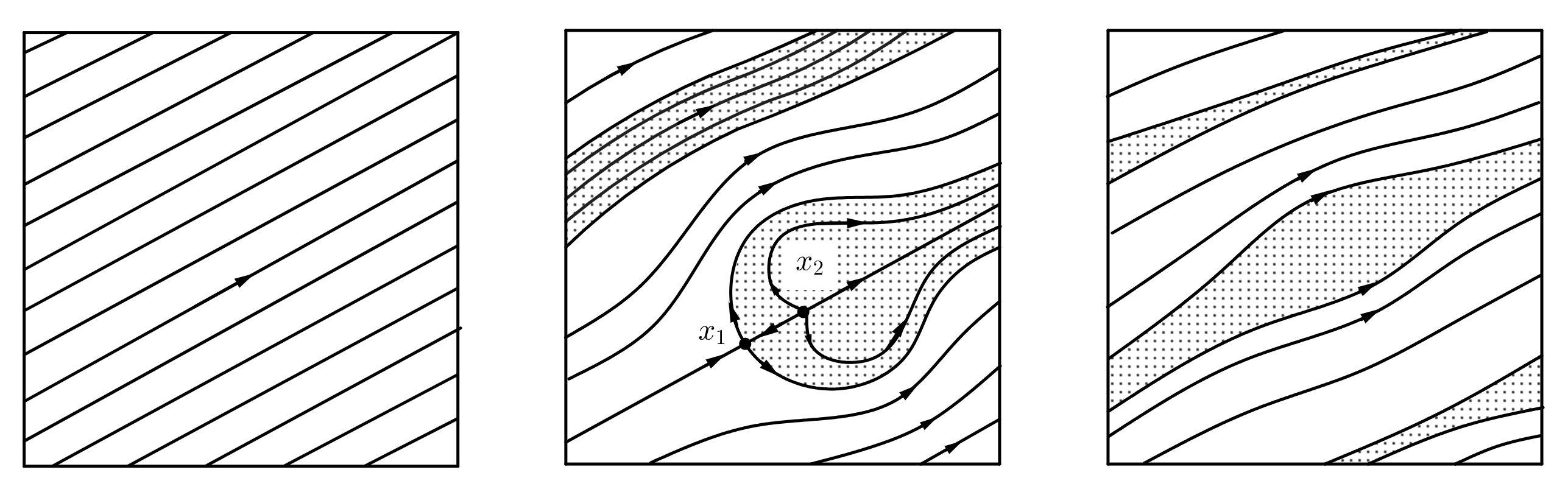}
\caption{\label{fig:flowstorus} Three examples of flows on the torus $\mathbb{T}^2=\mathbb{R}^2/\mathbb{Z}^2$: an irrational linear flow (left), a Cherry flow (center) and a Denjoy flow (right).}
\end{figure}

\subsubsection{Cherry flows}\label{sec:Cherryflow}
In 1937, T.~Cherry showed on the other hand in \cite{CherryFlow} that the Poincar{\'e} conjecture is true if one allows singularities: namely, he constructed a $C^{\infty}-$flow on the torus with a Cantor-like quasiminimal set and with two singularities, one of which is a saddle\footnote{A saddle singularity is a singularity with the same number of \textit{ingoing} and \textit{outgoing} separatrices, whereas for a source singularity, there are only outgoing separatrices} and one of which is a source.
Topologically, the Cherry flow on the torus is obtained from the irrational flow through choosing a point $p \in T^2$ and \textit{blowing up} only the forward (backward) trajectory of $p$, i.e replacing the forward (backward) trajectory with a band of shrinking size. One then places a source inside the strip and fills the band with orbits emitted from the source all going in the same direction. The original point $p$ is then replaced with a saddle singularity as shown in Figure 3. The Cantor-like quasiminimal $\Omega$ is then the closure of the complement of the band and for any $p \in T^2$, unless $p$ is a singularity or a a forward (backward) separatrix, it holds that $\omega(l(p)) = \Omega$ (resp. $\alpha(l(p)) = \Omega$).

\subsubsection{Attractors and repellors in genus one}\label{sec:attractorsrepulsors} A quasiminimal $\Omega$ of a flow $f^t$ is called an \textit{attractor} (\textit{repellor}) if there exists an invariant neighbourhood  $U \supset \Omega$ such that $\bigcap_{t \in \mathbb{R}_{>0}}f^{(-)t}(U) = \Omega$. One of the main differences between the Denjoy flow and the Cherry flow is that the quasiminimal $\Omega$ of a Cherry flow is an attractor (repellor) if one performs the forward (backward) construction, whereas the quasiminimal of a Denjoy flow is neither an attractor nor a repellor.

Levitt showed (see Theorem II.5, \cite{Levitt}) that in genus one, for $C^2$-foliations which are \textit{arational}, meaning foliations without closed trajectories and whose singularities satisfy certain restrictions\footnote{\label{fn:arational}i.e they are not allowed to be of thorn-saddle or thorn-thorn type, see also \cite{Levitt}.}, all quasiminimals whose closure is not a subsurface are either attractors or repellors.
He calls quasiminimals which are either attractors or repellors \textit{pure}, and quasiminimals which are neither attractors nor repellors \textit{non-pure}.



\subsection{Higher genus examples} On surfaces of genus $g\geq2$, in addition to generalizations of both the Denjoy and Cherry-type of construction (see \S~\ref{sec:DenjoyCherryHigherg}), one encounters new phenomena such as the coexistence of several quasiminimals (see \S~\ref{sec:2chambers}).

\subsubsection{Denjoy and Cherry-type construction in genus 2 for $C^1-$flows.\label{sec:DenjoyCherryHigherg}}
The Denjoy and Cherry-type constructions can be performed in higher genus too, namely one can build flows obtained by blowing up either a regular trajectory, or one (or more) seperatrices of a minimal flow; see \cite{Carrand} for explicit examples of $\mathcal{C}^1-$flows on a surface of genus two with an exceptional Cantor-like minimal set, obtained by perturbing the directional flow on a \emph{translation surface}\footnote{
A \emph{translation surface} is a surface with an atlas such that (outside a finite number of singular points) transition maps between two charts are translations, i.e.~maps of the form $z \to z + c$. Such surfaces can be obtained identifying pairs of parallel edges of a polygon in the plane as depicted in Figure \ref{fig:DenjoyCherryHigherg}.} of genus two.
In Figure \ref{fig:DenjoyCherryHigherg} (reproduced from \cite{Carrand}), one can see an example obtained by blowing up a regular trajectory (left) and one where all the separatrices emanating from a saddle (right) are blown up. The unique quasiminimal obtained in this way is in both cases, as for the Denjoy example in genus one, neither an attractor nor a repellor. 

\begin{figure}[h]
\centering
\includegraphics[width=0.39\textwidth]{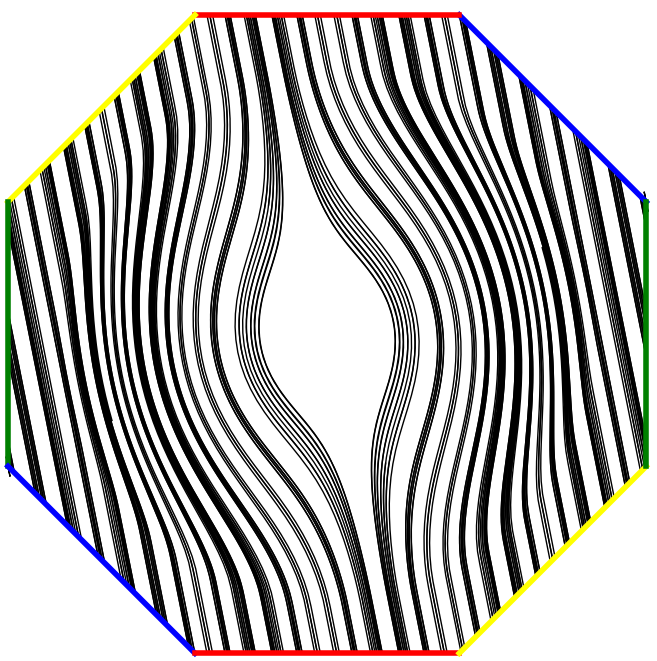}\hspace{10mm}
\includegraphics[width=0.39\textwidth]{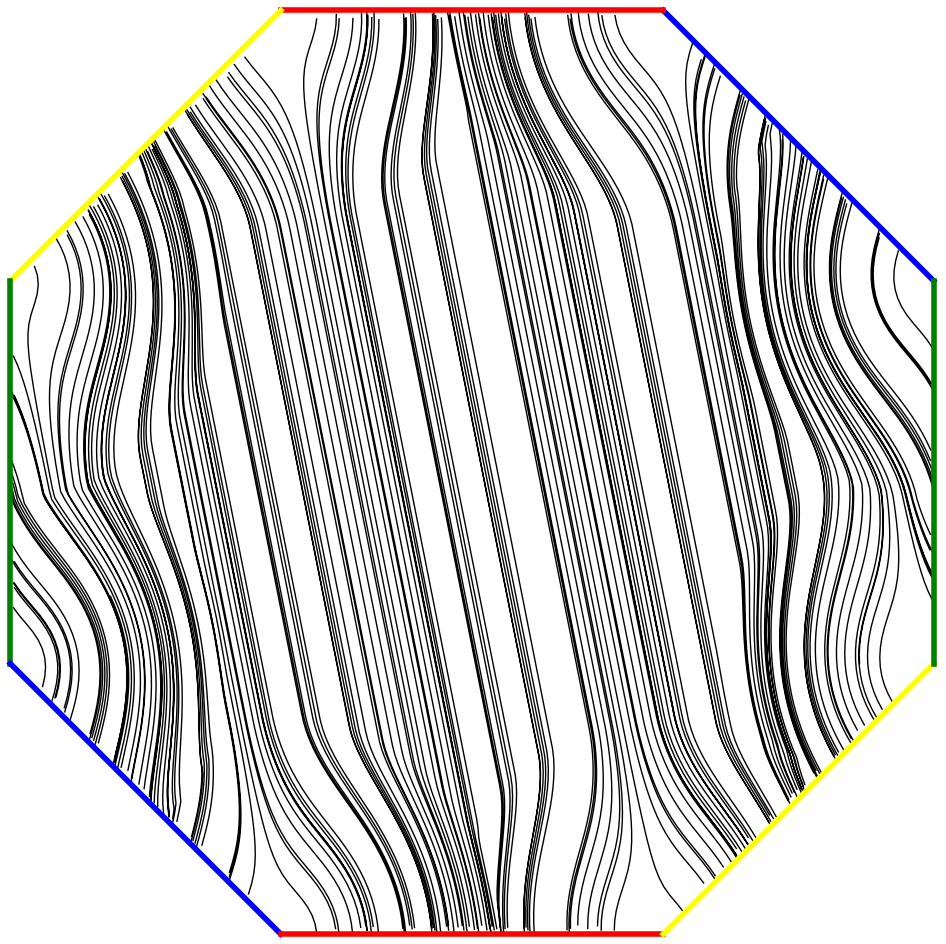}
\caption{\label{fig:DenjoyCherryHigherg} Denjoy-like and Cherry-like flows in $g=2$ (reproduced from \cite{Carrand}, courtesy of J.~Carrand).}
\end{figure}

\subsubsection{Existence of non-pure quasiminimals in higher genus}\label{sec:MMYquasiminimal} After Levitt proved that all quasiminimals in genus one for regularity $r \geq 2$ are \textit{pure} (see Section\S~\ref{sec:attractorsrepulsors}), he asked whether their exist nonpure quasiminimals in higher genus. To answer (affirmatively) this question, Camelier and Gutierrez in \cite{CAMELIER_GUTIERREZ_1997} constructed, using Rauzy-Veech induction, the first example of a uniquely ergodic AIET with a wandering interval; this map can be realized as the Poincar{\'e} map of the directional flow on a dilation surface
\footnote{A \emph{dilation} surface \label{footnote:dilation} is a translation surface for which the transition maps between two charts are affine maps, i.e.~maps of the form $z \to \lambda z + c$ where $\lambda \in \mathbb{R}_{>0}$.}
which contains a unique quasiminimal that is neither attracting nor repelling.  Similarly, the AIETs with wandering intervals produced by \cite{bressaud_persistence_2010} and, typically, by \cite{MarmiMoussaYoccoz} give examples in any genus $g\geq 2$.


\subsubsection{Cherry-type construction for smooth flows in higher genus} \label{sec:2chambers}
Finally, a higher genus example of a transversely affine foliation where two Cantor-like attractors coexist was recently studied in \cite{Dilationtori} as well as \cite{cascades} by Boulanger, Ghazouani, and the first author. Consider the genus two surface known as the \textit{two chamber surface}, obtained by identifying opposite edges of the same color of the polygon in Figure \ref{fig:twochambersurfacepol} using affine maps.
This surface may be cut along the pink dotted line (which is a simple closed curve on the surface after identifications, since all vertices of the polygon project to the same point) to obtain two \emph{chambers}; each chamber is, topologically, a genus one surface with boundary.

\begin{figure}[h]
\centering
\includegraphics[width=1.0\textwidth]{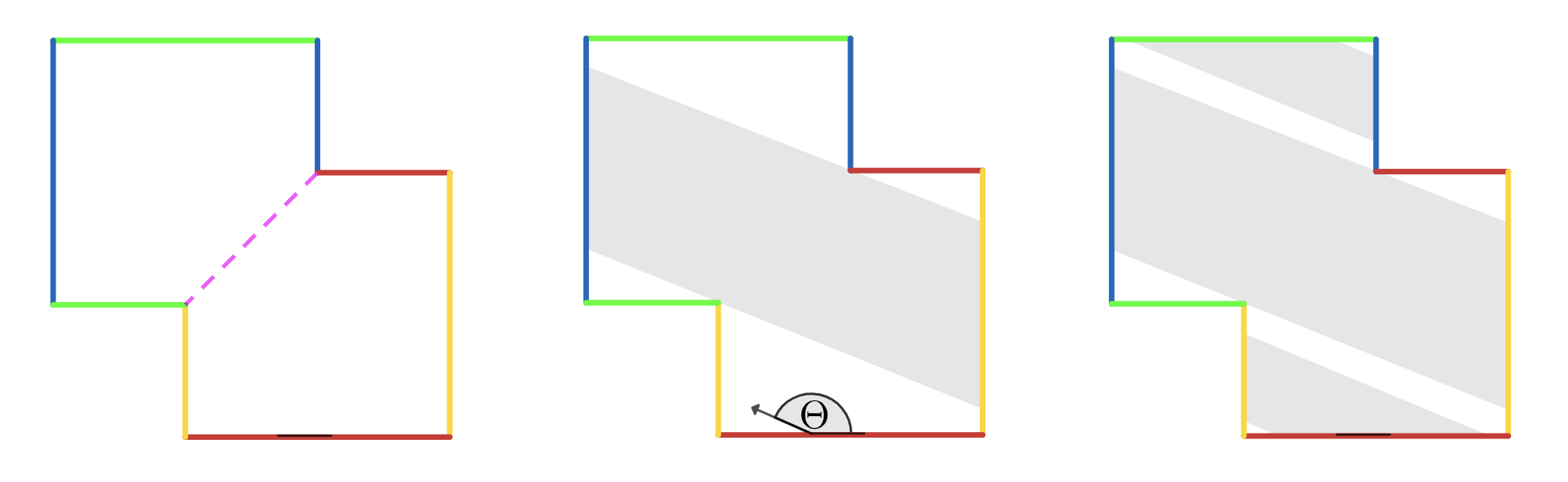}
\hspace{6mm}
\caption{The two chamber surface (polygonal representation). Consider the grey region in the center bounded by two parallel lines in direction $\theta$. For directions  $\theta \in S^1$ chosen in a carefully constructed exceptional set (uncountable, but of measure zero) the union of forward and backward iterates of this region with respect to $f^t_{\theta}$, where one iteration is depicted on the right, form an infinite band which \textit{winds} around the surface. Its complement consists of two Cantor sets $\Omega^+$ and $\Omega^-$, each contained in one of the chambers. }\label{fig:twochambersurfacepol}
\end{figure}
\begin{figure}[h]
\centering
\includegraphics[width=0.5\textwidth]{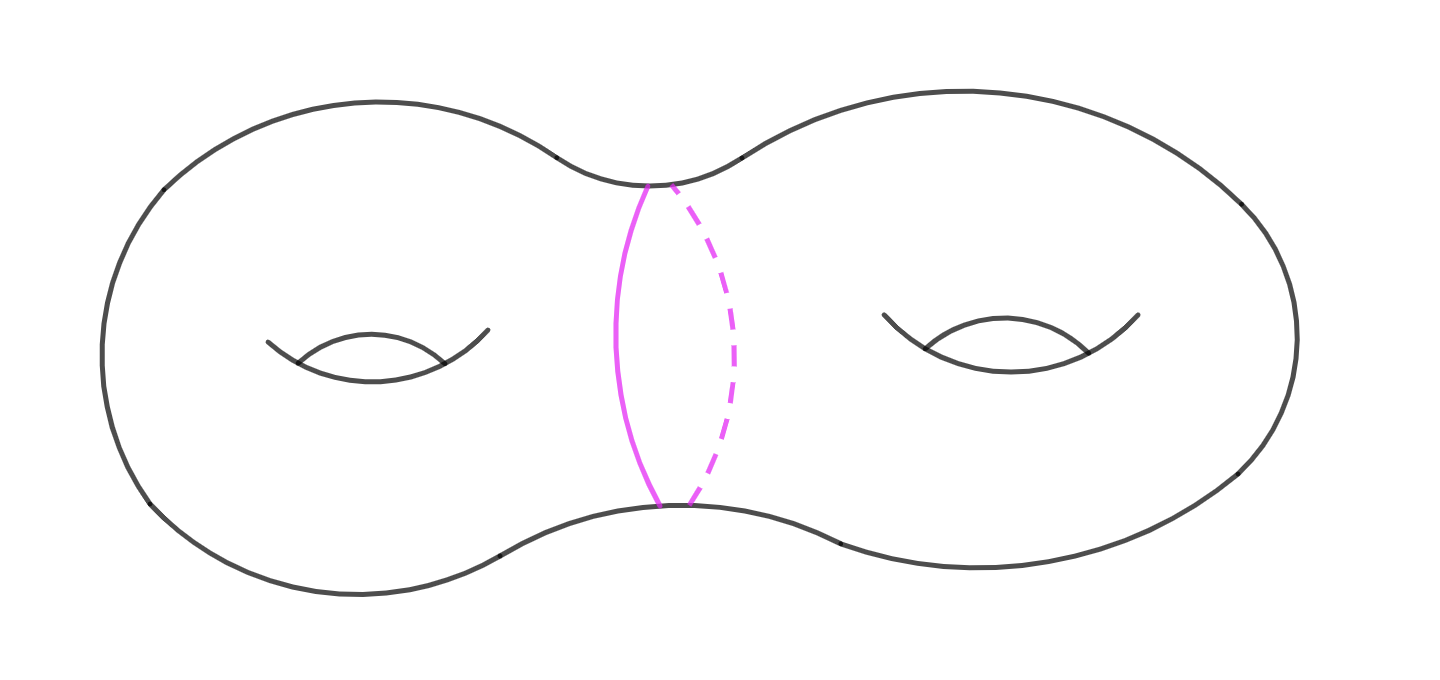}
\caption{Topological representation of the two chamber surface. The pink curve represents its dynamical decomposition (see also Theorem \ref{thm:gardiner}).}\label{fig:twochambersurfacetop}
\end{figure}


Consider the directional flow $f^t_{\theta}$, where $\theta \in S^1$. Notice that the top chamber is invariant under $f^t_{\theta}$, whereas the bottom chamber is invariant under $f^{-t}_{\theta}$ (or vice versa). It follows from the work in \cite{Dilationtori} and \cite{cascades} that for a full measure set of directions $\theta \in S^1$, the directional flow $f^t_{\theta}$ has two periodic orbits, one in each chamber, of which one is attracting and the other repelling. However, their work also shows that there exist infinitely many directions $\theta \in S^1$ (belonging to a Cantor set, of Lebesgue measure zero) for which $f^t_{\theta}$ has no saddle connections and for which there exists an invariant grey band as depicted in Figure \ref{fig:twochambersurfacetop} whose iterates are all disjoint and whose complement intersected with the top chamber forms an invariant Cantor-like quasiminimal $\Omega^+$, whereas its complement intersected with the bottom chamber forms an invariant Cantor-like quasiminimal $\Omega^-$. Since for any leaf $l$ which is not in $\Omega^+$ or $\Omega^-$ it holds that $\omega(l) = \Omega^+$ and $\alpha(l) = \Omega^-$ (see Figure \ref{fig:twochambersurfacetop}), it follows that $\Omega^+$ is an attractor, whereas $\Omega^-$ is a repellor. Note that the two quasiminimals may be \textit{separated} using the closed pink curve.

Topologically, the two chamber surface with this particular directional flow can be obtained by performing the Cherry construction on two tori, once for the forward flow to obtain a source within the wandering band and once for the backward flow to obtain a sink within the wandering band and then gluing the tori along the sink and source.

\subsection{Decomposition Theorems}\label{sec:decompositionthm} After it became known through the work of Cherry \cite{Cherry} and Maier \cite{Maier} that a flow on a surface of genus $g$ has at most $g$ quasiminimals, a natural question addressed throughout the 1980s was whether one can always find simple closed curves which separate the quasiminimals, as we have seen in the example of the two chamber surface. More precisely, the question was whether it is possible to decompose a flow into \emph{components}, obtained by restricting the flow to subsurfaces (with boundary) on which the flow either contains no quasiminimals or is \emph{(dynamically) irreducible}\footnote{A flow on a surface $S$ is called \textit{irreducible} if there is a unique quasiminimal and for every homotopically non-trivial closed curve $c$ on $S$ there exists $p \in S$ such that the forward trajectory through $p$ is non-trivially recurrent and intersects $c$.},
meaning morally that there exists a unique quasiminimal that \textit{fills} the entire space. The following result was shown by Gardiner \cite{gardiner} in 1985, after G. Levitt proved a similar result for foliations that we introduce below, and became later known as the \textit{Gardiner-Levitt decomposition}.

\begin{thm}[Gardiner, \cite{gardiner}]\label{thm:gardiner}
    Given a continuous flow on a surface with finitely many singularities, there exists a finite set $\mathcal{C}$ of homotopically non-trivial, disjoint, closed curves on $S$ such that the connected components of $S \backslash \mathcal{C}$ are either irreducible or do not contain a quasiminimal.
\end{thm}

A few years earlier, in 1982, G. Levitt studied  \textit{arational} foliations (see Footnote \ref{fn:arational}) on surfaces. He proved in in \cite{levitt_feuilletages_1982} and \cite{Levitt} that, after certain homotopy operations, the above decomposition can be obtained with the help of curves that are everywhere transversal. He does not use the term irreducible, instead, his connected components of $S \backslash \mathcal{C}$ are either a \textit{region of recurrence} or a \textit{region of transition}, a definition which we will re-use in \S~\ref{sec:domains} in the context of g-GIETs. He defines a region of recurrence to be a region with a unique quasiminimal where no leaf travels from boundary to boundary and where the limit set of any regular leaf is equal to the quasiminimal itself. A region of transition instead is a region in which every leaf travels from boundary to boundary.

  In 1986, Gutierrez established another structure theorem for any given flow on a surface $S$ (\cite{gutierrez}), this time relating the decomposition to interval exchange transformations. He proved that there exist finitely many open, connected and disjoint subsets $V_i$ of $S$ which he also calls \textit{regions of recurrence}, where each of these regions contains a circle transversal to the flow on which the first return map is either conjugated or semi-conjugated to a minimal interval exchange transformation and where there is no trajectory lying in $V_i$ and connecting two points on the boundary of $V_i$. In some sense, his result is the closest to our own decomposition for GIETs, however, the techniques we use to prove our results are very different and rely entirely on combinatorial arguments at the level of the Poincaré map.

\section{Interval exchange tranformations with gaps}\label{sec:gGIETsec}

In this chapter, we introduce \textit{interval exchange transformations with gaps} (g-GIETs). 
 We first give the formal definition of a g-GIET in \S~\ref{sec:gGIETsdef} and then explain why these are natural objects to study, see \S~\ref{sec:gGIETmotivation}.
We further classify in \S~\ref{sec:gGIETsorbits} the types of orbits such a map may have and define key notions such as \textit{the first return map}. In \S~\ref{sec:recurrence}, we further define quasiminimals for g-GIETs as well as domains of recurrence and transition. These definitions are accompanied by example decompositions of two GIETs according to Theorem \ref{thm:decomposition}. Finally, in \S~\ref{sec:towers}, we introduce the \textit{tower representation} of a g-GIET together with an example of a \textit{base} g-GIET which satisfies Theorem 1.2.

\subsection{Definition of a g-GIET and related notions}
\subsubsection{Formal definition} \label{sec:gGIETsdef}
An interval exchange transformation with gaps is defined as follows:

\begin{defs}\label{def:gGIETs}  Let $d \geq 2$ be an integer. A \textit{$C^r$-generalized interval exchange transformation with gaps (g-GIET)} of $d$ intervals is a map $T:I^t \to I^b$ where $I^t, I^b \subset [0,1)$, $|\mathcal{A}| = d$ and
\begin{enumerate}[label = \roman*)]
\item $I^t = \bigsqcup_{\alpha \in \mathcal{A}}I_{\alpha}^t$ and $I^b = \bigsqcup_{\alpha \in \mathcal{A}} I_{\alpha}^b$ are a disjoint union of $d$ right-open\footnote{We chose the convention of right-open instead of open intervals in the definition of g-GIETs since it simplifies the presentation of the proofs. However, using iii) one may easily pass from one convention to the other, and all results obtained in this paper apply equally to maps defined using open intervals.} subintervals $\{ I_{\alpha}^t\}_{\alpha \in \mathcal{A}}$, $\{ I_{\alpha}^b\}_{\alpha \in \mathcal{A}}$ called the \textit{top} and \textit{bottom} intervals;
\item for each $\alpha \in \mathcal{A}$, the map $T_{\alpha}$ obtained by restricting $T$ to $I_{\alpha}^t$ is an orientation preserving $C^r$-differomorphism onto $I_{\alpha}^b$ of class $C^r$;
\item each $T_{\alpha}$ extends on the closure of $I_{\alpha}^t$ to a $C^r$-diffeomorphism onto the closure of $I_{\alpha}^b = T(I_{\alpha}^t)$.
\end{enumerate}
\end{defs}

 The connected components of $(I^t)^c$ are called the \textit{top gaps}, whereas the connected components of $(I^b)^c$ are the \textit{bottom gaps}. For $\alpha \in \mathcal{A}$, we further call the restriction $T_{\alpha} = T|_{I_{\alpha}^t}$ of $T$ onto $I_{\alpha}^t$ a \textit{branch} of $T$. The \textit{inverse} of a g-GIET $T:I^t \to I^b$ is the map
\begin{align*}
    T^{-1}: I^b \to I^t \\
    x \to T^{-1}(x)
\end{align*}
which is again a g-GIET on $d$ intervals with with the property that $T^{-1} \circ T : I^t \to I^b$ and $T \circ T^{-1}: I^b \to I^t$.
The graph of a g-IET and of its inverse are shown in Figure~\ref{fig:gGIETgraph}. In the rest of the paper we will draw only the intervals of $I^t$ and $I^b$ using colors to indentify those that are mapped to each other (and if needed indicate the image of specific points), but we will not specify the diffeomorphisms on each domain of continuity.

\begin{figure}[h]
\centering
\includegraphics[width=0.45 \textwidth]{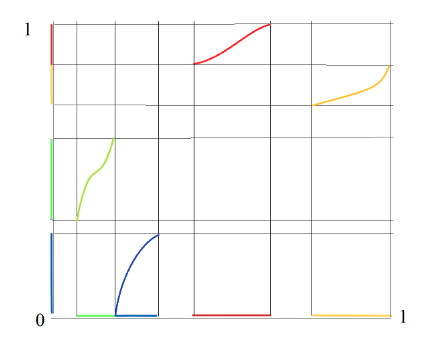}
\caption{The graph of a g-GIET with $4$ intervals. \label{fig:gGIETgraph}}
\end{figure}

\subsubsection{Right-open vs left-open g-GIET}\label{sec:rightopenvsleftopen} A map which satisfies Definition \ref{def:gGIETs}, where right-open intervals are replaced with left-open intervals, is called a \textit{left-open} g-GIET, whereas a g-GIET as in Definition \ref{def:gGIETs} may also be called a right-open g-GIET. Note that for each right-open g-GIET $T$, we obtain a unique left-open g-GIET $\overline{T}$ by first restricting $T$ to the interior of the intervals $I_\alpha^t$ and then extending it by continuity  to the right endpoint of each $I_\alpha^t$, which is possible in view of Assumption iii) in Definition \ref{def:gGIETs}.

\subsubsection{Motivation to study g-GIETs}\label{sec:gGIETmotivation}
Let $S$ be a surface with an orientable $C^r$ flow $f^t$, let $I$ be a transversal segment and consider the first return map $T$ to the segment $I$. For a point $p \in I$, there are two reasons why the forward trajectory through $p$ would not return to $I$: either the forward trajectory passes through a fixed point of the flow, or it vanishes into some other part of the surface and never crosses $I$ again. In the first case, the first return map has  a singularity at $p$ (given that the map is still defined on a punctured neighbourhood around $p$), whereas in the second case $p$ is contained in an open subinterval, also called a \textit{gap}, on which $T$ is not defined. As depicted in Figure 4, the map $T$ is hence piecewise $C^r$-differentiable and has finitely many intervals of continuity, singularities and gaps. From this map, we obtain a g-GIET as in our definition by extending the map to the left endpoint on each interval of continuity.

\begin{figure}[h]
\centering
\includegraphics[width=0.85 \textwidth]{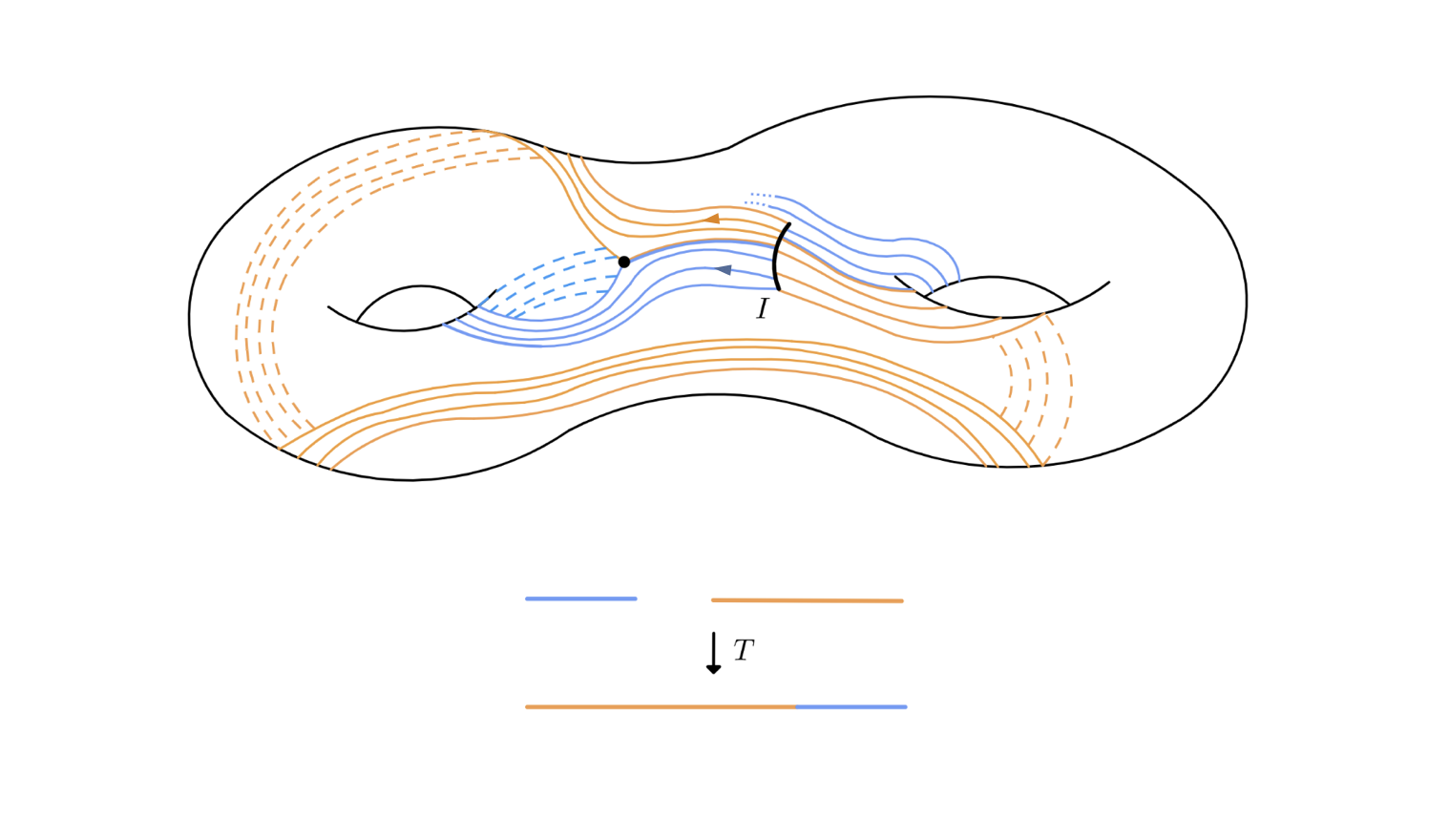}
\caption{The first return map to $I$ with two intervals of continuity and one gap arising from the blue trajectories which do not return to $I$.}
\end{figure}
\vspace{3mm}

\subsubsection{Standard and affine interval exchange transformations with gaps} Special cases of g-GIET's include standard interval exchange transformations with gaps (g-IET's) and affine interval exchange transformations with gaps (g-AIET's):

\begin{defs}
    A g-GIET $T$ is a \textit{(standard) interval exchange transformation with gaps} or an \textit{g-IET} if, for every $\alpha \in \mathcal{A}$,  $|I_{\alpha}^t|$ = $|I_{\alpha}^b|$ and the branches $T_{\alpha}$ of the map $T$ are translations, i.e of the form $ x \to x + \delta_{\alpha}$ for some $\delta_{\alpha} \in \mathbb{R}$.
\end{defs}

\begin{defs}
    A g-AIET $T$ is an \textit{affine interval exchange transformation with gaps} or an \textit{g-AIET} if the branches $T_{\alpha}$ of the map $T$, for every $\alpha \in \mathcal{A}$, are affine maps, i.e of the form $ x \to \lambda_{\alpha} x + \delta_{\alpha}$ for some $\lambda_{\alpha}, \delta_{\alpha} \in \mathbb{R}$, $\lambda_{\alpha} > 0$.
\end{defs}

When the flow is a measured flow, one can choose coordinates so that the first return map is a g-IET, while g-AIET's are first return maps of \textit{transversally affine foliations}. We remark also that g-IETs arise as the first return map of the directional flow on \textit{translation surfaces}, i.e surfaces obtained from glueing pairwise parallel edges of a polygon in the plane via translations. Similarily, g-AIETs arise as the first return map of the directional flow on \textit{dilation surfaces}, surfaces obtained from glueing pairwise parallel edges of a polygon in the plane via translations and dilations (see  Figure~\ref{fig:twochambtransverse} for an example).

\begin{defs}\label{def:GIET} A g-GIET (or g-AIET, g-IET) for which $I^t = I^b = [0,1)$ is called a \textit{generalized interval exchange transformation (GIET)} (or \textit{affine interval exchange transformation (AIET)}, or \textit{(standard) interval exchange transformation (IET))}.
\end{defs}

\subsubsection{Combinatorial data}
\label{sssect:comb}
For $T:I^t \to I^b$ a g-GIET on $d$ intervals, we think of $\alpha \in \mathcal{A}$ as the \textit{label} of the intervals $I_{\alpha}^t$ and $I_{\alpha}^b = T(I_{\alpha}^t)$ and call $\mathcal{A}$ the \textit{alphabet} consisting of labels.
To encode the order of the intervals and gaps (from left to right) at the top and bottom partition of a g-GIET, we adopt the convention\footnote{This use of two (rather than one only) permutation  became standard after its introduction in \cite{MMY_Cohomological} since it allows to keep track of \emph{labels} of intervals and is essential for certain definition (such as $\infty$-completeness, see Definition~\ref{def:infcomplete}).} of using two bijections
\begin{align*}
  \pi_t, \pi_b: \mathcal{A} \to \{1, . . . , d\},
\end{align*}
where $d$ is the number of intervals of $T$ and the maps 
$\pi_t$ and $\pi_b$, called the \textit{top (bottom) permutation}, are injective. The permutation $\pi_t$ (resp. $\pi_b$) describes the order of the intervals in the top (resp. bottom) partition, so that the order of the top (resp.~bottom)  intervals from left to right is $I_{\pi^{-1}_t(1)} , I_{\pi^{-1}_t(2)} , \dots I_{ \pi^{-1}_t(d)}$ (resp.~$I_{\pi^{-1}_b(1)}, I_{\pi^{-1}_b(2)} , \dots , I_ {\pi^{-1}_b(d)}$). 
We call the datum $\pi := (\pi_t, \pi_b)$ of this pair of permutations the \textit{combinatorial datum} of $T$, or simply the \textit{permutation} of $T$.
The combinatorial data is usually recorded writing:
\[
	\begin{pmatrix}
		\pi^{-1}_t(1) & \dots & \pi^{-1}_t(d)\\
		\pi^{-1}_b(1) & \dots & \pi^{-1}_b(d)
	\end{pmatrix}.
\]


\begin{rem}\label{rk:extended_combinatoics}
A different (more complete) choice of combinatorial data which allow to encode not only the order of the intervals, but also the \emph{location} of the gaps is possible: if $d_t$ (resp.~$d_b$) records the combined number of top (resp.~bottom) intervals \emph{and} top gaps, 
we can use bijections 
$ \pi_t: \mathcal{A} \to \{1, . . . , d_t\}$ and $
    \pi_b:  \mathcal{A} \to \{1, . . . , d_b\}$ to describe the position of the interval $I^t_{\alpha}$, where integers with no preimage in $\mathcal A$ correspond to the location of the gaps. We refer to this pair of permutation as the \textit{extended} combinatorial datum.
While we do not need this  convention for the purposes of this paper, it is the right encoding for certain questions (see e.g. \S~\ref{sec:extended_classes}).
\end{rem}

\subsection{Orbits of g-GIETs}\label{sec:gGIETsorbits}

The definition of an orbit of a g-GIETs is more complicated than in the case of GIETs because of the gaps in the domain of definition as well as in the image.

\subsubsection{Orbit of a point}Given a point $x \in I^t$, note that for $T^2(x)= T(T(x))$ to be defined we need that $T(x)\in I^t$ i.e.~that $x \in T^{-1}(I^t)=I^b$. Thus, the domain of definition of $T^2$, that we denote by $D(T^2)$, is $D(T^2)=I^t\cap I^b$. By induction, for any integer $n\geq 2$, the domain $D(T^n)$ of definition of $T^n$ is:
$$D(T^n):= \{ x\in I^t \text{ s.t } \; T(x), T^2(x), \dots, T^{n-1}(x) \in I^t \cap I^b  \}. $$
Similarly, the domain $D(T^{-n})$ of definition of backward iterates is
$$D(T^{-n}):= \{ x\in I^b \text{ s.t } \; T^{-1}(x), T^{-2}(x), \dots, T^{-(n-1)}(x) \in I^t \cap I^b  \}. $$

We may now define the forward orbit $\mathcal{O^+}(x)$ of a point $x \in I^t$ (resp. the backward orbit $\mathcal{O^-}(x)$ of a point $x \in I^b$):
\begin{align*}
    \mathcal{O^+}(x) &:= \{y \in I^b \hspace{1mm} | \hspace{1mm} \exists n \in \mathbb{N} \; \text{s.t} \;x\in D(T^n)  \text{and} \; T^{n}(x) = y \} \cup \{x\}\\
    \mathcal{O^-}(x) &:= \{y \in I^t \hspace{1mm} | \hspace{1mm} \exists n \in \mathbb{N} \; \text{s.t} \; x \in D(T^{-n}(x))
    \text{and} \; T^{-n}(x) = y \} \cup \{x\}.
\end{align*}
The \textit{full orbit} $\mathcal{O}(x)$ is then defined by
$$\mathcal{O}(x) := \mathcal{O}^+(x) \cup \mathcal{O}^-(x).$$

If $\mathcal{O}^+(x)$ is finite, we call $T^{n}(x)$ the \textit{forward endpoint} of $\mathcal{O}(x)$, where $n$ is the largest integer for which $T^{n}(x) \in \mathcal{O}(x)$, similarily, if $\mathcal{O}^-(x)$ is finite, we call $T^{-m}(x)$ the \textit{backward endpoint} of $\mathcal{O}(x)$, where $m$ is the largest integer for which $T^{-m}(x) \in \mathcal{O}(x)$. Note that $\mathcal{O}(x)$ is contained in $I^t \cap I^b$ except for its endpoints which are always contained in a gap.

\subsubsection{Pictorial representation of orbits}
In the illustrations throughout the paper, we adopt the following  convention to represent the map and its orbits. We draw the intervals in $I^t$ on the top and the intervals in $I^b$ at the bottom. The map then takes a point from a top interval to a point in a bottom interval (of the same color) which is then in turn again identified back to a point in the top interval, as indicated by the vertical arrows pointing upwards (see e.g. ~Figure~\ref{fig:transient}). Arrows which appear to \textit{enter} or \textit{leave} the g-GIET hence correspond to orbits whose backward (forward) endpoints lie in a gap.



\subsubsection{Different types of orbits}\label{sec:typeorbits} We say that $\mathcal{O}(x)$ is \textit{regular} if $\mathcal{O}(x) \subset I^t\cap I^b$. If $\mathcal{O}(x)$ is regular and finite, we call $\mathcal{O}(x)$ a \textit{periodic} orbit. Among orbits which are \textit{not} regular, we say that
$\mathcal{O}(x)$ is \textit{transient} if $\mathcal{O}(x)$ is finite and its forward and backward endpoint lie in a top and in a bottom gap respectively (see Figure~\ref{fig:transient}).


\begin{figure}[h]
\centering
\includegraphics[width=0.6\textwidth]{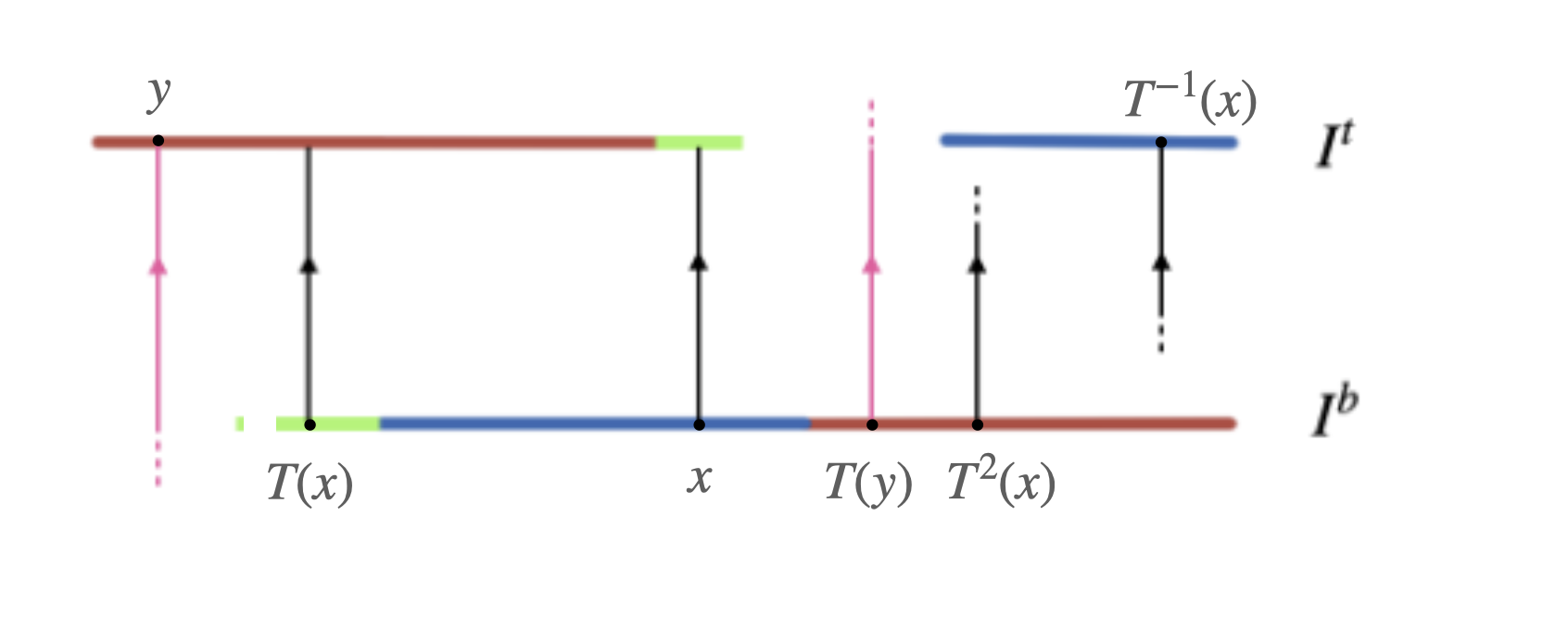}
\caption{\label{fig:transient} A g-GIET on three intervals. The pink orbit is a transient orbit, whereas the black orbit is not.}
\end{figure}


\smallskip
We further define the notion of an \textit{orbit at an endpoint} as follows:

\begin{defs}\label{def:orbitatendpoint} Let $T:I^t \to I^b$ be a g-GIET and let $\overline{T}$ be the corresponding left-open g-GIET as defined in Section \S~\ref{sec:rightopenvsleftopen}. Then a periodic orbit for $T$ which contains at least one left endpoint of an interval is called an \textit{periodic orbit at a left endpoint of $T$}. A periodic orbit for $\overline{T}$ which contains at least one right endpoint of an interval is called a \textit{periodic orbit at a right endpoint of $T$}.
\end{defs}
\begin{rem}
Note that for a g-GIET $T$, a periodic orbit at a left endpoint is a regular orbit. However, on a surface whose first return map (extended by continuity to the left endpoints) yields this g-GIET, the trajectory corresponding to this periodic orbit is in fact a concatenation of saddle connections.

\end{rem}




\subsubsection{First return map of a g-GIET}\label{sec:Poincaremap} Let $T:I^t \to I^b$ be a g-GIET and consider a subinterval $J \subset I^t$. Let $J^t :=  \{x \in J \hspace{1mm}|\hspace{1mm} \mathcal{O}^+(x) \backslash \{x\} \cap J \neq \emptyset \}$.
Thus, $J^t$ contains all points which \emph{return} to $J$ under $T$ and is the domain of the first return map.
For $x \in J^t$, define
\begin{align*}
    n(x) := \text{min}\{i \in \mathbb{N} \hspace{1mm} | \hspace{1mm} T^i(x) \in J \}.
\end{align*}
Then the \textit{first return map} $T_J$ of $T$ to $J$ is then defined as
\begin{align*}
    T_J: J^t &\to T_J(J^t)
    \\
    x &\to T^{n(x)}(x)
\end{align*}
We set $J^b := T_J(J^t) \subset J$. One can verify that $T_J: J^t \to J^b$ is again a g-GIET with top and bottom intervals
$$J^t = \bigsqcup_{\beta \in \mathcal{B}}J_{\beta}^t \quad \text{and} \quad
J^b = \bigsqcup_{\beta \in \mathcal{B}}J_{\beta}^b,
$$
where $T_J(J_{\beta}^t) = J_{\beta}^b$ and all points in $J_{\beta}^t$ share the same return time denoted by $n_{\beta}$ for all $\beta \in \mathcal{B}$.

\begin{rem}\label{rk:number_intervals_Poincare}
One can check that if $T$ is a g-GIET of $d$ intervals,
the first return map $T_j$ on the subinterval $J$ has at most $d+2$ intervals, i.e.~ $|\mathcal{B}| \leq  |\mathcal{A}|+2$.
Note though that any of the cases $|\mathcal{A}| = |\mathcal{B}|$, $|\mathcal{A}| > |\mathcal{B}|$ or $|\mathcal{A}| < |\mathcal{B}|$ may occur.
\end{rem}
\subsection{Recurrent trajectories and quasiminimals}\label{sec:recurrence} We repeat here, for the context of g-GIETs, definitions analogous to those given already for flows on surfaces in Section \S~\ref{sec:flowsintro} such as limit sets and quasiminimals and define domains of transition and recurrence.

\subsubsection{Limit sets}\label{sec:limitset}
Let $T: I^t \to I^b$ be a g-GIET and let $x \in I^t$. If $\mathcal{O}^+(x)$ is infinite then the $\omega$-limit set $\omega(x)$ is the set of accumulation points of $\mathcal{O^+}(x)$ in $I^t$. If $\mathcal{O}^-(x)$ is infinite then the $\alpha$-limit set $\alpha(x)$ is the set of accumulation points of $\mathcal{O^-}(x)$ in $I^t$.

A point $x \in I^t$ is said to be \textit{$\omega$-recurrent} if $\mathcal{O}^+(x)$ is regular and $x \in \omega(x)$, it is said to be \textit{$\alpha$-recurrent} if $\mathcal{O}^-(x)$ is regular and $x \in \alpha(x)$ and \textit{recurrent} if $x$ is both $\omega$- and $\alpha$-recurrent. If $x$ is recurrent, then any point of the $\mathcal{O}(x)$ is also recurrent since limit sets are invariant under $T$.
We can split recurrent orbits into two types: a \textit{trivially recurrent orbit} is a periodic orbit. All other recurrent orbits are called \textit{non-trivially recurrent}.

\subsubsection{Quasiminimals and wandering intervals}
As for flows on surfaces, a quasiminimal is defined as the closure of a non-trivially recurrent orbit:
\begin{defs} Let $T:I^t \to I^b$ be a g-GIET. The topological closure (in $I^t$) of a (trivially or non-trivially) recurrent orbit is a \textit{recurrent orbit closure} of $T$. The topological closure (in $I^t$) of a non-trivially recurrent orbit is a \textit{quasiminimal} of $T$.
\end{defs}



\noindent We can also define the notion of a wandering interval:

\begin{defs}\label{def:wanderinginterval} Let $T:I^t \to I^b$ be a g-GIET. We say that $J \subset I^t$ is a \emph{forward} (\emph{backward}) \emph{wandering interval} if all iterates $\{ T^n(J) \,| \, n \in \mathbb{N} \}$ ($\{ T^{-n}(J)\,  | \, n \in \mathbb{N} \}$) are well-defined and disjoint. We say that $J$ is a \emph{wandering interval} if it is both a forward and a backward wandering interval.
\end{defs}

 \begin{rem} When the g-GIET is obtained from the Poincar{\'e} map to a transversal segment of a flow on a surface, the orbit of a wandering interval is given by the intersection of the transversal segment with the \emph{bands} which are \emph{blown up} from a trajectory or a separatrix of a minimal flow in the Denjoy or Cherry-like examples described in sections \S~\ref{sec:Denjoyflow}, \S~\ref{sec:Cherryflow}, and  \S~\ref{sec:DenjoyCherryHigherg}.
 \end{rem}


\subsubsection{Examples of trivial orbit closures} \label{sec:ex_periodic}
To give examples of trivially recurrent orbits (i.e.~periodic orbits) one can use the following simple remark: if there exists $\alpha\in \mathcal{A}$ such that $I^b_\alpha\subset I^t_\alpha$ (or viceversa, $I^b_\alpha\subset I^t_\alpha$), then, since $T(I^t_\alpha)\subset I^t_\alpha$ (or $T^{-1}(I^b_\alpha)\subset I^b_\alpha$), there exists, by a classical fixed point theorem, a fixed point $x_\alpha\in I^t_\alpha\cap I^t_\alpha$  for $T$, i.e.~$T(x_\alpha)=x_\alpha$ and hence a trivially recurrent orbit.

\begin{figure}[h]
\centering
\includegraphics[width=0.4\textwidth]{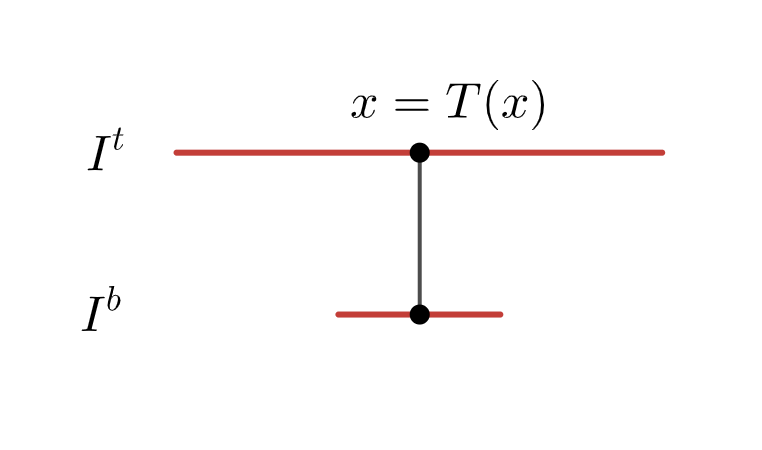}
\caption{A g-GIET on one interval with a fixed point. \label{fig:trivrec} }
\end{figure}


\subsubsection{Examples of non-trivial recurrent orbit closures} \label{sec:ex_quasiminimal}
An example of a g-GIET with a non-trivially recurrent orbit closure arises as the first return map for the directional flow $f^t_{\theta}$ on the transversal segment $I$ of the two chamber surface (see Figure \ref{fig:twochambtransverse}). It follows from the results in \cite{ghazouani}, \cite{Dilationtori} as explained in \S~\ref{sec:2chambers} that for a measure zero set of directions $\theta \in S^1$ the first return map of the flow $f^t_{\theta}$ to the segment $I$ is a g-GIET on two intervals with one bottom gap, as depicted on the right of Figure \ref{fig:twochambtransverse}, with the following dynamical behavior: this g-GIET has a (forward) wandering interval (depicted here in grey) all of whose images are disjoint, and whose complement forms an attracting quasiminimal $\Omega^+$. Furthermore, there exists no transient orbit (see definition in \S~\ref{sec:typeorbits}), as each orbit which \textit{enters} through the bottom gap is \textit{trapped} in the grey wandering interval and will never leave the domain $I^t$ again.

\begin{figure}[h]
\centering
\includegraphics[width=0.8\textwidth]{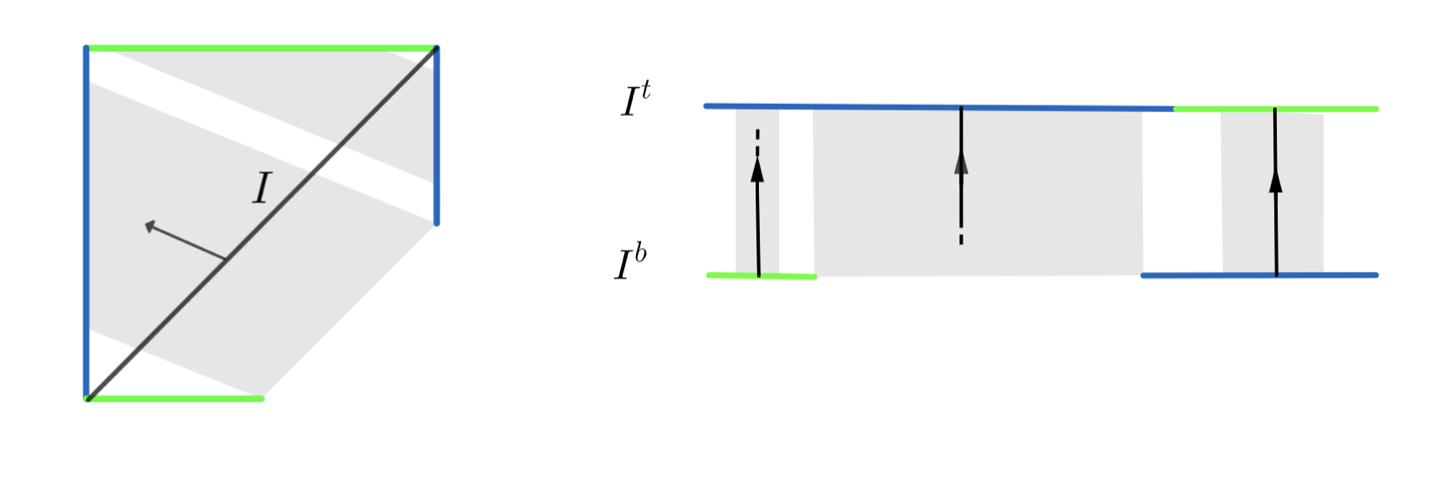}
\caption{\label{fig:twochambtransverse} The g-GIET obtained as the first return map of the directional flow $f^t_{\theta}$ to $I$.}
\end{figure}

From the example in Figure \ref{fig:twochambtransverse} one may build the AIET depicted in Figure \ref{fig:discoev}, which we also refer to as the \textit{Disco map}, as it corresponds to the Disco surface introduced in \cite{cascades}. As shown in \cite{cascades}, for a measure zero set of parameters of the pink interval, this AIET contains two Cantor-like quasiminimals, one of which is attracting and the other one repelling.

\begin{figure}
    \centering
    \includegraphics[width=0.6\linewidth]{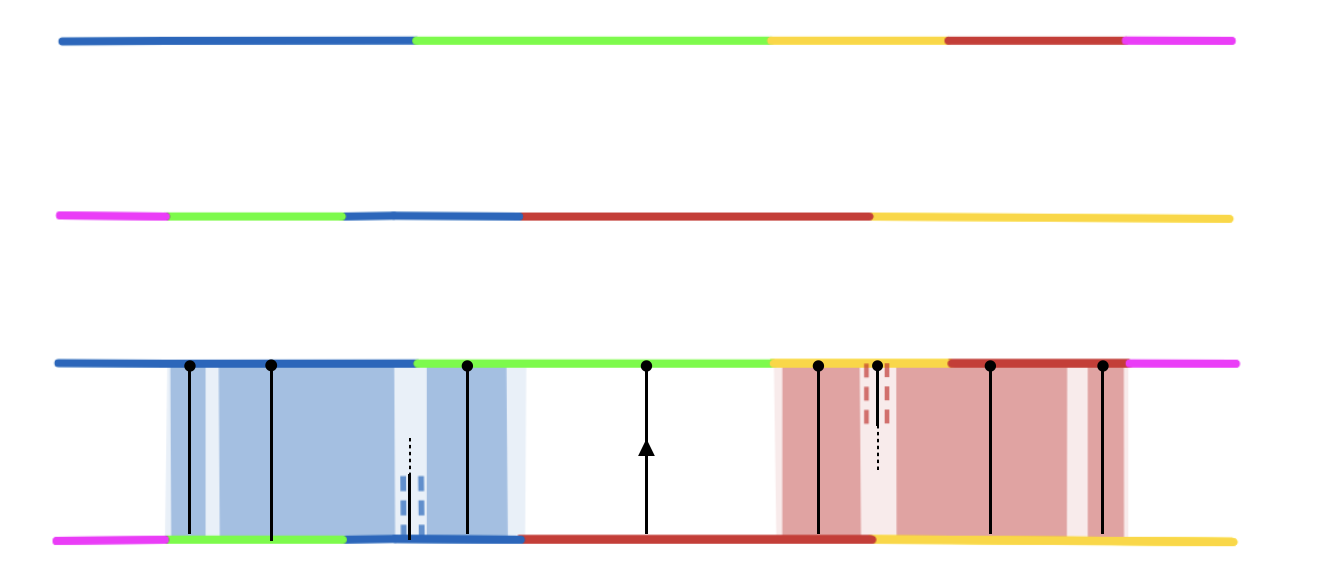}
    \caption{\label{fig:discoev} An AIET $T$ on five intervals, where the dilation factors are all either $2$, $\frac{1}{2}$ or $1$ and which is symmetric under rotation by $\pi$. For a measure zero set of parameters of the pink interval the dynamical behavior is the following: there exist two subsurfaces, one of which is invariant under $T$ (in light blue), the other which is invariant under $T^{-1}$ (in light red). The complement of the wandering intervals for $T$ (in dark blue) and the wandering intervals for $T^{-1}$ (in dark red) form two invariant Cantor-like quasiminimals $\Omega^+$ and $\Omega^-$. Any trajectory which is not contained in one of the quasiminimals is in the future attracted to $\Omega^+$ and in the past to $\Omega^-$.
    \label{fig:enter-label}}
\end{figure}

The examples in Figure \ref{fig:trivrec}, \ref{fig:twochambtransverse} and \ref{fig:discoev} share a special property: they have no transient orbits, i.e no orbits which both \textit{enter} through a gap and \textit{exit} through a gap. We call the domains of g-GIETs with this type of behaviour \textit{domain of recurrence}.

\subsubsection{Restriction of a g-GIET and domains of recurrence}\label{sec:domains}
Before we may give the formal definition of a \textit{domain of recurrence}, we must define the notion of a \textit{restriction of a g-GIET}. Let $T:I^t \to I^b$ be a g-GIET and let $J \subset I^t$ be a finite union of right-open intervals. 
Then the \textit{restriction} $T|_{J}$ of $T$ to $L$ is defined as the map:
\begin{align*}
    T|_{J}: J &\to T(J)\\
    x &\to T(x).
\end{align*}
Note that $T|_{J}: J  \to T(J)$ is again a g-GIET whose top and bottom intervals are the connected components of $J$ and $T(J)$.

\begin{figure}[h]
\includegraphics[width=0.6\textwidth]{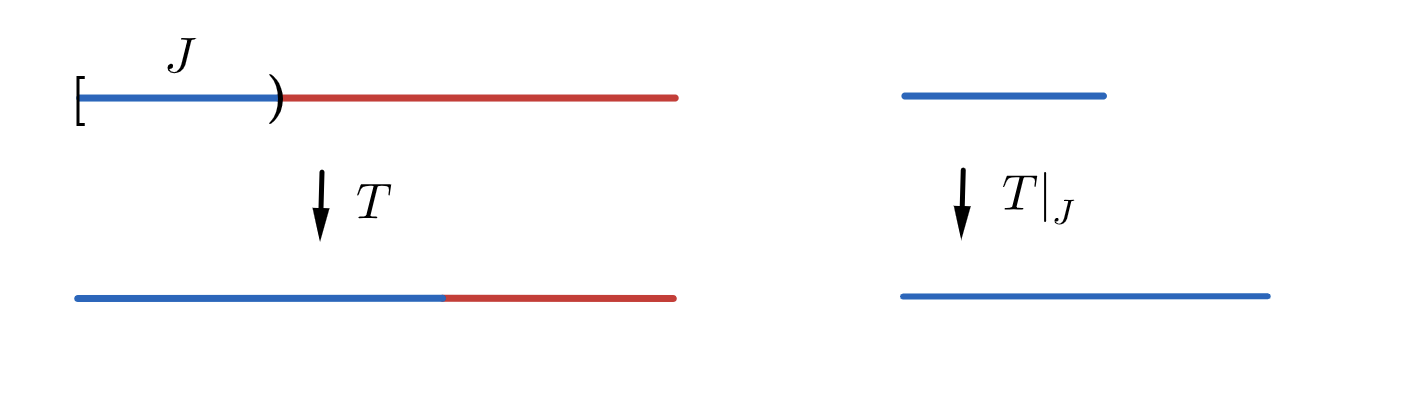}
\caption{\label{fig:restriction}
An example of the restriction of a GIET $T$ to $J$, where $J$ is taken to be the blue subinterval.}
\end{figure}

We are now able to give the formal definition of a domain of recurrence:

\begin{defs} We say that $J \subset I^t$ is a \textit{domain of transition} if for all $x \in J$, the orbit $\mathcal{O}_{T_J}(x)$ of $x$ under the map $T_J$ is transient. If no such $x \in J$ exists, then we call $J$ a domain of recurrence.
\end{defs}

\subsubsection{Example decompositions} \label{sec:ex_decomposition}
We give two examples of decompositions of an interval exchange transformation into domains of transition and domains of recurrence according to Theorem \ref{thm:decomposition}. \\

\noindent {\it Example 1:} Consider the IET in Figure \ref{fig:periodicdecomp} below, where the lengths of the green and blue interval are chosen to be irrational. If $R_1$ denotes the yellow interval and $R_2$ denotes the union of the red, blue, green and pink intervals, then every orbit in $R_1$ is periodic, whereas every orbit in $R_2$ is dense in $R_2$. Hence, the decomposition given on the right hand side of Figure \ref{fig:periodicdecomp} satisfies the assumptions of Theorem \ref{thm:decomposition}, where $R_1$ is a periodic domain and $R_2$ is a quasiminimal domain. We will return to this example later and show that region $R_1$ can be written as a \emph{tower representation} (see \S~\ref{sec:tower_example}).

\begin{figure}[h]
\centering
\includegraphics[width=0.9\textwidth]{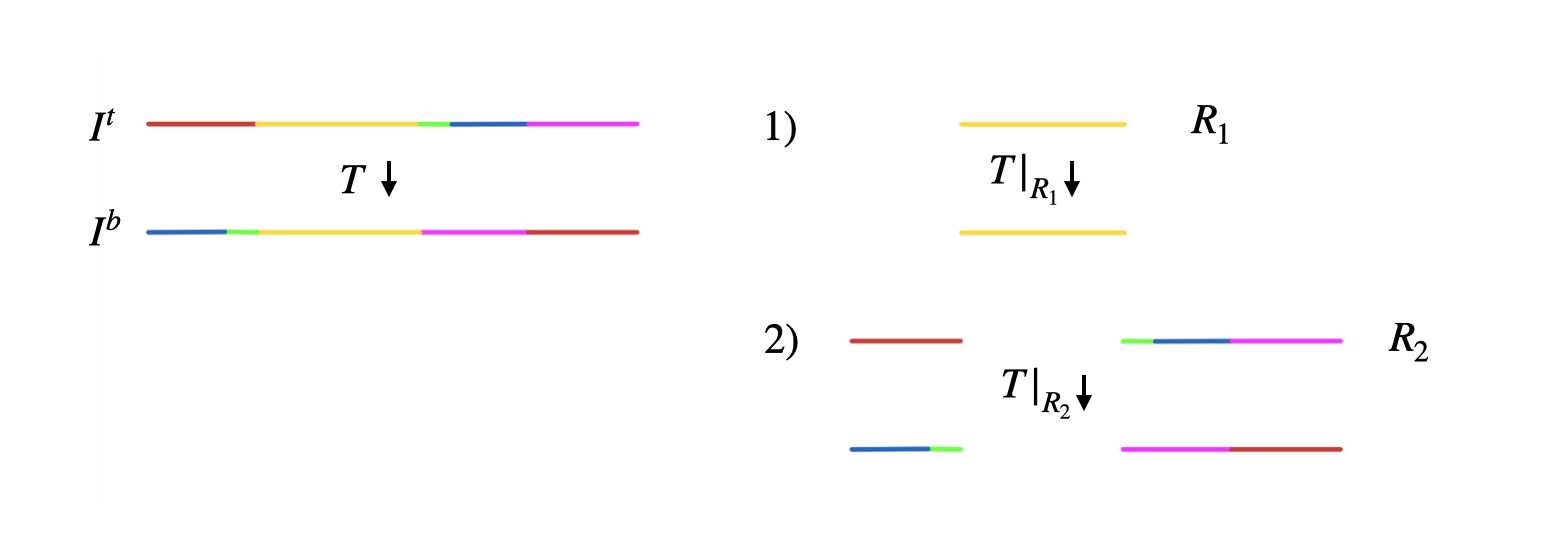}
\caption{\label{fig:periodicdecomp}Decomposition of the IET on the left according to Theorem \ref{thm:decomposition}. The lengths of the green and blue interval are chosen to be irrational.}
\end{figure}

\smallskip
\noindent {\it Example 2:} The dynamical decomposition of the AIET from Figure \ref{fig:discoev} is depicted in Figure \ref{fig:Cantordecomp} below. In this case $R_1$ and $R_2$ are both quasiminimal domains, since, as explained in Figure \ref{fig:discoev}, they contain a unique Cantor-like quasiminimal and they are also a domain of recurrence. The region $R_3$ is a domain of transition, since every orbit which \textit{enters} also \textit{leaves} again.

\begin{figure}[h]
\centering
\includegraphics[width=1.02\textwidth]{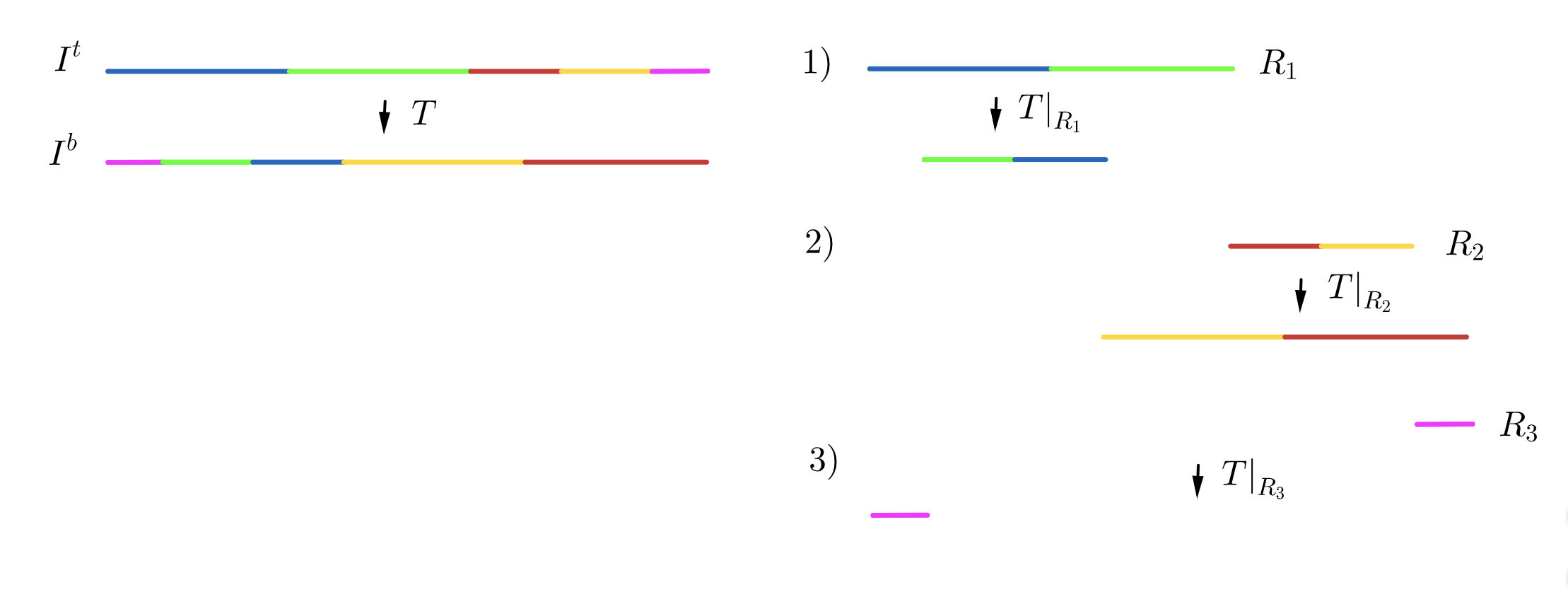}
\caption{\label{fig:Cantordecomp} Decomposition of the Disco map according to Theorem \ref{thm:decomposition}.}
\end{figure}


\smallskip

\subsection{Tower representation of a GIET with gaps}
\label{sec:towers} Now that we have introduced all the definitions necessary for the statement of Theorem \ref{thm:decomposition}, we turn towards the statement of Theorem \ref{thm:structure}. For this, we introduce the notion of a tower representation of a g-GIET. This tower representation will allow us to describe the dynamical behavior of the maps $T|_{R_i}$ in terms of a simpler map defined on the \textit{base} intervals of $R_i$ (see also the example in \S~\ref{sec:tower_example}).



\subsubsection{Definition of tower representation} The notions of \emph{tower} and \textit{tower representation} of a g-GIET (in \S~\ref{sec:towers})
generalize the definition of a  \emph{Rohlin tower} and of a \emph{skyscraper} in ergodic theory to the context of g-GIETs.


\smallskip
Let $T: I^t \to I^b$ be a g-GIET and let $\tilde{T}: \tilde{I^t} \to \tilde{I^b}$ be a g-GIET with
 \begin{align*} \tilde{I^t} &:= \bigcup_{\beta \in \mathcal{B}} \tilde{I^t_{\beta}}
  \subset I^t,
 \\
\tilde{I^b} &:= \bigcup_{\beta \in \mathcal{B}} \tilde{I^b_{\beta}} \subset I^b. \end{align*}

\begin{defs}[Tower representation]\label{def:towerrep} We say that  $T$ has a  \textit{tower representation} over $\tilde{T}$ if there exists positive integers $(n_\beta)_{\beta\in\mathcal{B}}$ such that, for any $\beta\in \mathcal{B}$, the intervals
$$
\tilde{I^t_{\beta}},  \  T (\tilde I^t_{\beta}),\  \dots \ , T^{n_\beta-1}(\tilde
 I^t_{\beta})$$
are pairwise disjoint,  all contained in $I^t$ and
\begin{equation}\label{eq:acceleration}
\tilde{T}(x) = T^{n_{\beta}}(x), \quad  \text{for\  all}\  x \in \tilde{I^t_{\beta}}, \qquad \forall \beta \in \mathcal{B}.
\end{equation}
\end{defs}
\noindent
In this case we define the  \textit{tower} for $T$ over 
the  base interval $\tilde{I^t_{\beta}}$  as the finite union of intervals
\begin{align}\label{eq:tower_def}
    \text{Tow}(\tilde{I^t_{\beta}}) := \bigcup_{i=1}^{n_\beta-1} T^{i}(\tilde{I^t_{\beta}}).
\end{align}
 The intervals in the union \eqref{eq:tower_def} are called \emph{floors} of the tower.
Note that by definition $ \text{Tow}(\tilde{I^t_{\beta}}) \subset I^t$. The \emph{tower representation} can then be represented as a  disjoint union of towers:
\begin{align*}
    \text{Rep}(\tilde{I}^t) := \bigcup_{\beta \in \mathcal{B}} \text{Tow}(\tilde{I^t_{\beta}}),
\end{align*}
(see Figure~\ref{fig:tower_rep}). We say that the intervals $\{\tilde{I^t_{\beta}}\}_{\beta \in \mathcal{B}}$ are  the \textit{base intervals} of the representation and the $(n_\beta)_{\beta\in \mathcal{B}}$ are the \emph{heights} of the towers. Inside the tower $\text{Tow}(\tilde{I^t_{\beta}}) $, $T$ maps each floor $T^{i}(\tilde{I^t_{\beta}})$ with $0\leq i<n_\beta$ to the floor $T^{i+1}(\tilde{I^t_{\beta}})$ represented just \emph{above}, while the  \emph{top floor} $T^{n_{\beta}-1} (\tilde{I}^t_\beta)$ is mapped to $\tilde{I}^b_\beta$.


\begin{figure}[h]
\centering
\includegraphics[width=0.8\textwidth]{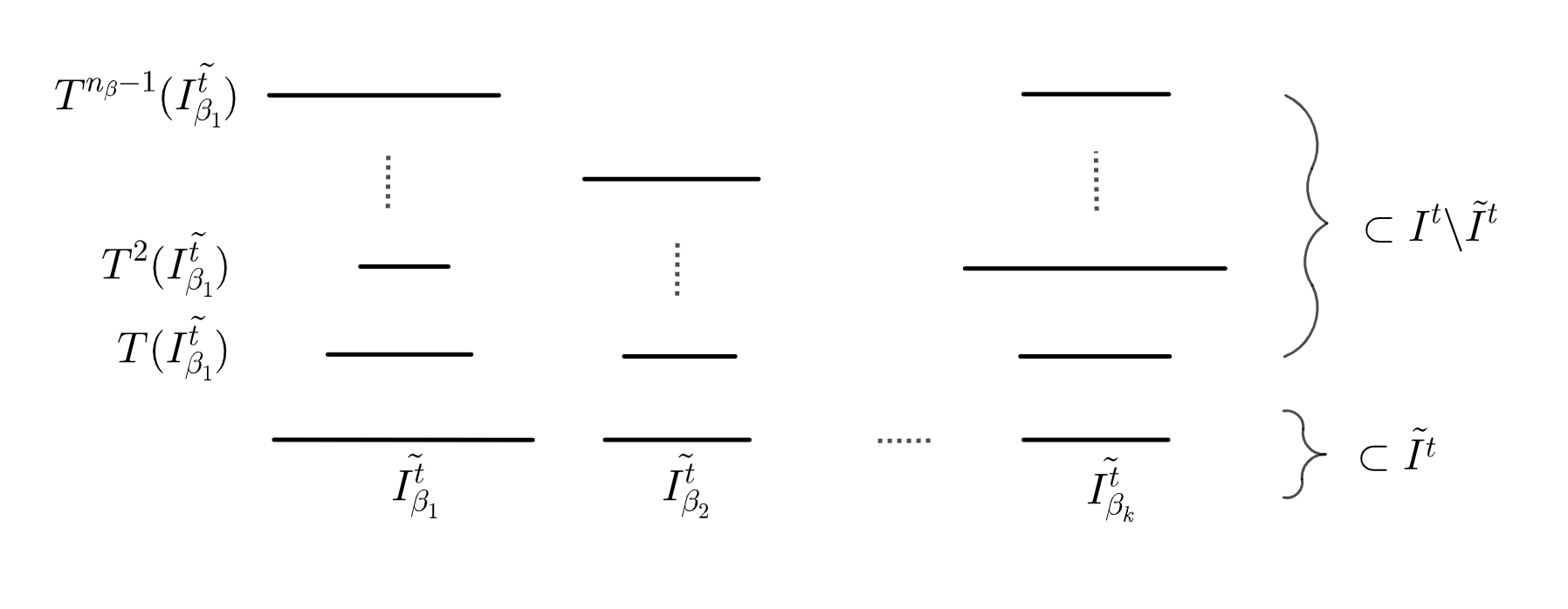}
\caption{\label{fig:tower_rep} Schematic picture of a tower representation.}
\end{figure}

\begin{rem}
The classical example of tower representations arise when $\tilde{T}$ is the first return of $T$ on a subinterval $J$, in which case the base intervals $\tilde{I}^t_\beta$ of $\text{Tow}(\tilde{I^t_{\beta}})$ are contained in $J$ (in this case $\text{Tow}(\tilde{I^t_{\beta}})$ is also called a \emph{skyscraper}). In general, though, note that $\tilde{T}$ \emph{is not} the first return map of $T$ on the base, since in general the intervals of $\tilde{I}_{\beta}^b$ do not have to be contained in $\tilde{I}_{\beta}^t$.

\end{rem}



\begin{defs}[Acceleration]
If $T$ has a tower representation over $\tilde{T}$, we say that the g-GIET  $\tilde{T}$ is an \emph{acceleration} map of $T$.
\end{defs}
\noindent Conversely, if $\tilde{T}$ is an acceleration of $T$, then by \eqref{eq:acceleration}   each branch of $\tilde{T}$  is a given by a \emph{iterate} of $T$.  We  conclude this section by recording a simple observation on the dynamical structue of towers which will be useful in some proofs.
\begin{rem}\label{rk:tower}
If $R=$Tow$(\tilde{I^t_{\beta}})$ is a tower representation of $T$, any $x\in R$ which belongs to a floor of the tower has a backward iterate in the base, i.e.~there exists $m\in \mathbb{N}$ and $\tilde{x} \in \tilde{I^t_{\beta}}$ such that  $T^m (\tilde{x}) = x$. This follows simply since backward iterates of $x$ will move down in the tower until they reach the base floor.
\end{rem}



\subsubsection{Example of a tower representation}\label{sec:tower_example}
In the example decomposition shown in Figure \ref{fig:periodicdecomp} (see Section \S~\ref{sec:ex_decomposition}), the  g-IET $T|_{R_1}$ obtained restricting $T$ to ${R_1}$ has a tower representation over the base g-IET $T_1:E_1^t \to E_1^b$ as depicted in Figure \ref{fig:towerper}. Here, $E_1^t$ is the top red interval of $T$ from Figure \ref{fig:periodicdecomp}. Note that the base g-GIET $T_1$ is now simply an irrational rotation, and that the dynamics of the base g-GIET $T_1$ determines the dynamics of $T$ restricted ro $R_1$. We will show that this IET satsifies assumption (i) of Theorem \ref{thm:decomposition} (i.e that its \textit{combinatorial rotation number} is \textit{infinite complete}).

\begin{figure}[h]
\includegraphics[width=0.5\textwidth]{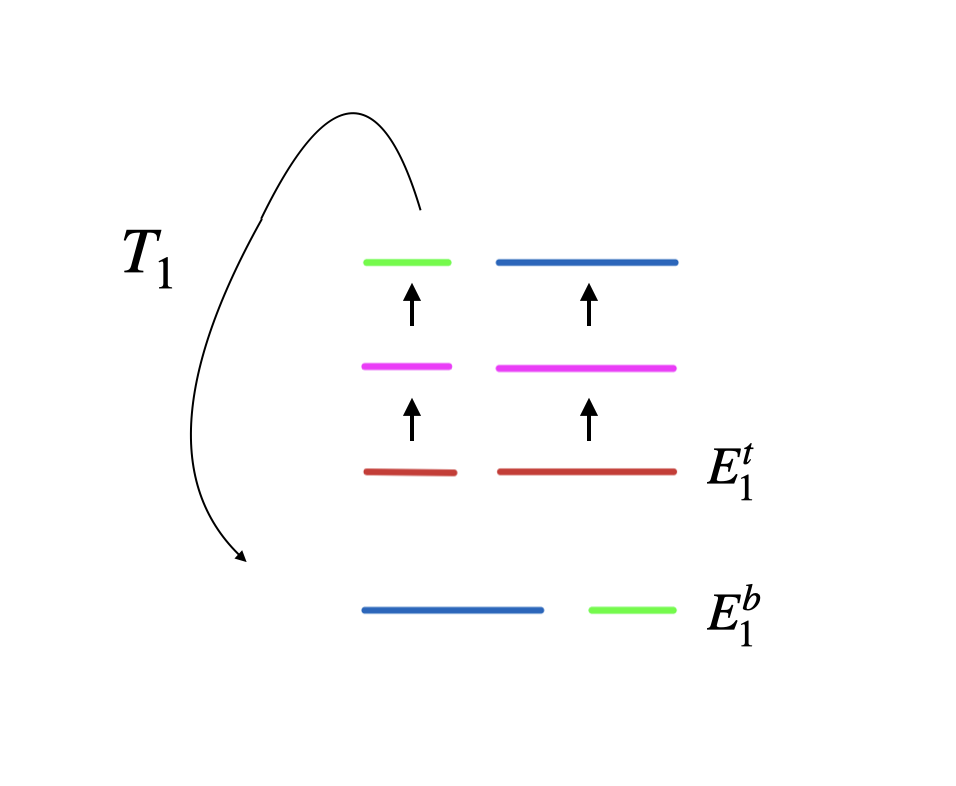}
\caption{\label{fig:towerper}
A pictorial representation of the tower representation of $T|_{R_1}$ over $T_1$.}
\end{figure}

\subsection{Maier's Theorem for g-GIETs}\label{sec:maierthm} We conclude this section by proving a result on the structure of quasiminimal sets, for which we will use the first return map and the tower representation of a g-GIET. The main statement of Proposition \ref{prop:maierthm} and the related Corollary \ref{cor:maierthm} below are that any non-trivially recurrent orbit contained in a quasiminimal is dense therein.
This result
will be of importance in Section \S~\ref{sec:renormandrecurrence}, when \textit{separate} quasiminimals using Rauzy-Veech induction.

The analogous result for flows on surfaces was proved by Maier (see \cite{Maier}) in the 1930s for flows on surfaces. His proof uses a fact from Poincaré-Bendixson theory, namely that there exist no non-trivially recurrent trajectories on a simply connected domain with possibly a finite number of disks removed\footnote{Maier exploits the fact that a finite set of curves on a surface of genus $g$ bounds such a domain in order to \textit{trap} a non-trivially recurrent trajectory in such a region and lead to a contradiction. For more details, see \cite{zhuzoma}.}.
We give an independent proof in the setting of g-GIETs,  which does not use Poincaré-Bendixson theory. Instead, it crucially relies on the fact that the first return map of a g-GIET to an interval is again a g-GIET on \textit{finitely} many intervals.

\begin{prop}\label{prop:maierthm}
Let $T:I^t \to I^b$ be a g-GIET with two non-trivially recurrent orbits $l_1, l_2 \subset I^t$ and let $\Omega_1$, $\Omega_2$ be the quasiminimals obtained by taking the closure of $l_1, l_2$ in $I^t$. Then if $l_2 \subset \Omega_1$, also $l_1 \subset \Omega_2$.
\end{prop}

\begin{proof} Since $\Omega_1 = \omega(l_1) \cup \alpha(l_1)$ (see Section \S~\ref{sec:limitset}) it suffices to show
\begin{align*}
    l_2 \subset \omega(l_1) \implies l_1 \subset \omega(l_2),
\end{align*}
which together with the corresponding statement for the $\alpha$-limit set, whose proof is analogous to the one we give below, implies the result.

Assume by contradiction that $l_2 \subset \omega(l_1)$ but $l_1 \not \subset \omega (l_2)$. It follows that there exists $p_1 \in l_1$ and a right-open neighbourhood $U_1 \subset I^t$ which contains $p_1$ such that
    \begin{align}\label{eq:1}
    l_2 \cap U_1 = \emptyset.
    \end{align}
    Since by assumption $l_2 \subset \omega(l_1)$, there exists also a point $p_2 \in l_2$ and a right-open neighbourhood $U_2 \subset I^t$ which contains $p_2$ as well as a sequence
    \begin{align}\label{eq:2}
        (x_n)_{n \in \mathbb{N}} \in U_2 \cap l_1 \hspace{2mm} \text{s.t.} \lim_{n \to \infty}x_n = p_2,
    \end{align}
 where w.l.o.g we may assume that $x_n$ accumulates to $p_2$ from the right.

\begin{figure}[h]
\centering
\includegraphics[width=0.8\textwidth]{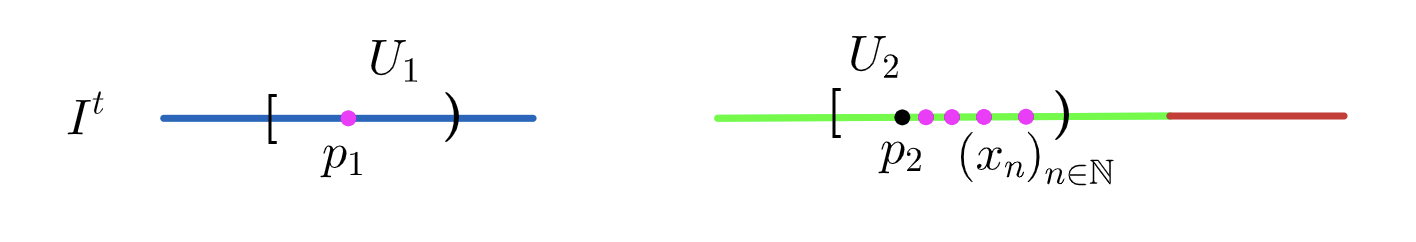}
\caption{\label{fig:accumulate} The pink points belong to $l_1$, the black point to $l_2$.}
\end{figure}

Let $\tilde{T}: \tilde{I}^t \to \tilde{I}^b$ be the first return map of $T$ to $U_1$, which is again a g-GIET on a finite number of right-open intervals (see Remark~\ref{rk:number_intervals_Poincare}). Consider the tower representation Rep($\tilde{I}^t$) of $T$ over $\tilde{T}$. Then Rep($\tilde{I}^t$) consists of a finite union of right-open intervals. Furthermore, we claim that
    \begin{align}\label{eq:3}
    l_1 \subset \text{Rep}(\tilde{I}^t).
    \end{align}
    Indeed, consider any $x \in l_1 \subset I^t$. To show (\ref{eq:3}), it is enough to show that $x \in \text{Rep}(\tilde{I}^t)$, since $\text{Rep}(\tilde{I}^t)$ is invariant for $T$. Since $l_1$ is non-trivially recurrent, there must exist minimal $m, n \in \mathbb{N}$ such that $T^{-m}(x) \in U_1$ and $T^{n}(x) \in U_1$. Thus, by construction $n+m$ is the first return of $T^{-m}(x)\in U_1\subset I^t$ to $U_1$, so
    the tower of $\text{Rep}(\tilde{I}^t)$ over the interval in $\tilde{I}^t$ which contains $T^{-m}(x)$ also contains
    $x$  and (\ref{eq:3}) follows.

    Note that since $\text{Rep}(\tilde{I}^t)$ is a \emph{finite} union of left-closed intervals, it is left-closed, i.e.~
     any 
    right-sided limit points of points in $\text{Rep}(\tilde{I}^t)$
    must also be contained in $\text{Rep}(\tilde{I}^t)$. Hence, by (\ref{eq:2}), it follows that $p_2 \in \text{Rep}(\tilde{I}^t)$, so in particular it belongs to a tower. Since every point in a tower has an iterate which belongs to the base, this implies that there exists $l \in \mathbb{N}$ such that $T^l(p_2) \in U_1$. In particular, since $p_2 \in l_2$
    it follows $l_2 \cap U_1 \neq \emptyset$, which is a contradiction to (\ref{eq:1}). The proposition follows.
\end{proof}

\begin{cor}\label{cor:maierthm}
    Let $T:I^t \to I^b$ be a g-GIET which contains a quasiminimal $\Omega_1$. Then if $l_2 \subset \Omega_1$ is a non-trivially recurrent orbit, and $\Omega_2$ is the quasiminimal obtained from the closure of $l_2$ in $I^t$, then $\Omega_1 = \Omega_2$.
\end{cor}

\begin{proof} Let $l_1$ be a non-trivially recurrent orbit such that its closure is equal to $\Omega_1$. By assumption, $l_2 \subset \Omega_1$ and in particular $\Omega_2 \subset \Omega_1$. But by Proposition \ref{prop:maierthm}, it follows that  $l_1 \subset \Omega_2$ and in particular $l_1 \subset \Omega_2 \subset \Omega_1$. In particular, since the closure of $l_1$ is equal to $\Omega_1$, and since $\Omega_2$ is closed, it follows that $\Omega_1 = \Omega_2$.
\end{proof}
\begin{rem}\label{rem:maierthm} In particular, Corollary \ref{cor:maierthm} implies that if two quasiminimals contain the same non-trivially recurrent orbit, they must be equal.
\end{rem}

\section{Rauzy-Veech induction for GIETs with gaps} \label{sec:RVsection}




We introduce in this section the Rauzy-Veech induction for g-GIETs, which will be our key tool for determining a decomposition of a g-GIET into the dynamically independent components described in Theorem 1.1. The Rauzy-Veech induction for g-GIETs is a generalization of the standard Rauzy-Veech induction for (G)IETs. The main difference is that while in the standard case only intervals \textit{play} against each other, in our case intervals may also \textit{play} against gaps. This may result in the number of intervals decreasing by one. Another difference is that for the purpose of using the induction later on to find a dynamical decomposition, we allow the induction to be well-defined even if the two rightmost top and bottom intervals have the same left endpoints.




\subsection{Elementary step of Rauzy-Veech induction}\label{sec:RV g-GIETs} Let $T:I^t \to I^b$ be a g-GIET on $d$ intervals. We assume throughout that $d>0$.  
For the remainder of this chapter, we use the following notation:
\begin{defs} Let $\alpha_t$ and $\alpha_b$ be the letters in  ${ \mathcal{A}}$ defined by
$$\pi^{-1}_t(d)=\alpha_t, \qquad \pi^{-1}_b(d)=\alpha_b,$$
so that $I^t_{\alpha_t}$ is  the rightmost top interval of $I^t$, and $I^t_{\alpha_b}$ is the rightmost bottom interval of $I^b$. Name $l^t$ and  $r^t$ the end points of $I^t_{\alpha_t}$ and by  $l^b$ and $r^b$ the endpoins of $I^t_{\alpha_b}$, so that
\begin{align}
I_{\alpha_t}= I_{\pi^{-1}_t(d)} = [l^t, r^t), \qquad I_{\alpha_b}= I_{\pi^{-1}_b(d)} = [l^b, r^b).
\end{align}
\end{defs}
We distinguish between four cases;  in each of them, we will define a suitable interval $[0,\lambda^{(1)})$ on which to induce $T$.
\begin{defs}
The g-GIET  $T^{(1)}$ obtained by one step of Rauzy-Veech induction from $T$ is the g-GIET obtained as first return map of $T$ to the interval $[0,\lambda^{(1)})$ defined according to the cases below. Explicitely, $T^{(1)}:I^{t,1} \to I^{b,1}$ is a g-GIET with alphabet $\mathcal{A}^{(1)}$ and intervals $I^{t,1} = \bigcup_{\alpha \in \mathcal{A}^{(1)}} I^{t,1}_{\alpha}$ and $I^{b,1} = \bigcup_{\alpha \in \mathcal{A}^{(1)}} I^{b,1}_{\alpha}$ as described below.

\begin{itemize}
\item \textbf{Case 1} \textit{(interval vs interval)}: assume that $r^t = r^b$, i.e the two rightmost intervals have the same right endpoint. We then define $\lambda^{(1)} := \textit{max\hspace{1mm}}(l^t, l^b)$.  

\smallskip
There are two subcases:
\begin{itemize}

\vspace{1mm}
\item \textbf{Case 1a} \textit{(an interval wins)} if $l^t \neq l^b$, i.e the two rightmost intervals do not have the same length. If $l^t < l^b$ we say that the \textit{top interval wins} and \textit{bottom interval looses}, if $l^t > l^b$ we say that the \textit{bottom interval wins} and \textit{top interval looses}. Note that in this case, $T^{(1)}$ is again a g-GIET on $d$ intervals.

Explicitely, $\mathcal{A}^{(1)} = \mathcal{A}$ and if $l^t < l^b$ (top interval wins), then
\begin{align*}
    I_{\alpha}^{t,1} = I_{\alpha}^{t} \cap [0,l^b), \hspace{2mm} \alpha \in \mathcal{A},\\
   T^{(1)}_{\alpha} = T|_{I_{\alpha}^{t,1}}, \hspace{2mm} \alpha \neq \alpha_b,\\
   T^{(1)}_{\alpha} = T^2|_{I_{\alpha}^{t,1}}, \hspace{2mm} \alpha = \alpha_b.
\end{align*}
If $l^t > l^b$ (bottom interval wins), then
\begin{align*}
    I_{\alpha}^{t,1} = T^{-1}(I_{\alpha}^{b} \cap [0,l^t)), \hspace{2mm} \alpha \neq \alpha_t,\\
    I_{\alpha}^{t,1} = T^{-1}(I_{\alpha}^{b} \cap [l^t,1)), \hspace{2mm} \alpha = \alpha_t,\\
    T^{(1)}_{\alpha} = T|_{I_{\alpha}^{t,1}}, \hspace{2mm} \alpha \neq \alpha_t,\\
   T^{(1)}_{\alpha} = T^2|_{I_{\alpha}^{t,1}}, \hspace{2mm} \alpha = \alpha_t.
\end{align*}


\begin{figure}[h]
\centering
\includegraphics[width=0.9\textwidth]{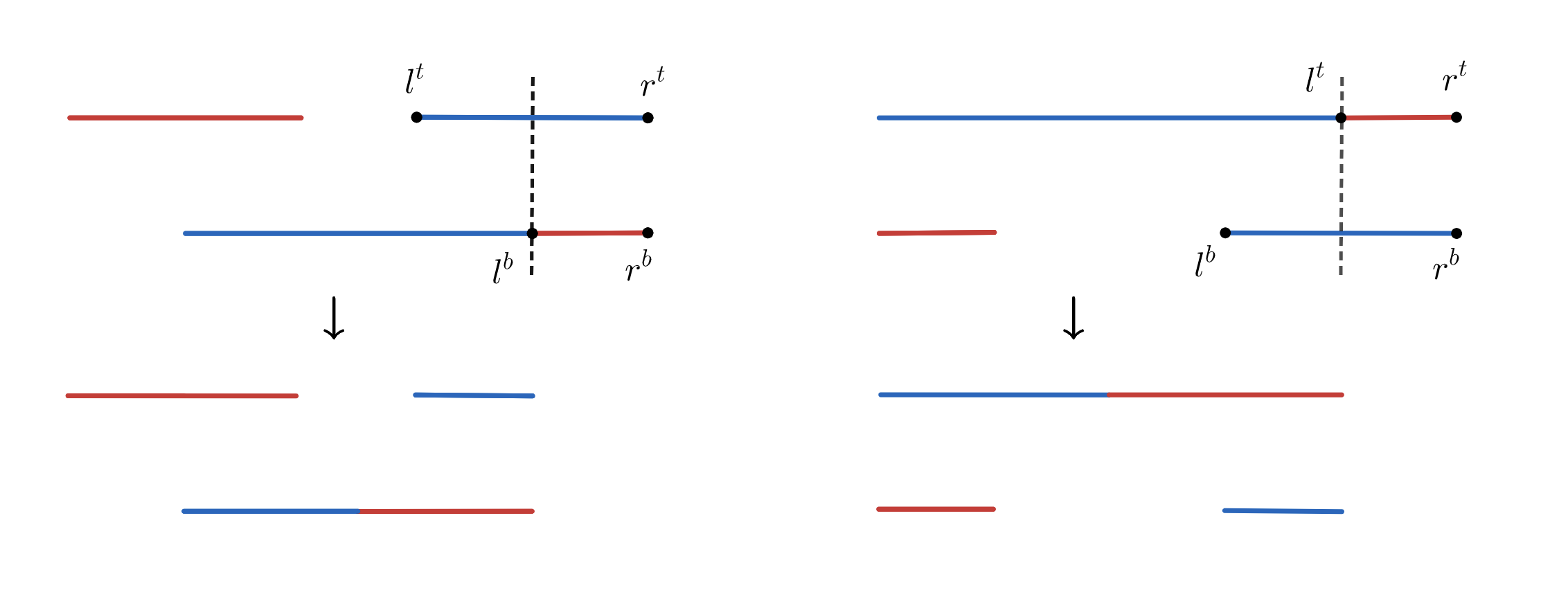}
\caption{Case 1a: on the left the top interval wins, on the right, the bottom interval wins.}
\end{figure}

\item \textbf{Case 1b} \textit{(neither wins)}: assume that $l^t = l^b$, i.e the two rightmost intervals have the same length. In this case, we say that \textit{neither wins}. Note that $T^{(1)}$ is a g-GIET on $d-1$ intervals.

Explicitely, if $l^t = l^b$ and $\mathcal{A}^{(1)} := \mathcal{A} \backslash \{\alpha_t\}$ then
\begin{align*}
    I_{\alpha}^{t,1} = I_{\alpha}^{t}, \hspace{2mm} \alpha \in \mathcal{A}^{(1)},\\
   T^{(1)}_{\alpha} = T|_{I_{\alpha}^{t,1}}, \hspace{2mm} \alpha \in \mathcal{A}^{(1)} \backslash \{\alpha_b \}.\\
   T^{(1)}_{\alpha} = T^2|_{I_{\alpha}^{t,1}}, \hspace{2mm} \alpha = \alpha_b.\\
\end{align*}
\end{itemize}

\begin{figure}[h]
\centering
\includegraphics[width=0.45\textwidth]{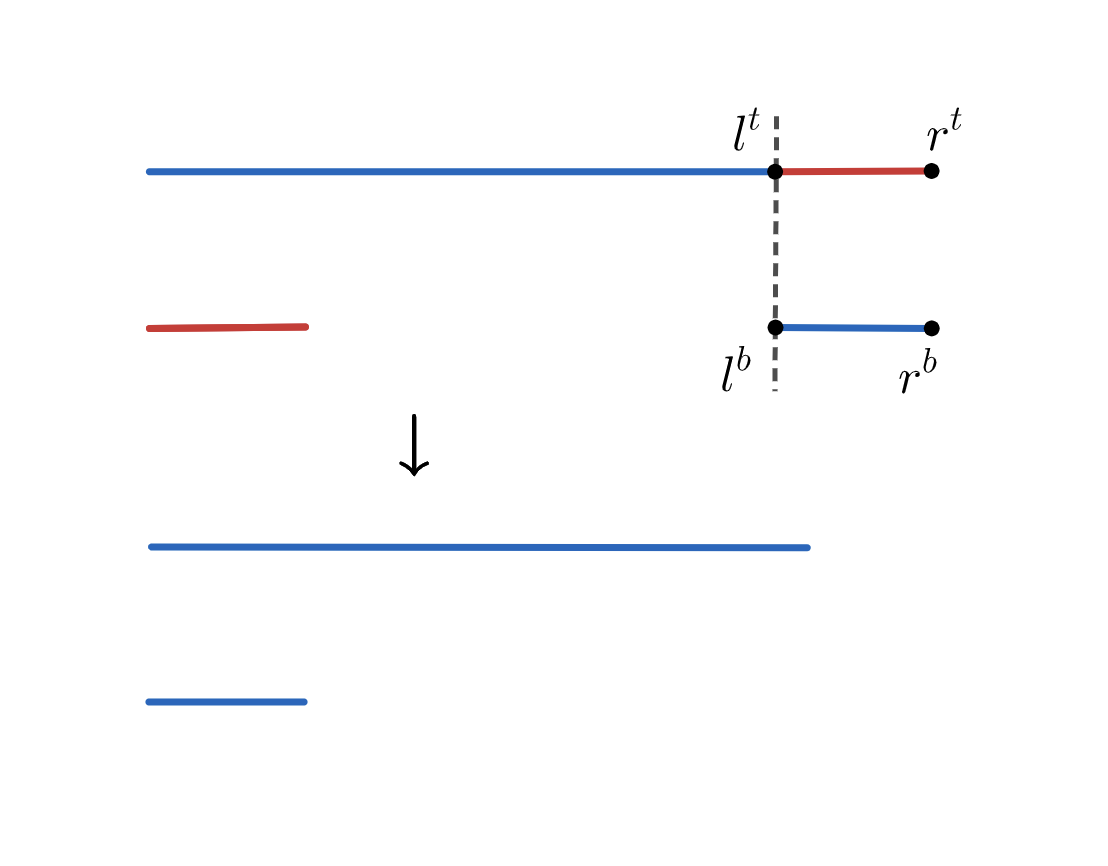}
\caption{Case 1b: neither wins.}
\end{figure}

\item \textbf{Case 2} \textit{(gap vs interval)}: assume that $r^t > r^b$ or $r^t < r^b$, i.e there is a gap to the right of the rightmost top, respectively bottom, interval. We distinguish further between the two cases where the gap looses, and where the interval looses: \vspace{2mm}

\begin{itemize}
   \item \textbf{Case 2a} \textit{(interval wins)}: if max$(l^t, l^b) < $ min$(r^t, r^b)$, we define $\lambda^{(1)} :=$ min$(r^b, r^t)$.
   If $r^t > r^b$ we say that the \textit{top interval wins} and \textit{bottom gap looses}. Similarly, if $r^b > r^t$ we say that the \textit{bottom interval wins} and \textit{top gap looses}. Note that in this case, $T^{(1)}$ is again a g-GIET with gaps on $d$ intervals and $\mathcal{A}^{(1)} = \mathcal{A}$.

   Explicitely, if $r^t > r^b$ (top interval wins), then
\begin{align*}
    I_{\alpha}^{t,1} = I_{\alpha}^{t} \cap [0,r^b), \hspace{2mm} \alpha \in \mathcal{A},\\
   T^{(1)}_{\alpha} = T|_{I_{\alpha}^{t,1}}, \hspace{2mm} \alpha \in \mathcal{A}.
\end{align*}
If $r^b > r^t$ (bottom interval wins), then
\begin{align*}
    I_{\alpha}^{t,1} = T^{-1}(I_{\alpha}^{b} \cap [0,r^t)), \hspace{2mm} \alpha \in \mathcal{A},\\
    T^{(1)}_{\alpha} = T|_{I_{\alpha}^{t,1}}, \hspace{2mm} \alpha \in \mathcal{A}.
\end{align*}

\begin{figure}[h]
\centering
\includegraphics[width=0.9\textwidth]{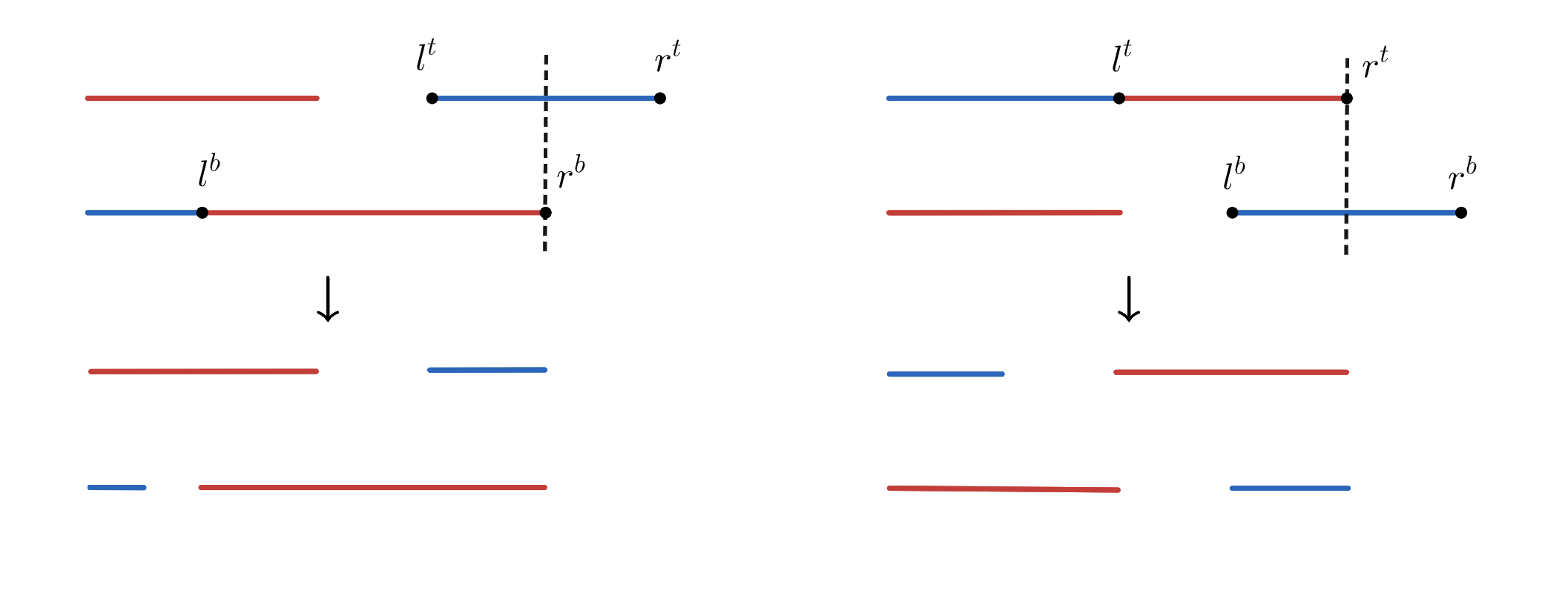}
\caption{Case 2a: on the right, the top interval wins and on the left, the bottom interval wins.}
\end{figure}

    \item \textbf{Case 2b} \textit{(gap wins)} if max$(l^t, l^b) > $min$(r^t, r^b)$ we define $\lambda^{(1)} :=$ max$(l^b, l^t)$.
    If $r^t < r^b$ we say that the \textit{top gap wins} and the \textit{bottom interval looses}. Similarly, if $r^t > r^b$ we say that the \textit{bottom gap wins} and the \textit{top interval looses}.

    Note that in this case, $T^{(1)}$ is a g-GIET on $d-1$ intervals. \smallskip

    Explicitely, if $r^t < r^b$  (top gap wins), then $\mathcal{A}^{(1)} := \mathcal{A} \backslash \{\alpha_b\}$ and if $r^t >r^b$ (bottom gap wins), then  $\mathcal{A}^{(1)} := \mathcal{A} \backslash \{\alpha_t\}$. In both cases
    \begin{align*}
    I_{\alpha}^{t,1} = I_{\alpha}^{t}, \hspace{2mm} \alpha \in \mathcal{A}^{(1)},\\
   T^{(1)}_{\alpha} = T|_{I_{\alpha}^{t,1}}, \hspace{2mm} \alpha \in \mathcal{A}^{(1)}.
\end{align*}


\begin{figure}[H]
\centering
\includegraphics[width=0.9\textwidth]{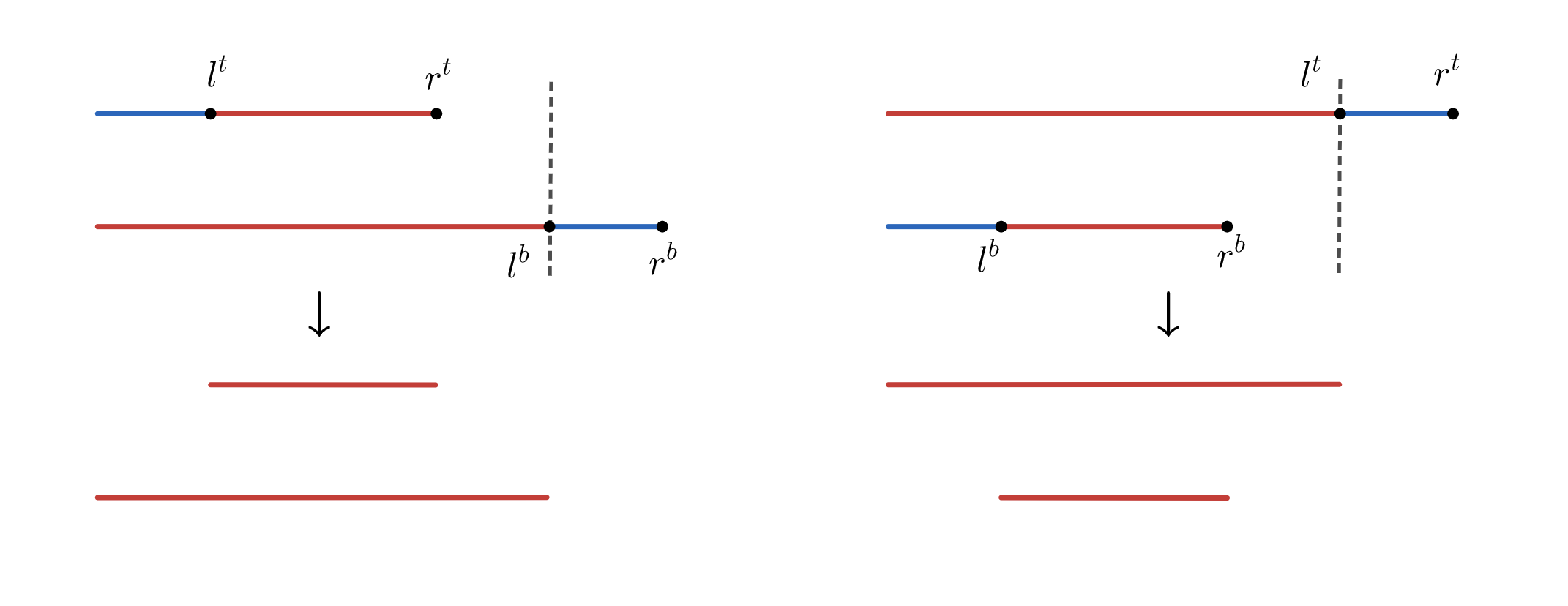}
\caption{Case 2b: On the right, the top gap wins, on the left, the bottom gap wins.}
\end{figure}

\end{itemize}
\end{itemize}

Note that in case 1a) and 2a), the number of intervals stays the same, whereas in case 1b) and 2b), the number of intervals decreases by one.
   \end{defs}
\subsubsection{Winning letters} If the winners and loosers are intervals and not gaps, then we can record their labels as follows. Recall that $\alpha_t$ and $\alpha_b$ are the letters in  ${ \mathcal{A}}$ defined by   $\pi^{-1}_t(d)=\alpha_t$ and $\pi^{-1}_b(d)=\alpha_b$. If the top (bottom) interval wins, we say $\alpha_t$ wins and $\alpha_b$ looses. Conversely, if the bottom interval wins, we say $\alpha_b$ wins and $\alpha_t$ looses.
We further say the letter $\alpha$ \textit{plays} if $\alpha$ either wins or looses, and that $\alpha$ \textit{does not play} if $\alpha$ neither wins nor looses.

\subsection{Rotation numbers and Poincar{\'e}-Yoccoz Theorem}
Let us now show that when the induction can be iterated forever, we can exploit it to associate to an \emph{infinitely renormalizable} g-GIET a \emph{combinatorial rotation number} (see  \S~\ref{sec:rotnumbers}). We then show that (infinite complete) rotation numbers characterize semi-conjugacy classes of g-GIETs, generalizing a result proved by Yoccoz for GIETs (see Theorem~\ref{thm:PoincareYoccoz} in \S~\ref{sec:PoincareYoccoz}).
\subsubsection{Iterates of the induction.} We say that \textit{the induction stops} if $d=1$ and $T$ is either in case 1b) or in case 2b), i.e $T^{(1)}$ is a g-GIET on $d-1 = 0$ intervals. After $n$ elementary steps of the induction, provided that at no point the induction stops, we obtain a sequence of g-GIETs $$T^{(i)} : I^{t,i} \to I^{b,i}, \; i \in \{1, \dots, n\},$$ on $d_i > 0$ intervals, for $i \in \{1, \dots, n\}$. We call a g-GIET for which the induction never stops \textit{infinitely renormalizable}.

\subsubsection{RV-stable g-GIETs}\label{sec:RV-stable} Iterating RV-induction, as we remarked, the number of intervals can decrease;
furthermore,  some letters which had played before may
stop playing. Both of these phenomena however can happen only finitely many times.
 In order to identify the g-GIETs in an orbit for which both these phenomena have ceased to happen, we give the following definition of an \textit{RV-stable} g-GIET:

\begin{defs}\label{def:RVstable} Let $T:I^t \to I^b$ be a g-GIET on $d \geq 1$ intervals and alphabet $\mathcal{A}$. We call an infinitely renormalizable g-GIET $T$ \textit{RV-stable} if for any  $\alpha \in \mathcal{A}$, either $\alpha$ never plays or plays infinitely often, where in the latter case it must either only win, or only loose, or both win and loose infinitely often.
\end{defs}

\begin{rem}\label{rk:RVstable} Every g-GIET $T$  which is infinitely renormalizable eventually becomes RV-stable after a finite steps of RV-induction, i.e.~there exists $N \in \mathbb{N}$ such that $T^{(n)}$ is RV-stable for all $n \geq N$. It  suffices to choose $N$ large enough so that for all steps $n\geq N$ no letter eventually stops playing, and all letters neither win nor loose only finitely often.
\end{rem}

We record in a lemma some properties of RV-stable g-GIETs which follow immediately from the definition.
\begin{lem}\label{lemma:RVstable}
Let   $T:I^t \to I^b$
be a g-GIET which is infinitely renormalizable and RV-stable.
Consider its iterates $(T^{(n)})_{n\in\mathbb{N}}$ under RV-induction.
Then:
\begin{enumerate}
\item  At every step of the induction, either two intervals play against each other and one wins, or 
an interval and a gap play against each other and the interval wins.
\item The number of intervals of the iterates of $T$ under RV induction is constant.
\end{enumerate}
\end{lem}
\begin{proof}
(1) follows from the fact that the number of intervals \emph{decreases} at the $n$-th step of RV induction only in the cases 1b) or 2b) described in \S~\ref{sec:RV_GIETs}, i.e.~if an interval plays against a gap and looses, or if an interval plays against an interval but neither wins. Hence, for any $n\in\mathbb{N}$, $T^{(n)}$ is always in  case 1a) or 2b), which means that $(2)$ holds.

(2) follows from Definition \ref{def:RVstable}: if the number of intervals decreases at the $n$-th RV-induction step, it means that an interval which has played at the $n$-th induction step stops playing, contradicting the definition of RV-stable.
\end{proof}

\subsubsection{Combinatorial rotation number}
\label{sec:rotnumbers} If $T$ is an RV-stable g-GIET with permutation given by $\pi$, then at each step of the induction an interval wins. Hence we can define the \textit{combinatorial rotation number} of $T$ as the set of data rot$(T)$ = $( \pi, \gamma )$ such that
$$\gamma = \alpha_1 \alpha_2 \alpha_3 \dots$$
where $\alpha_i \in \mathcal{A}$ is the label of the interval that wins at the $i$-th induction step.
\begin{rem}\label{rk:combinatorics} In \cite{Yoccoz_NotesClay}, the rotation number of a IET is a \emph{path} on the \emph{Rauzy-graph}\footnote{The \emph{Rauzy-graph} of a permnutation is the graph whose vertices are all possible combinatorial data obtained by performing Rauzy-Veech induction starting from an IET with combinatorial data $\pi$, and arrows represents top and bottom moves.}. For g-GIETs, one
can record steps of the induction of g-IET on a graph which contains as vertices all \emph{extended} combinatorial data, see \S~\ref{sec:extended_classes}.
\end{rem}

\begin{defs}\label{def:infcomplete} Let $T$ be RV-stable with combinatorial rotation number rot$(T)$ = $( \pi, \gamma )$. Then $T$ is said to be \textit{$\infty$-complete} if every $\alpha \in \mathcal{A}$ appears infinitely often in $\gamma$.
\end{defs}

We recall that if $T$ is an RV-stable \emph{standard} IET, the combinatorial rotation number $rot(T)$ is infinite-complete, as it was shown in \cite{MMY_Cohomological} (see also \cite{Yoccoz_NotesClay}).

\subsubsection{Semi-conjugacy for g-GIETs}\label{sec:semiconjugacy}
Infinite-complete rotation numbers for GIETs have been shown to characterize \emph{semi-conjugacy} classes. To show that this result extends to g-GIETs, let us first give the definition of semi-conjugacy between g-GIETs.

\begin{defs}Let $T: I^t \to I^b$, $T_0: I_0^t \to I_0^b$ be two g-GIETs. We say that $T$ is semi-conjugated to $T_0$ if there exists a map $h: I^t \to I_0^t$ that is continuous, monotone and surjective such that
\begin{align*}
    T_0 \circ h(x) = h \circ T(x)
\end{align*}
for all $x \in I^t$. If in addition $h$ is injective, we say that $T$ and $T_0$ are \textit{conjugated}.
\end{defs}
\subsubsection{Poincar{\'e}-Yoccoz theorem}\label{sec:PoincareYoccoz}
The following result extends the Poincaré theorem for circle diffeomorphisms and is a generalization of the same result proved by Yoccoz in \cite{Poincaré-Yoccoz} for GIETs (i.e without gaps). 

\begin{thm}\label{thm:PoincareYoccoz} Let $T$ be an RV-stable, $\infty$-complete g-GIET on $d \geq 2$ intervals. Then $T$ is semi-conjugated to a minimal IET $T_0$.
\end{thm}
\noindent The proof (given here below) follows closely the proof given by Yoccoz in \cite{Yoccoz_NotesClay} in the case when $T$ has no gaps, but requires some additional arguments. The main idea is that, since the combinatorial rotation number characterizes the steps of the RV induction, for each $n \in \mathbb{N}$, both $T$ and $T_0$ can be represented as \emph{towers} over  $T^{(n)}$ (see Definition~\ref{def:towerrep} whose floors arrange themselves in the \emph{same order} within the interval $[0,1)$. By letting $n$ go to infinity, one obtains a \textit{conjugacy} of $T$ and $T_0$ on the set of orbits of the left endpoints of the intervals of $T$ and $T_0$, which can then be extended to a \textit{semi-conjugacy} between $T$ and $T_0$ on the entire domains.

\begin{proof} Let $T$ be a g-GIET with gaps. Let rot$(T)$ = $( \pi, \gamma )$ be the rotation number of $T$, where $\gamma = \alpha_1 \alpha_2 \alpha_3 \dots$. Let $\tilde{\gamma}$ be obtained from $\gamma$ by \textit{removing} all letters which win against a gap, meaning that
\begin{align}\tilde{\gamma} = \alpha_{i_1} \alpha_{i_2} \alpha_{i_3} \dots
\end{align}
where for all $n \in \mathbb{N}$, $\alpha_{i_n}$ wins against an interval at the $i_n$-th induction step and for all $n$ with $n \in \mathbb{N} \backslash (i_n)_{n \in \mathbb{N}}$, $\alpha_n$ wins against a gap at the $n-$th induction step. Note that for a letter to win against a gap at the $n-th$ induction step, the letter must have won at the $(n-1)$-th induction step as well, hence also in $\tilde{\gamma}$ every letter in  $\mathcal{A}$ appears infinitely often. It  follows from Corollary 5 on p. 42 in \cite{Poincaré-Yoccoz} that there exists an RV-stable IET $T_0: I_0^t \to I_0^b$ with combinatorial rotation number rot($T_0$)=$(\pi, \tilde{\gamma})$.

For every $n \in \mathbb{N}$, consider the tower representation $\text{Rep}(I^{t,n})$ (see Section \S~\ref{sec:towers}) of $T$ over $T^{(n)}$ where $k^n_{\alpha}$ is the first return time of $I^{t,n}_{\alpha}$ to $I^{t,n}$,

\begin{align}
    \text{Rep}(I^{t,n}) = \bigcup_{\alpha \in \mathcal{A}} \text{Tow}(I_{\alpha}^{t,n})\\
    \text{Tow}(I_{\alpha}^{t,n}) = \bigcup_{i=1}^{k^n_\alpha-1} T^{i}(I^{t,n}_{\alpha})
\end{align}

Let $(\alpha, i) \in \mathcal{A} \times \mathbb{N}$ denote the label of the interval $T^{i}(I^{t,n}_{\alpha}) \subset $ Rep $(I^{t,n})$, and consider the map $ord_n$, defined on some subset of $\mathcal{A} \times \mathbb{N}$, which describes the order of the intervals, i.e
\begin{equation}\label{eq:order}
     ord_n(\alpha,i) = l
\end{equation}
if and only if the interval with label $(\alpha, i)$ is the $l$-th interval in $[0,1]$ of Rep $(I^{t,n})$ counted from left to right. Define $\text{Rep}(I_0^{t,n})$, $\text{Tow}(I_{0,\alpha}^{t,n})$ and $ord_{n,0}$ analogously for $T_0$, where $k^n_{0,\alpha}$ is the first return time of $I^{t,n}_{0,\alpha}$ to $I_0^{t,n}$.

Let $n \in \mathbb{N} \backslash (i_n)_{n \in \mathbb{N}}$, i.e $T^{(n)}$ is obtained from $T^{(n-1)}$ by performing one step of Rauzy-Veech induction where a gap wins. Then
\begin{align}
    k_{\alpha}^{n} = k_{\alpha}^{n-1},\\
    ord_{n} = ord_{n-1},
\end{align}
for all $\alpha \in \mathcal{A}$.
In particular, for all $n \in \mathbb{N}$ and all $\alpha \in \mathcal{A}$.
\begin{equation}\label{eq:kfirst}
    k_{\alpha}^{i_n+1} = k_{\alpha}^{i_{n+1}}, \hspace{2mm}\\
\end{equation}
\begin{equation}\label{eq:ksecond}
   ord_{i_n+1} = ord_{i_{n+1}}.
\end{equation}
We now want to prove by induction that for all $\alpha \in \mathbb{A}, n \in \mathbb{N}$,
\begin{equation}\label{eq:kequal}
    k_{\alpha}^{i_n} = k_{0,\alpha}^{n}
\end{equation}
\begin{equation}\label{eq:kequal2}
    ord_{i_n} = ord_{0,n},
\end{equation}
Since $T$ and $T_0$ have the same permutation $\pi$, the claim follows for $n = 0$. Assume the claim is true for $n=k$. Then the same letter wins at the $n$-th induction step of $T_0$ and the $i_n$-th induction step of $T$.
W.l.o.g. assume that the letter at the top wins, and let $\gamma$ be the winner and $\beta$ be the looser. Then $k_{\beta}^{i_n+1} = k_{\beta}^{i_n} +  k_{\gamma}^{i_n}$ and $k_{\alpha}^{i_n+1} = k_{\alpha}^{i_n}$ for all $\alpha \neq \beta$, and similarly also $k_{0,\beta}^{n+1} = k_{0,\beta}^{n} +  k_{0,\gamma}^{n}$ and $k_{0,\alpha}^{n+1} = k_{0,\alpha}^n$ for all $\alpha \neq \beta$. Hence, by the induction assumption, also  $k_{\alpha}^{i_n+1} = k_{0,\alpha}^{n+1}$, and hence by (\ref{eq:kfirst}) it follows $k_{\alpha}^{i_{n+1}} = k_{0,\alpha}^{n+1}$, which concludes the proof of (\ref{eq:kequal}).

Furthermore,
\begin{align} \text{Tow}(I_{\beta}^{t,i_n+1}) &\subset Tow(I_{\beta}^{t,i_n}) \cup \text{Tow}(I_{\gamma}^{t,i_n}),\\
\text{Tow}(I_{\gamma}^{t,i_n+1}) &\subset Tow(I_{\gamma}^{t,i_n}),\\
\text{Tow}(I_{\alpha}^{t,i_n+1}) &\subset Tow(I_{\alpha}^{t,i_n}) \text{ for all } \alpha \in \mathcal{A} \backslash \{\beta, \gamma \},
\end{align}
where in each expression, each interval on the left is contained in exactly one interval on the right, and the same inclusions hold for the map $T_0$. Therefore, since the intervals on the right hand side have the same order for $T$ and $T_0$ by the induction assumption, it follows that the intervals on the left hand side also have the same order for $T$ and $T_0$, in particular $ord_{i_n+1} = ord_{0,n+1}$ and hence by (\ref{eq:ksecond}) it follows that $ord_{i_{n+1}} = ord_{0,n+1}$ which concludes the proof of (\ref{eq:kequal2}).

Now pick $\alpha \in \mathcal{A}$ such that $\alpha$ looses at the bottom infinitely many times (this is possible since at every induction step, a letter has to loose, and there are only finitely many letters).

Let $l_{\alpha}^t$ denote the left endpoint of the interval $I_{\alpha}^t$, and $l_{0,\alpha}^t$ denote the left endpoint of $I_{\alpha}^t$. Since $\alpha$ looses infinitely often at the bottom (i.e infinitely often, the tower corresponding to $\alpha$ \textit{grows} in forward time, see below), and since $l_{\alpha}^t$ ($l_{0,\alpha}^t$) are always the left endpoint of an interval in the tower corresponding to $\alpha$, it follows that for any $n \in \mathbb{N}$, there exists $m_n \in \mathbb{N}$, $m_n \to \infty$ as $n \to \infty$, such that
\begin{equation}\label{eq:tow1}
\{l_{\alpha}^{t}, T(l_{\alpha}^{t}), \dots, T^{m_n}(l_{\alpha}^{t}) \} \subset \text{Tow}(I_{\alpha}^{t,n}),
\end{equation}
\begin{equation}\label{eq:tow2}
\{l_{0,\alpha}^{t}, T(l_{0,\alpha}^{t}), \dots, T^{m_n}(l_{0,\alpha}^{t}) \} \subset \text{Tow}(I_{0,\alpha}^{t,n}).
\end{equation}
Here the fact that $m_n$ is the same both for $T$ and $T_0$ comes from the fact that by (\ref{eq:kequal}), the heights of $\text{Tow}(I_{\alpha}^{t,n})$ and $\text{Tow}(I_{0,\alpha}^{t,n})$ are equal. By  (\ref{eq:kequal2}), the orbits in (\ref{eq:tow1}) and (\ref{eq:tow2}) have the same order in $I^t$ and $I_0^t$.

Let now $\mathcal{O}_T^+(l_{\alpha}^t)$ denote the forward orbit of $T$ through $l_{\alpha}^t$, and $\mathcal{O}_{T_0}^+(l_{0,\alpha}^t)$ denote the forward orbit of $T_0$ through $l_{0,\alpha}^t$. Since $m_n \to \infty$, also the sets $\mathcal{O}_T^+(l_{\alpha}^t)$ and $\mathcal{O}_{T_0}^+(l_{0,\alpha}^t)$ have the same order in $I^t$ and $I_0^t$, meaning that the map
\begin{align}
h: \mathcal{O}_T^+(l_{\alpha}^t) \to \mathcal{O}_{T_0}^+(l_{0,\alpha}^t) \\
h(T^{j}(l_{\alpha}^{t,n})) = T_0^{j}(l_{0,\alpha}^{t,n})
\end{align}
is monotone. Furthermore, by construction, $h$ conjugates $T$ to $T_0$ on $\mathcal{O}_T^+(l_{\alpha}^t)$, i.e
\begin{align}
h \circ T (x) = T_0 \circ h (x), \hspace{2mm} x \in \mathcal{O}_T^+(l_{\alpha}^t).
\end{align}

If $\bar{\mathcal{O}}_T^+(l_{\alpha}^t), \bar{\mathcal{O}}_{T_0}^+(l_{0,\alpha}^t)$ denote the closure of $\mathcal{O}_T^+(l_{\alpha}^t)$, $\mathcal{O}_{T_0}^+(l_{0,\alpha}^t)$ in $I^t, I_0^t$, then $h$ can be extended to a continuous function
\begin{align}
\bar{h}: \bar{\mathcal{O}}_T^+(l_{\alpha}^t) \to \bar{\mathcal{O}}_{T_0}^+(l_{0,\alpha}^t) = I_0^t.
\end{align}
which conjugates $T$ to $T_0$ on $\bar{\mathcal{O}}_T^+(l_{\alpha}^t)$. Note that $\bar{\mathcal{O}}_{T_0}^+(l_{0,\alpha}^t) = I_0^t$, since $T_0$ is a minimal IET, meaning every positive (semi-)orbit is dense in $I_0^t$. To extend the domain of $\bar{h}$ to all of $I^t$, define $H:I^t \to I_0^t$ to be
\begin{align}
H(x) = \begin{cases}
\bar{h}(x) & x \in \bar{\mathcal{O}}_T^+(l_{\alpha}^t),\\
\bar{h}(x^-) & x \in I^t \backslash \bar{\mathcal{O}}_T^+(l_{\alpha}^t).\\
\end{cases}
\end{align}
where $x^- \in \bar{\mathcal{O}}_T^+(l_{\alpha}^t)$ is the left endpoint of the connected component in $I^t \backslash \bar{\mathcal{O}}_T^+(l_{\alpha}^t)$ which contains $x$. Then $H:I^t \to I_0^t$ defines a semi-conjugacy between $T$ and $T_0$ and the proof is completed.

\end{proof}




\section{Renormalization and Recurrence}\label{sec:renormandrecurrencesec} The goal of this section is to introduce the renormalization scheme that will allows us to show that any given g-GIET can be decomposed into different dynamical components according to Theorem 1.1. In this chapter, we establish the elementary step of this renormalization scheme by using Rauzy-Veech induction for g-GIETs in order to find a first dynamical component of the decomposition. In particular, in Section \S~\ref{sec:renormlimit}, we show that the induction applied to a g-GIET $T$ is able to detect a specific recurrent orbit closure of $T$, namely the `\emph{rightmost}' one (in a sense made more precise in Proposition \ref{mainprop}). In Section \S~\ref{sec:somecombinatorialarguments}, we then show that the induction also allows us to find a region of recurrence for $T$ which contains the recurrent orbit closure and has the structure of a \textit{tower representation} for $T$ (see Section \S~\ref{sec:towers}) over a collection of subintervals of $I^t$.

\subsection{Renormalization limit  of a g-GIET}\label{sec:renormlimit} We define the \textit{renormalization limit} of a g-GIET $T$ and explain its relationship with recurrent orbits of $T$.

\subsubsection{Definition} Consider a g-GIET $T: I^t \to I^b$ on $d$ intervals, as before let
\begin{align}
I_{\pi^{-1}_t(d)} = [l^t, r^t), \hspace{2mm} I_{\pi^{-1}_b(d)} = [l^b, r^b).
\end{align}
Let $T$ be $n$ times renormalizable. We obtain, after $n$ steps of Rauzy-Veech induction, a g-GIET $T^{(n)}$ on $d_n$ intervals with permutation given by $\pi^n = (\pi_t^{n}, \pi_b^n)$, where $d_n \leq d$ is the number of intervals of $T^{(n)}$. We let
\begin{align}
I_{\pi_t^{-n}(d_n)}^n = [l^{t,n}, r^{t,n}), \hspace{2mm} I_{\pi_b^{-n}(d_n)}^n = [l^{b,n}, r^{b,n}),
\end{align}
be the right and leftmost interval of the g-GIET $T^{(n)}$. We are interested in the limit of the right endpoint of the rightmost interval of $T^{(n)}$ as $n$ tends to infinity:



\begin{defs} \label{def:renormlimit}Let $T:I^t \to I^b$ be an infinitely renormalizable g-GIET. Then the \textit{renormalization limit } $l$ of $T$ is defined as
\begin{align*}
l = \lim_{n\to\infty} \hspace{1mm}r^{t,n}. \vspace{1.5mm}
\end{align*}
If $T$ is RV-stable and if $l=0$ we say that $T$ is \textit{indecomposable}.
\end{defs}

\subsubsection{The special case of the cylinder}\label{sec:convention} We prove a first lemma that illustrates how the renormalization limit  of a g-GIET is able to `\emph{detect}' recurrent orbit closures in the case when the rightmost top and bottom label are equal, i.e when they form what we call a \textit{cylinder}. We show that if the rightmost top and bottom endpoints are not equal (i.e $r_t \neq r_b$) and if there are fixed points in the top rightmost interval, then the renormalization limit  is equal to the rightmost of all fixed points and that if there are no fixed points, the renormalization limit  is not contained in the intersection of the rightmost top and bottom interval.

\begin{lem} \label{cylinderlemma} Let $T: I^t \to I^b$ be infinitely renormalizable. Let $r^t \neq r^b$, let $\alpha := \pi^{-1}_t(d) = \pi^{-1}_b(d)$. Then
\begin{align*}
    l \geq \text{max}\hspace{1mm}(l^t,l^b) \iff \exists x_0 \in I_{\alpha}^t \hspace{1mm} \text{such that}\hspace{1mm} T(x_0)=x_0.
\end{align*}
In this case, $T(l) =l$ and $l = $max $\{x \in I_{\alpha}^t \hspace{1mm}| \hspace{1mm}T(x) = x\}$.
\end{lem}

\begin{figure}[h]
\centering
\includegraphics[width=0.6\textwidth]{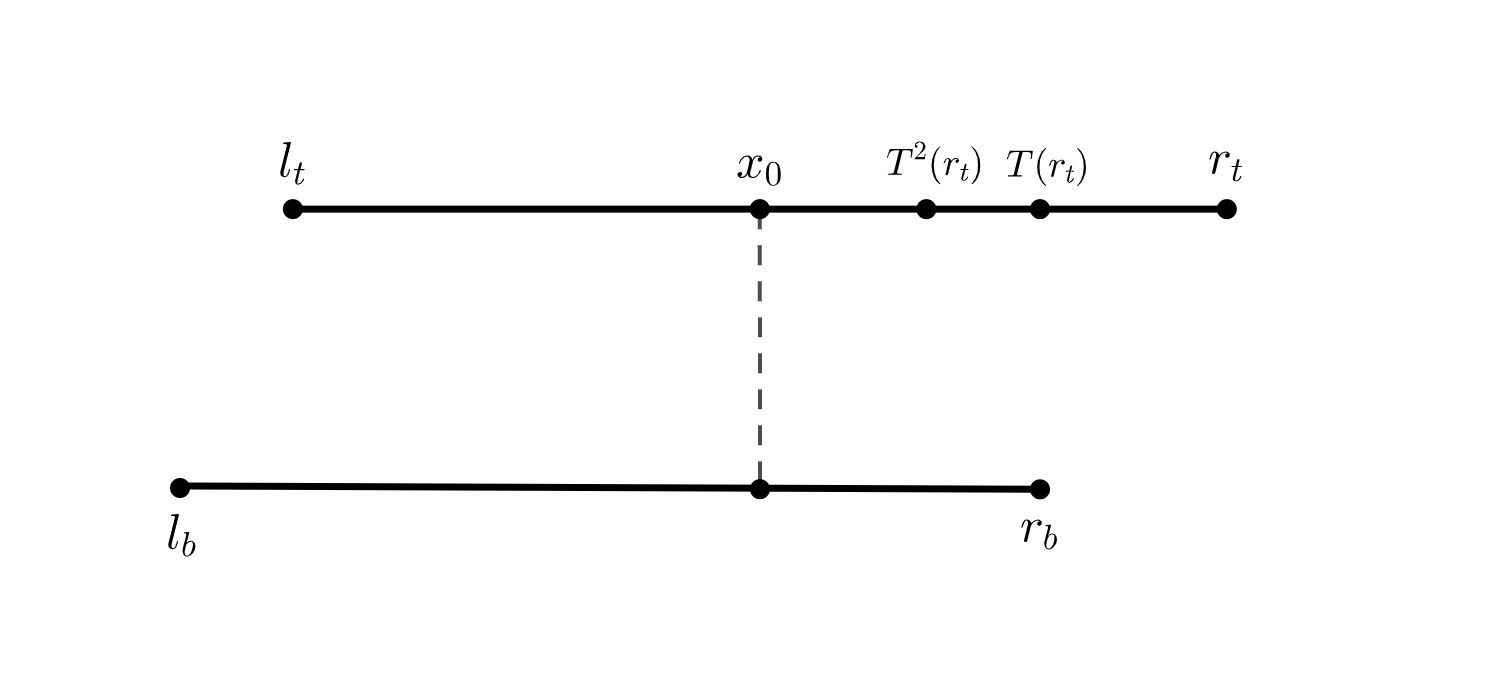}
\caption{$r^t > r^b$ and $\lim_{n\to\infty}T^{n}(r^t) = x_0$.}
\end{figure}

\begin{proof} Assume first there exists a fixed point $x_0$ in $I_{\alpha}^t$. We consider two cases:
\begin{enumerate}[]
\item $r^t > r^b$. Then for any induction step  $n$ it holds $r^{t,n}= T^{n}(r^t)$, where $T^{n}$ is the $n$-th fold composition of $T$ with itself. As $T$ is orientation preserving, for all $x \in I_{\alpha}^t$ it holds that $\lim_{n\to\infty}T^{n}(x) \geq x_0$. We conclude that $l = \lim_{n\to\infty}r^{t,n} = \lim_{n\to\infty}T^{n}(r^t) \geq x_0$.
\item $r^t < r^b$. Then for any induction step $n$ it holds $r^{b,n}=T^{-n}(r^b)$. Furthermore, for any $x \in I_{\alpha}^b$ it holds that $\lim_{n\to\infty}T^{-n}(x) \geq x_0$. But then we conclude $l = \lim_{n\to\infty}r^{t,n} = \lim_{n\to\infty}r^{b,n} = \lim_{n\to\infty}T^{-n}(r^b) \geq x_0$.
\end{enumerate}
In both cases, $l \geq x_0 \geq \text{max}(\hspace{1mm}l^t,l^b)$.
Conversely, in case (1), if $l = \lim_{n\to\infty}T^{n}(r^{t}) > \text{max}\hspace{1mm}(l^t, l^b)$, then as $T$ is continuous, $T(l) =  \lim_{n\to\infty}T^{n+1}(r^{t}) = l$ and $l$ is a fixed point. The same argument applies for case (2), where $T(l) = \lim_{n\to\infty}T^{-n}(r^b) = l$. Together with (1) and (2) it then follows that $l = $max $\{x \in I_{\alpha}^t \hspace{1mm}| \hspace{1mm}T(x) = x\}$.
\end{proof}

\subsection{Some combinatorial arguments}\label{sec:somecombinatorialarguments}Before we can prove that RV-induction for g-GIETs also detects recurrent orbit closures other than fixed points, we must first develop a few combinatorial lemmas and propositions on RV-induction for g-GIETs. The main proposition we prove in this section is Proposition \ref{intervalthatcontainslremainsunchanged}, which states that if we consider the top and bottom interval of an RV-stable g-GIET that contains $l$, then the label of these intervals are invariant under RV-induction and win infinitely often.

We first prove the following simple lemma, which states that the top and bottom interval corresponding to a label $\alpha$ remain unchanged under Rauzy-Veech induction as long as $\alpha$ does not play:

\begin{lem} \label{intervalsthatdontplayremainunchanged} Let $T:I^t \to I^b$ be an infinitely renormalizable g-GIET. Let $\alpha \in \mathcal{A}$ and let $m \in \mathbb{N}$ be such that $\alpha$ plays for the first time at the $m-$th induction step. Then $I_{\alpha}^{t,n} = I_{\alpha}^{t}$ and $I_{\alpha}^{b,n} = I_{\alpha}^{b}$ for all $n \in \{1, \dots, m-1 \}$.
\end{lem}

\begin{proof}
    Follows directly from the definition of Rauzy-Veech induction for g-GIETs in Section \S~\ref{sec:RV g-GIETs}: at each step, only the top and bottom intervals that carry the labels of the players are modified in the next induction step, whereas all other intervals remain the same.
\end{proof}

The next lemma asserts that intervals whose labels win infinitely often are never contained in $[0,l)$.

\begin{lem} \label{winninglettersintersectl}
    Let $T: I^t \to I^b$ be an RV-stable g-GIET and let $l$ be its renormalization limit. Let $\mathcal{A}_{\infty}$ be the set of letters that win infinitely often. Then $I_{\alpha}^{t,n} \cap [l,1) \neq \emptyset$ and $I_{\alpha}^{b,n} \cap [l,1) \neq \emptyset$ for all $\alpha \in \mathcal{A}_{\infty}$ and for all $n \in \mathbb{N}$.
\end{lem}

\begin{proof} Since $T^{(n)}$ is again RV-stable for all $n \in \mathbb{N}$, it is enough to show that $I_{\alpha}^{t} \cap [l,1) \neq \emptyset$ and $I_{\alpha}^{b} \cap [l,1) \neq \emptyset$ for all $\alpha \in \mathcal{A}_{\infty}$. Assume $I_{\alpha}^t \cap [l,1) = \emptyset$. Note that then $I_{\alpha}^{t,n} \cap [l,1) = \emptyset$ for all $n \in \mathbb{N}$ and therefore, as $\alpha \in \mathcal{A}_{\infty}$, the interval $I_{\alpha}^b$ has to win infinitely often. But if $n \in \mathbb{N}$ is such that $I_{\alpha}^{b,n}$ wins, if we denote the interval that looses by $I_{\beta}^{t,n}$, then $I_{\beta}^{t,n+1} \subset I_{\alpha}^t$. Hence for $m > n$ both $I_{\alpha}^{t,m}$ and $I_{\beta}^{t,m}$ are contained in $[0,l)$ and the label $\beta$ cannot loose against the label $\alpha$ anymore. As this is true for any interval which looses against $I_{\alpha}^{t,n}$ for some $n \in \mathbb{N}$, the label $\alpha$ can only win finitely many times against other intervals. But $\alpha$ can also only win finitely many times against a gap as there are only finitely many gaps. This is a contradiction to $\alpha$ winning infinitely often, and hence $I_{\alpha}^t \cap [l,1) \neq \emptyset$. The argument to show that $I_{\alpha}^b \cap [l,1) \neq \emptyset$ follows analogously.
\end{proof}

We are now in the position to prove the main proposition of this chapter.

\begin{prop} \label{intervalthatcontainslremainsunchanged} Let $T: I^t \to I^b$ be an RV-stable g-GIET. Then there exist $\alpha, \beta \in \mathcal{A}_{\infty}$ such that $l \in I_{\alpha}^{t,n} \cap I_{\beta}^{b,n}$ for all $n \in \mathbb{N}$.
\end{prop}

\begin{proof} If $l$ was contained in a top or bottom gap, then there exists $n \in \mathbb{N}$ such that the gap plays at the $n-$th induction step and remains unchanged for all previous induction steps. But as $T$ is RV-stable, the gap must loose, as otherwise the number of intervals of $T$ would decrease by one. But then the renormalization limit  $l$ must be strictly smaller than the leftmost endpoint of the gap, which is a contradiction. Hence there exist $\alpha, \beta \in \mathcal{A}$ such that  $l \in I_{\alpha}^{t,n} \cap I_{\beta}^{b,n}$.

We first consider the label $\alpha$. By Lemma \ref{intervalsthatdontplayremainunchanged} there exists $n \in \mathbb{N}$ such that the label $\alpha$ plays at the $n$-th induction step and $I_{\alpha}^{t,n}=I_{\alpha}^t$ and $I_{\alpha}^{b,n} = I_{\alpha}^b$.

If $\alpha = \beta$, then the proposition follows from Lemma \ref{cylinderlemma} (note that the rightmost endpoints of $I_{\alpha}^{t,n}$ and $I_{\alpha}^{b,n}$ cannot be the same as otherwise the renormalization limit  must be strictly smaller than $l$). Hence we may assume that  $\alpha \neq \beta$.


We claim that $I_{\alpha}^{b,n} \cap [l,1) \neq \emptyset$: assume by contradiction that $I_{\alpha}^{b,n} \subset [0,l)$. Then $\alpha$ plays for the first time at the top, i.e at time $n$ the interval $I_{\alpha}^{t,n}$ either wins or looses. But as $I_{\alpha}^{t,n}$ contains $l$, the interval $I_{\alpha}^{t,n}$ must win at the $n$-th induction step and thus $l \in I_{\alpha}^{t,n+1}$. Hence, at the $(n+1)$-th induction step, $I_{\alpha}^{t,n+1}$ must win again and $l \in I_{\alpha}^{t,n+2}$. By induction, the interval $I_{\alpha}^{t,m}$ must win for all $m \geq n$. But then by Lemma \ref{winninglettersintersectl} we deduce that $I_{\alpha}^{b,n} \not \subset [0,l)$.

Hence we can assume that $I_{\alpha}^{b,n} \cap [l,1) \neq \emptyset$. As $\alpha \neq \beta$ it follows that $l \notin I_{\alpha}^{b,n}$ and in particular $I_{\alpha}^{b,n} \subset (l,1)$. We further note that $I_{\alpha}^{b,n} \not \subset I_{\alpha}^{t,n}$, as otherwise there exists a fixed point $x \in I_{\alpha}^{b,n}$ for $T^{(n)}$ and by Lemma \ref{cylinderlemma} it follows that $l \in I_{\alpha}^{b,n}$, which we excluded. Hence, sup $I_{\alpha}^{b,n}$ > sup $I_{\alpha}^{t,n}$, and if $\alpha$ plays for the first time at the $n$-th induction step, then it must play at the bottom.

Now let $\gamma \in \mathcal{A}$ be such that $I_{\alpha}^{b,n}$ and $I_{\gamma}^{t,n}$ play against each other at the $n$-th induction step (if $\alpha$ plays against a gap, the argument follows analogously). Then one of the following cases holds: either a) $\alpha$ wins and $T^{(n)}(l) \in I_{\gamma}^{t,n}$, or b) $\gamma$ wins, or c) $\alpha$ wins and $T^{(n)}(l) \notin I_{\gamma}^{t,n}$ (see also Figure \ref{fig:9} and \ref{fig:10}).

\begin{figure}[h]
\centering
\includegraphics[width=1.0\textwidth]{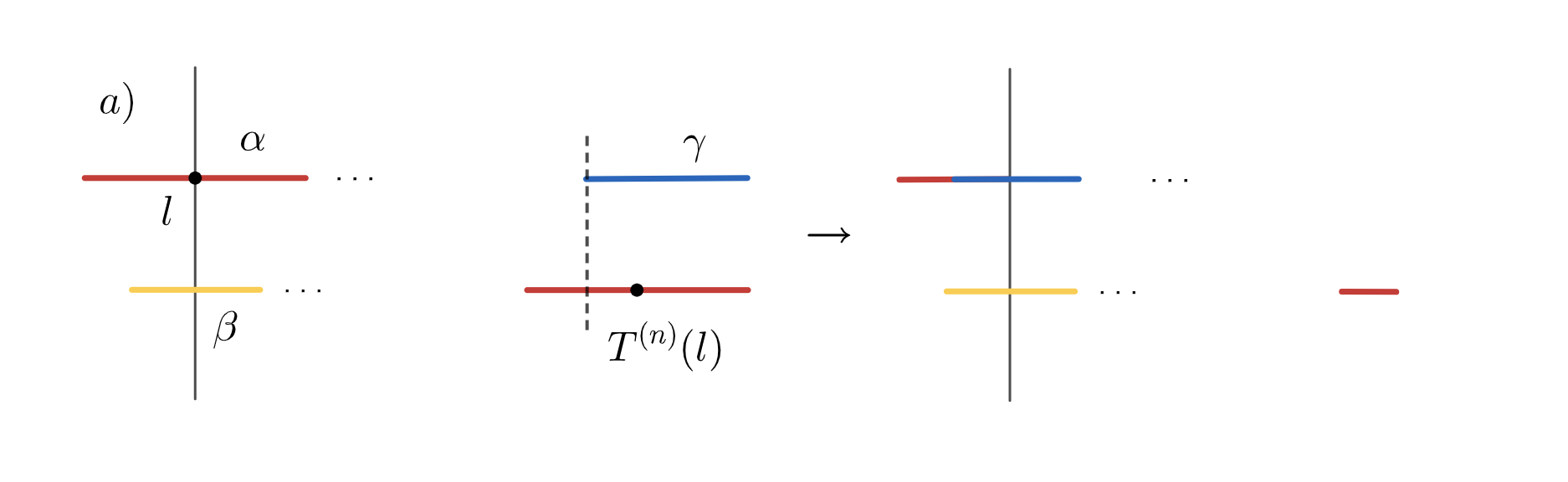}
\caption{\label{fig:9} Case a) $\alpha$ wins and $T^{(n)}(l) \in I_{\gamma}^{t,n}$.}
\end{figure}

We claim that case a) cannot arise: indeed, in case a) the letter $\alpha$ wins. But as depicted in Figure 10, after applying one more RV-induction step, $I_{\alpha}^{t,n+1} \cap [l,1) = \emptyset$ and hence by Lemma \ref{winninglettersintersectl} it follows that $\alpha$ always looses, which is a contradiction as $\alpha$ wins at least once, namely at the $n$-th induction step.

\begin{figure}[h]
\centering
\includegraphics[width=1.0\textwidth]{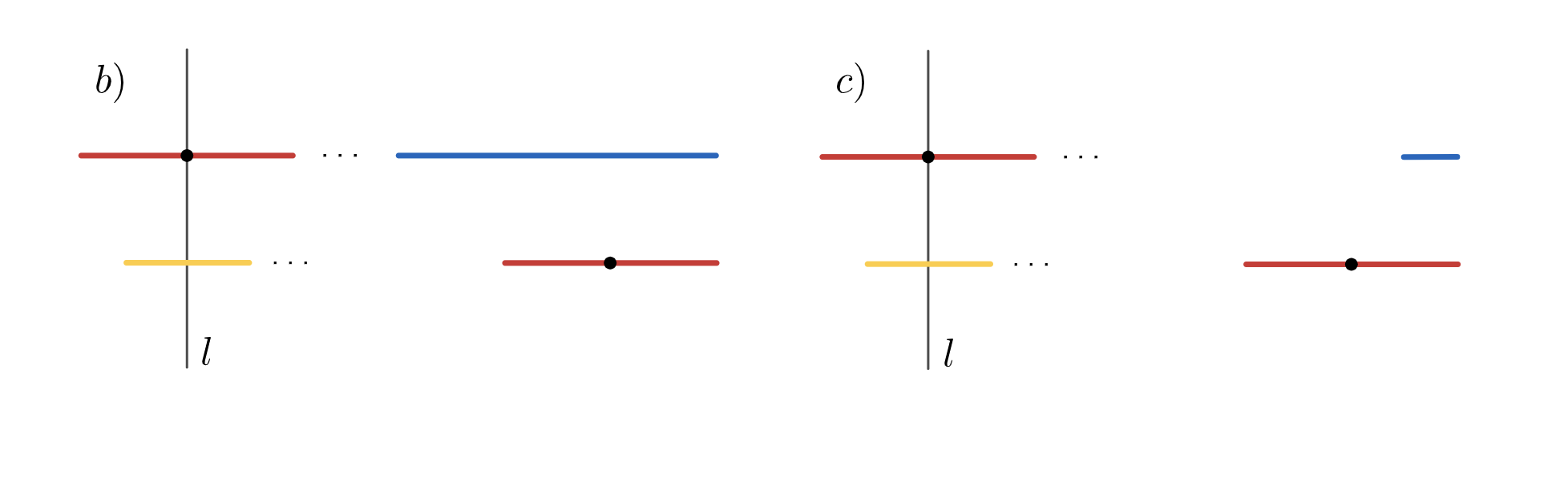}
\caption{\label{fig:10} Case b) $\gamma$ wins and case c) $\alpha$ wins and $T^{(n)}(l) \notin I_{\gamma}^{t,n}$}.
\end{figure}

Hence, only case b) and case c) may arise for Rauzy-Veech induction at the $n$-th induction step. But note that both in case b) and in case c) it holds that $l \in I_{\alpha}^{t,n+1}$, i.e the label of the interval whose closure contains $l$ does not change. Hence, by induction,  $l \in I_{\alpha}^{t,m}$ for all $m \in \mathbb{N}$. To show that $\alpha \in \mathcal{A}_{\infty}$, assume by contradiction that $\alpha$ always looses, i.e whenever $\alpha$ plays it is in case b). Then $I_{\alpha}^{t,n}= I_{\alpha}^{t}$ for all $n \in \mathbb{N}$, which is a contradiction as the renormalization limit  $l$ is smaller than the rightmost endpoint of $I_{\alpha}^{t}$. Thus, we have shown that $l \in I_{\alpha}^{t,m}$ for all $m \in \mathbb{N}$ and that $\alpha \in \mathcal{A}_{\infty}$.

The proof that $l \in I_{\beta}^{b,m}$ for all $m \in \mathbb{N}$ and that $\beta \in \mathcal{A}_{\infty}$ follows from the same argument applied to $T^{-1}$, hence the proposition follows.
\end{proof}

\subsection{Renormalization limit  and recurrence: the general case}\label{sec:renormandrecurrence}The goal of this section is to prove Proposition \ref{mainprop} below. Its statement can be summarized as follows: given a g-GIET $T$, if we consider the set of the leftmost endpoints of all recurrent orbit closures of $T$, then the renormalization limit  is equal to the maximum of that set. In Figure 8, a simple example with two periodic orbits is depicted.

\begin{prop} \label{mainprop}
     Let $T$ be an RV-stable g-GIET and let $l$ be its renormalization limit . Then there exists a recurrent orbit closure $\mathcal{W}$ whose leftmost endpoint is equal to $x$ such that $l = x$. Furthermore, for any other recurrent orbit closure $\mathcal{W}'$ with leftmost endpoint equal to $x'$ it holds that $x' < x$.
\end{prop}

\begin{figure}[h]
\centering
\includegraphics[width=0.9\textwidth]{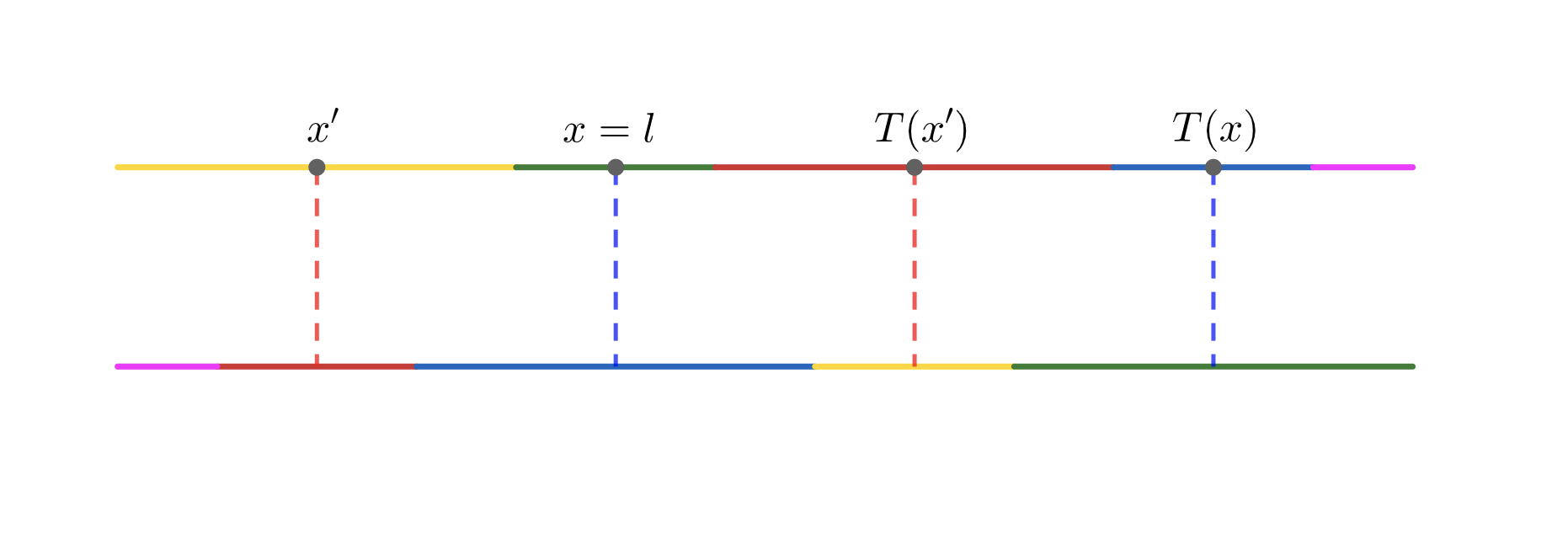}
\caption{A g-GIET with two periodic orbits $\{x,T(x)\}$ in blue and $\{x',T(x')\}$ in red. The renormalization limit  $l$ is equal to $x$, as $x > x'$.}
\end{figure}

We split the proof of this proposition into the following two lemmas:

\begin{lem} \label{recurrencelemma}  Let $T: I^t \to I^b$ be RV-stable with renormalization limit  $l$. Then $l$ is the leftmost endpoint of the closure of a recurrent orbit.
\end{lem}
\begin{proof} 
By Lemma \ref{intervalthatcontainslremainsunchanged} there exist $\alpha$, $\beta \in \mathcal{A}_{\infty}$ such that $l \in I_{\alpha}^{t,n} \cap I_{\beta}^{b,n}$ for all $n \in \mathbb{N}$. If $\alpha = \beta$, it follows from Lemma \ref{cylinderlemma} that $l$ is a fixed point and hence the leftmost endpoint of the closure of a recurrent orbit.

If $\alpha \neq \beta$, then since Lemma \ref{winninglettersintersectl} together with Lemma \ref{intervalthatcontainslremainsunchanged} imply that $I_{\alpha}^{b,n} \, \cap \, [l,1) \neq \emptyset$ for all $n \in \mathbb{N}$, it follows that $I_{\alpha}^{b,n} \subset (l,r^{b,n})$ for all $n \in \mathbb{N}$. But as $l = \lim_{n\to\infty} \hspace{1mm}r^{t,n} = \lim_{n\to\infty} \hspace{1mm}r^{b,n}$, it follows that for all $\epsilon > 0$ there exists $n \in \mathbb{N}$ such that $|l - T^{(n)}(l)| < \epsilon$. But then the forward orbit through $l$ is non-trivially $\omega$- recurrent for $T$. Similarly,  $I_{\beta}^{t,n} \subset (l,r^{t,n})$ for all $n \in \mathbb{N}$ and hence the backward orbit through $l$ is non-trivially backward recurrent, in particular, the full orbit through $l$ is recurrent and its closure is a quasiminimal. The fact that $l$ is the leftmost endpoint of the quasiminimal follows from the fact that $I_{\alpha}^{b,n}, I_{\beta}^{t,n} \not \subset [0,l)$ for all $n \in \mathbb{N}$.



\end{proof}

The second lemma needed in order to complete the proof of Proposition \ref{mainprop} is the following:

\begin{lem} \label{leftendpointlemma} Let $T:I^t \to I^b$ be a g-GIET and let $\mathcal{W}$ be a recurrent orbit closure of $T$ with leftmost endpoint equal to $x$. Then $T$ is infinitely renormalizable and $l \geq x$.
\end{lem}

\begin{proof} It is enough to show that $r^{t,n} > x$ for all $n \in \mathbb{N}$, as then $l = \lim_{n\to\infty} \hspace{1mm}r^{t,n} \geq x$ exists and hence $T$ is infinitely renormalizable.

We first show that there exists no $n \in \mathbb{N}$ such that $0 < r^{t,n} \leq x$. Assume by contradiction that such an $n$ exists and assume further that $n$ is the smallest integer for which $r^{t,n} \leq x$, i.e $r^{t,n-1} > x$. Note that as $x$ is the leftmost endpoint of a recurrent orbit closure, the first return map in a neighbourhood around $x$ is always defined and $x$ is never contained in a gap. Therefore, at the $n$-th induction step, two intervals are playing against each other, namely the rightmost top and bottom intervals $[l^{t,n-1}, r^{t,n-1})$ and $[l^{b,n-1}, r^{b,n-1})$ of $T^{(n-1)}$ which both contain $x$ and have the same rightmost endpoint $r^{t,n-1} = r^{b,n-1}$.

We consider two cases: if the label of the top and bottom rightmost interval of $T^{(n-1)}$ are the same, then $x$ is necessarily a fixed point for $T^{(n-1)}$ and by Lemma \ref{cylinderlemma} $r^{t,m} > x$ for all $m \in \mathbb{N}$ which is a contradiction to the assumption that $r^{t,n} \leq x$.

If the top and bottom rightmost label of $T^{(n-1)}$ are different, then $T^{(n-1)}(x) < l^{b,n-1} < x$ which is a contradiction to $x$ being the left-most point of the closure of the recurrent orbit. Hence we showed that there exists no $n \in \mathbb{N}$ such that $0 < r^{t,n} \leq x$.

We now show that there exists no $n \in \mathbb{N}$ such that $r^{t,n} = 0$. By contradiction, let $n \in \mathbb{N}$ be such that $d_n = 0$ and $d_{n-1} = 1$. As by the first part, $r^{t,n-1} > x$, it follows that $T^{(n-1)}$ is the first return map to some interval which contains an open neighbourhood around $x$, in particular, $x$ is contained both in a top and in a bottom interval of $T^{(n-1)}$. But as $d_{n-1} = 1$, the label of the top and bottom interval must be the same and contain a fixed point $x$, hence by Lemma \ref{cylinderlemma} the g-GIET is RV-stable, in particular, $d_n$ is never equal to zero.
\end{proof}

The two lemmas \ref{recurrencelemma} and \ref{leftendpointlemma} directly imply Proposition \ref{mainprop}:

\begin{proof}[Proof (of prop \ref{mainprop})] By Lemma \ref{recurrencelemma}, if $T:I^t \to I^b$ is an RV-stable g-GIET with renormalization limit $l$, then there exists a recurrent orbit closure $\mathcal{W}$ with leftmost endpoint $x$ such that $l=x$, and by Lemma \ref{leftendpointlemma} the leftmost endpoint of any other recurrent orbit closure is smaller or equal to $x$. To conclude the proof, we have to show that this inequality is strict, in particular, we claim that two distinct recurrent orbit closures $\mathcal{W}_1$ and $\mathcal{W}_2$ cannot have the same left-most endpoint. Indeed, note that by the proof of Lemma \ref{recurrencelemma} it follows that the left-most endpoint of a quasiminimal belongs to a non-trivially recurrent orbit. In particular, if either $\mathcal{W}_1$ or $\mathcal{W}_2$ is a closed orbit, then the claim is immediate. If both $\mathcal{W}_1$ and $\mathcal{W}_2$ are quasiminimals with the same left-most endpoint $x$, then the non-trivially recurrent trajectory through $x$ is contained both in $\mathcal{W}_1$ and in $\mathcal{W}_2$. However, we have seen in Section \S~\ref{sec:maierthm} that two quasiminimals which contain the same non-trivially recurrent orbit must be equal.
\end{proof}

\subsection{Renormalization limit  and winning letters} We have seen that Rauzy-Veech induction is able to detect the "rightmost" recurrent orbit of a g-GIET $T$. We would now like to determine whether this recurrent orbit is trivially or non-trivially recurrent. The proposition below provides a way of doing so using the set of letters that win infinitely often.

\begin{prop} \label{winningletterstrichotomy}Let $T:I^t \to I^b$ be an infinitely renormalizable g-GIET with renormalization limit  $l$. Let $\mathcal{A}_{\infty} \subset \mathcal{A}$ denote the set of letters that win infinitely often. We have the following trichotomy:

\begin{enumerate}
    \item $\mathcal{A}_{\infty} = \emptyset$ if and only if $I^t$ is a domain of transition,
    \item $|\mathcal{A}_{\infty}| =1$ if and only if $l$ is the leftmost endpoint of a periodic orbit,
    \item $|\mathcal{A}_{\infty}| > 1$ if and only if $l$ is the leftmost endpoint of a quasiminimal.
\end{enumerate}
\end{prop}

\begin{proof} Lemma \ref{recurrencelemma} together with Lemma \ref{leftendpointlemma} implies that $\mathcal{A}_{\infty} \neq \emptyset$ if and only if $I^t$ is not a domain of transition, which is exactly the statement of (1).

For (2), if $l=x_0$ where $x_0$ is the leftmost endpoint of a periodic orbit, then once $n$ is large enough so that $r^{t,n} < $ min $\{T^l(x_0), 0 < l \leq p\}$, then $x_0$ is contained in a cylinder and hence there is only one letter winning infinitely often.

Conversely, let $w$ be the unique letter that wins infinitely often. Let $n \in \mathbb{N}$ be such that $T^{(n)}$ is RV-stable. As by Lemma \ref{winninglettersintersectl} the labels of the intervals which contain $l$ have to win infinitely often, it follows that $l \in I_w^{t,n} \cap I_w^{b,n}$. But then $l$ is a fixed point for $T^{(n)}$ and hence is the leftmost endpoint of a periodic orbit of $T$.

The last statement (3) follows directly from (2) and the fact that by Lemma \ref{recurrencelemma} $l$ must be the leftmost endpoint of either a periodic orbit or a quasiminimal.
\end{proof}

\section{Decomposition into domains of recurrence} \label{sec:decompositionproof}

\subsection{Finding a domain of recurrence}  We have seen that if one applies Rauzy-Veech induction to a g-GIET $T:I^t \to I^b$ infinitely many times, the induction accumulates at the leftmost endpoint of a recurrent orbit closure $\mathcal{W}$ of $T$. We now want to determine a subset of $I^t$ that contains $\mathcal{W}$ and at the same time is a domain of recurrence for $T$. In particular, we want to prove Proposition \ref{findingdomrec} below. For the definition of the restriction of a g-GIET we refer the reader to Section \S~\ref{sec:domains}.

\begin{prop} \label{findingdomrec} Let $T: I^t \to I^b$ be RV-stable with renormalization limit  $l$ and corresponding recurrent orbit closure $\mathcal{W}$. Let $\mathcal{A}_{\infty}$ be the set of letters that win infinitely often and define $I_{\infty}^t := \bigcup_{\alpha \in \mathcal{A}_{\infty}} I_{\alpha}^t$. Then $I_{\infty}$ is a domain of recurrence for the restriction $T|_{I_{\infty}^t}: I_{\infty}^t \to T(I_{\infty}^t)$. Furthermore,

\begin{itemize}
    \item if $\mathcal{W}$ is a fixed point, then $T|_{I_{\infty}^t}$ contains no recurrent orbits other than fixed points.
    \item if $\mathcal{W}$ is a quasiminimal, then $T|_{I_{\infty}^t}$ contains a unique quasiminimal and all recurrent orbits are non-trivially recurrent.
\end{itemize}
\end{prop}

\subsubsection{Wandering intervals} In order to prove Proposition 6.1, we first relate the letters that loose infinitely often to wandering intervals. For the definition of a wandering interval we refer the reader to Definition \ref{def:wanderinginterval}.

\begin{lem}\label{lem:wanderinginterval}
    Let $T:I^t \to I^b$ be RV-stable, let $\alpha \in \mathcal{A} \backslash \mathcal{A}_{\infty}$ be a label that always looses. If $\alpha$ always looses at the bottom (top), then $I_{\alpha}^t$ is a forward (backward) wandering interval. If $\alpha$ looses both at the top and at the bottom, then $I_{\alpha}^t$ is both a forward and backward wandering interval. In both cases, $I_{\alpha}^t$ does not contain a recurrent orbit.
\end{lem}

\begin{proof} If $\alpha$ always looses at the bottom, then from the definition of RV-induction it follows that for any $n \in \mathbb{N}$, the tower $\text{Tow}(I_{\alpha}^{t,n})$ is equal to
\begin{align}
\text{Tow}(I_{\alpha}^{t,n}) = \{ I_{\alpha}^{t}, T(I_{\alpha}^{t}), \dots, T^{k_n}(I_{\alpha}^{t}) \},
\end{align}
i.e it consists of the forward images of $I_{\alpha}^{t}$ up to time $k_n$, where $k_n$ is the height of the tower at time $n$. In particular, since all the floors of a tower are disjoint, all forward images of $I_{\alpha}^{t,n}$ are disjoint. Similarly, if $\alpha$ always looses at the top, then the corresponding tower consists of all backward images of $I_{\alpha}^t$ up to time $k_n$ and hence all backward images of $I_{\alpha}^{t,n}$ are disjoint. Hence, $I_{\alpha}^t$ corresponds to a wandering interval (either in the future, past, or both).
\end{proof}

\subsubsection{The quasiminimal case}

In this subsection, we always assume that $\mathcal{W}$ is a quasiminimal. We have seen in Section \S~\ref{sec:somecombinatorialarguments} that for $T:I^t \to I^b$ RV-stable, there exists $\alpha, \beta \in \mathcal{A}_{\infty}$ such that $l \in I_{\alpha}^{t} \cap I_{\beta}^{b}$. We now want to show that $I_{\alpha}^{t} \cap [0,l)$ is in fact a wandering interval for $T$ all of whose iterates are contained in $(l,1)$.

Consider $\hat{T}$, where $\hat{T}$ has been obtained from $T$ by introducing a "fake" singularity at $l$ as follows: define
\begin{align*}
    I_{\alpha_1}^t = [0,l) \cap I_{\alpha}^t, \\
    I_{\alpha_2}^t = [l,1) \cap I_{\alpha}^t, \\ I_{\alpha_1}^b = T(I_{\alpha_1}^t),\\
    I_{\alpha_2}^b = T(I_{\alpha_2}^t),
\end{align*}
and let $\hat{T}$ be obtained from $T$ by replacing $I_{\alpha}^t$ with $I_{\alpha_1}^t$ and $I_{\alpha_2}^t$ as well as replacing $I_{\alpha}^b$ with $I_{\alpha_1}^b$ and $I_{\alpha_2}^b$. In Lemma \ref{alwaysloose} below, we want to show that if $l$ is the leftmost endpoint of a quasiminimal, then $\alpha_1$ always looses under RV induction on $\hat{T}$.

\begin{lem} \label{alwaysloose} Let $T:I^t \to I^b$ be an RV-stable g-GIET with renormalization limit  $l$, where $l$ is the leftmost endpoint of a quasiminimal. Let $\hat{T}: I^t \to I^b$ and $I_{\alpha_1}^t, I_{\alpha_2}^t$ be as above. Then the letter $\alpha_1$ always looses for Rauzy-Veech induction on $\hat{T}$.
\end{lem}

\begin{proof} Note that as in the proof of Proposition \ref{intervalthatcontainslremainsunchanged}, there exist $\alpha, \beta \in \mathcal{A}$ such that $l \in I_{\alpha}^{t,n} \cap I_{\beta}^{t,n}$ and $\alpha \neq \beta$ since we assume that $\mathcal{W}$ is a quasminimal. But then we have shown that only case b) and c) in the proof of Proposition \ref{intervalthatcontainslremainsunchanged} may arise for any step of the RV-induction on $T$. If we now consider $\hat{T}:I^t \to I^b$ (i.e the g-GIET obtained from $T$ by adding a fake singularity at $l$), then $\alpha_1$ plays for the Rauzy-Veech induction on $\hat{T}$ only after case b) arises. But then $\alpha_1$ must always loose.
\end{proof}

In particular, by Lemma \ref{lem:wanderinginterval}, $I_{\alpha_1}^t$ is a forward wandering interval for $\hat{T}$ (and hence also for $T$) and does not contain a recurrent orbit. We further remark that analogously, by considering $T^{-1}$, if $l \in I_\beta^b$ and $I_{\beta_1}^b$, $I_{\beta_2}^b$ are obtained by introducing a fake singularity at $l$ at the bottom then the letter $\beta_1$ always looses under Rauzy-Veech induction. Hence, $I_{\beta_1}^b$ is a backward wandering interval for $T$ that does not contain a recurrent orbit.


\subsubsection{The general case} We now assume $\mathcal{W}$ to be again either a periodic orbit or a quasiminimal. Before we turn to the proof of Proposition 6.1, we first prove the following auxiliary lemma below.

\begin{lem} Let $T: I^t \to I^b$ be RV-stable with renormalization limit  $l$ and corresponding orbit closure $\mathcal{W}$. Then $T$ restricted to $I_{\infty}^t$ is RV-stable.
\end{lem}

\begin{proof} If $\mathcal{W}$ is a periodic orbit, then $T$ restricted to $I_{\infty}^t$ consists of one interval which always wins, in particular $T$ is RV-stable.

If $\mathcal{W}$ is a quasiminimal, then the gaps of $T$ restricted to $I_{\infty}^t$ are either in the complement of $I^t$ or in $I^t \backslash \mathcal{A}_{\infty}$. Whenever a gap plays, if it is in the former case, it must loose as $T$ is RV-stable. In the latter case, it must also loose as it replaces an interval that always looses. Hence, the number of intervals of $T$ restricted to $\mathcal{A}_{\infty}$ is constant, and as all intervals win infinitely often, $T$ is RV-stable.
\end{proof}

We now prove the following lemma:

\begin{lem} Let $T: I^t \to I^b$ be RV-stable. Then $I_{\infty}^t$ is a domain of recurrence for $T|_{I_{\infty}^t}$.
\end{lem}

\begin{proof} By the previous lemma, $T$ restricted to $I_{\infty}^t$ is again RV-stable. Now assume that there exists a transient point $x \in I_{\infty}^t$, i.e there exist $m,l \in \mathbb{N}$ such that $T^{m+1}(x) \notin I^t$ and $T^{-(l+1)}(x) \notin I^b$, but \{$T^{-l}(x), \dots, x, \dots ,T^m(x)\} \in I^t \cap I^b$. There are two cases: either the transient orbit is entirely contained in $[l,1)$, or some point of the transient orbit lies in $[0,l)$. If the entire orbit is contained in $[l,1)$, let
\begin{align*}
x_0 &:= \text{min}\{T^{-l}(x), \dots, x, \dots T^m(x)\},\\
x_1 &:= \text{min}\{T^{-l}(x), \dots, x, \dots T^m(x)\} \backslash \{x_0\}.
\end{align*}

\begin{figure}[h]
\centering
\includegraphics[width=1.0\textwidth]{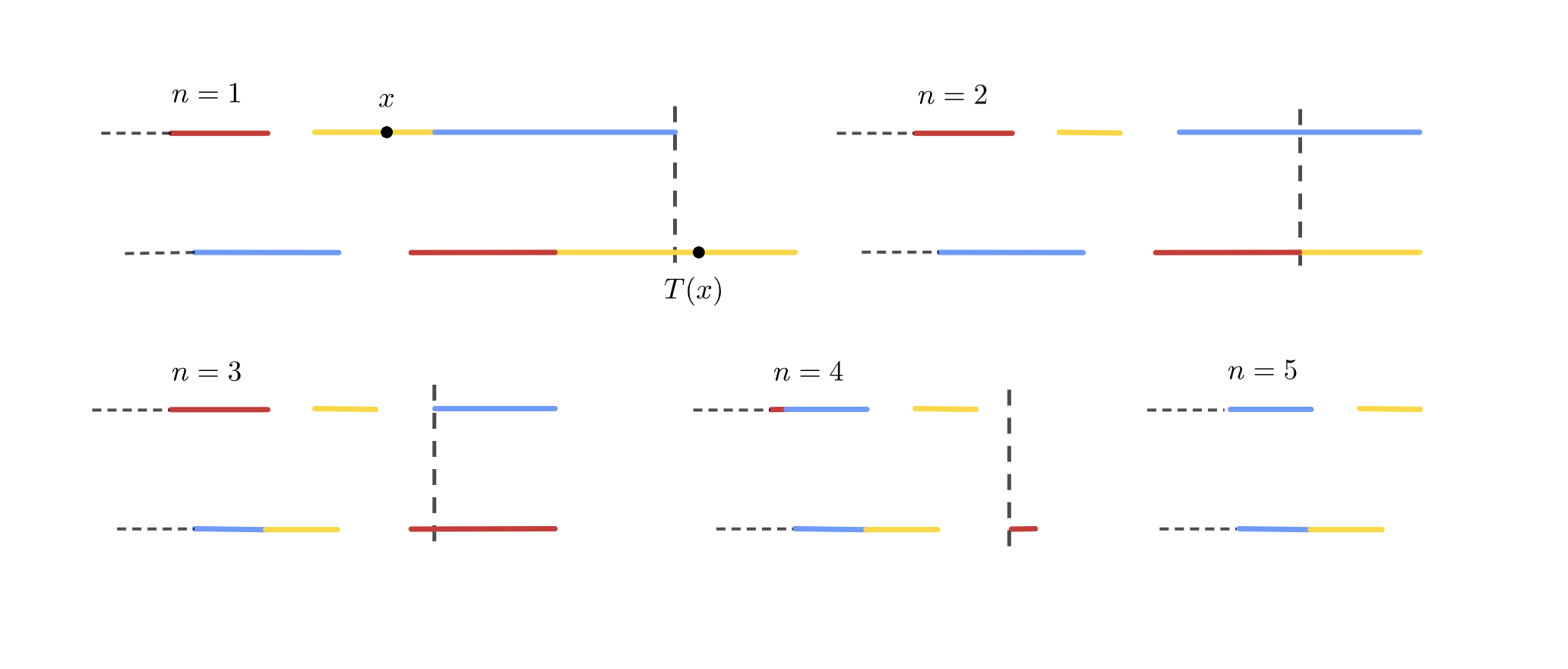}
\caption{\label{fig:recurrence}An example of a transient point $x \in I^t$ for which $T(x) \notin I^t$ and $x \notin I^b$ (i.e $l=1, m=0$). The induction removes the red interval.}
\end{figure}

As the renormalization limit  is equal to $l$, there exists $n \in \mathbb{N}$ such that $x_0 \leq r^{t,n} \leq x_1$. As $T^{(n)}$ is the first return map of $T$ to $[0, r^{t,n})$ and as $x_0$ is the minimum of the finite orbit through $x$, neither $T^{(n)}(x_0)$ nor $(T^{(n)})^{-1}(x_0)$ are well-defined, hence $T^{(n)}$ has a gap at $x_0$ both at the top and at the bottom. Both gaps persist under further iteration of RV-induction until one of the gaps wins (see Figure \ref{fig:recurrence}). But from the definition of RV-induction, when a gap wins the induction removes a letter (see Case 2b in $\S$ \ref{sec:RV g-GIETs}), which is a contradiction to $T$ being RV-stable, i.e in particular the number of intervals of $T$ should remain constant.

If a point of the orbit is contained in $[0,l)$, it must pass through $I_{\alpha_1}^t$ or $I_{\beta_1}^b$. However, by the proof of Lemma \ref{lem:wanderinginterval}, all iterates $\{ T^{n}(I_{\alpha_1}^t), n \in \mathbb{N} \}$ and all iterates  $\{ T^{-n}(I_{\beta_1}^b), n \in \mathbb{N} \}$ are contained in $[l,1) \cap I_{\infty}^t$, hence a point in $I_{\alpha_1}^t$ or $I_{\beta_1}^b$ can not be transient.
\end{proof}

Finally, we are able to prove Proposition 6.1.

\begin{proof}[Proof (of Proposition 6.1)] From the previous lemma, we know that $I_{\infty}^t$ is a domain of recurrence for $T|_{I_{\infty}^t}$. Hence any recurrent orbit of $T$ either is entirely contained in  $I_{\infty}^t$ or is disjoint with $I_{\infty}^t$, as otherwise by recurrence we could construct a transient point in $I_{\infty}^t$ for $T|_{I_{\infty}^t}$.

Therefore, since $l \in \mathcal{W} \,\cap\, I_{\infty}^t$, it follows that $\mathcal{W} \subset  I_{\infty}^t$. If $\mathcal{W}$ is a fixed point, then $I_{\infty}^t$ consists of one interval. Hence any recurrent orbit that intersects this interval must by the previous observation lie entirely within this interval and is hence a fixed point.

If $\mathcal{W}$ is a quasiminimal, then we claim that $I_{\infty}^t$ does not intersect any other quasiminimal $\hat{\mathcal{W}}$: indeed, if this was the case, then $\hat{\mathcal{W}}$ must lie entirely in $I_{\infty}^t$ and as its leftmost endpoint must be smaller than $l$ by Proposition \ref{leftendpointlemma}, it must intersect $I_{\alpha_1}^t$ from Lemma \ref{alwaysloose} which is a contradiction to Lemma \ref{alwaysloose} which states that $I_{\alpha_1}^t$ does not contain any recurrent orbits.
\end{proof}

\subsection{Renormalization scheme and proof of main theorems} \label{sec:proofs}  The proof of Theorem 1.1, 1.2 and 1.3 will follow by induction, namely we will define a \emph{renormalization scheme} which consists of iterating the elementary steps from the previous section finitely many times to obtain the desired dynamical decomposition.


\smallskip
\noindent {\bf Basic step of the renormalization scheme:} Let $T:I^t \to I^b$ be a g-GIET on $d\geq 1$ intervals.  Recall that we denote by $d_n$ the number intervals exchanged by $T^{(n)}$.  Consider the (possibly finite) sequence of g-GIETs $T^{(1)}, T^{(2)}, T^{(3)} \dots$ obtained when iterating RV induction starting from $T$.

    \vspace{.9mm}
\noindent
Consider the three following possible cases (which are mutually exclusive, as we show in  Lemma~\ref{lemma:exclusive} below):
\begin{enumerate}[]
    \item
    \textbf{Induction detects two singular orbits:} {\it there exists a smallest $n \in \mathbb{N}$ such that the rightmost top and bottom label of $T^{(n)}$ are equal, i.e $\pi_t^n(d_n) = \pi_b^n(d_n)$ = $\alpha$, and the right endpoints are equal, i.e $r^{t,n} = r^{b,n}$. } In this case, $T|_{I_{\alpha}^t}$ is a g-GIET on one interval for which all recurrent orbits are (possibly singular) fixed points, and $I_{\alpha}^t$ is a domain of recurrence for $T|_{I_{\alpha}^t}$.

    \vspace{1.3mm}

    \item \textbf{Induction accumulates at recurrent orbit closure:} {\it there exists $n \in \mathbb{N}$ such that $T^{(n)}$ is RV-stable with renormalization limit $l$}. Here $l>0$ or $l=0$ are both possible.
    In this case, let $I_{\infty}^{t}$ be the set of top intervals that win infinitely often. From Proposition \ref{findingdomrec} we know that $I_{\infty}^t$ is a domain of recurrence for the restriction $T|_{I_{\infty}^t}$. Furthermore, $T|_{I_{\infty}^t}$ is either a g-GIET on one interval for which all recurrent orbits are fixed points, of a g-GIET on more than one interval which contains a unique quasiminimal and for which all recurrent orbits are non-trivially recurrent. In the latter case, since the rotation number of $T|_{I_{\infty}^t}$ is infinite complete, it follows from Theorem \ref{thm:PoincareYoccoz} that $T|_{I_{\infty}^t}$ is semi-conjugated to a minimal IET with the same rotation number.

     \vspace{1.3mm}
    \item \textbf{Induction trivializes:} {\it there exists $n \in \mathbb{N}$ such that $d_{n+1}=0$, and we are not in case $(1)$.} In this case, $T$ does contain any recurrent orbit (all orbits either enter or escape from a gap) and $I^t$ is a transition domain.
\end{enumerate}
We now show that these three cases describe all the possible outcomes for the first step of the scheme:
\begin{lem}\label{lemma:exclusive}
For any map in the (possibly finite) sequence $T^{(1)}, T^{(2)}, \dots$ obtained by applying Rauzy-Veech induction to a g-GIET $T$ on $d\geq 1$ intervals, one and only one of the three possibilities $(1)$, $(2)$ and $ (3)$ descibed above holds.
\end{lem}
\begin{proof}
Assume first that there is no $n$ such that $d_{n+1}<d_n$, i.e.~the number of exchanged intervals is constant along the RV orbit. In this case,  $T$ is infinitely renormalizable (since the induction can stop only after decreasing the number of intervals, see \S~\ref{sec:RV g-GIETs}). Recalling  Remark~\ref{rk:RVstable}, if $n$ is sufficienly large, $T^{(n)}$ is RV-stable, so case $(3)$ holds.

If on the other hand there exists an $n$ such that $d_{n+1}<d_n$  (i.e. the number of intervals at some point decreases), let $n_1<n_2< \dots < n_k$ be all positive integers $n$ such that $1\leq d_{n+1}<d_n$
(here $1\leq k< d$ since the number of intervals can only decrease $d-1$ times).  
  From the description of one step of RV induction in  \S~\ref{sec:RV g-GIETs}, since  the number of intervals decreases either when two intervals play and neither wins (Case 1b in \S~\ref{sec:RV g-GIETs}) or when a gap wins against an interval (Case 2b in \S~\ref{sec:RV g-GIETs}), for each  $n\in \{ n_1,\dots, n_k\}$, $T^{(n)}$ is in one of these two cases.
If  there exists $n \in \{ n_1, \dots , n_k\}$  such that $T^{(n)}$ is in Case 1b), then we are in $(1)$.  If not, consider $n_k$ and the corresponding $d_{n_k}$. If $ d_{n_{k}}=1$, then, after  applying Step 2b, $d_{n_{k}+1}=0$ and hence $(3)$ holds.
Finally, if  instead $d_{n_k}>1$, the number of exchanged intervals is constant for any $n\geq n_k$ so reasoning as at the beginning of the proof, by Remark \ref{rk:RVstable}, there exists $n\geq n_k$ such that $T^{(n)}$ is RV-stable and again $(2)$ holds. This concludes the proof.
\end{proof}

\smallskip
\noindent {\bf Output of the first step}:
If $T$ is in case (1), i.e.~ the renormalization detected two singular orbits and  $I^t_\alpha$ as a domain of trivial recurrence,  or in  case (2), i.e.~the induction accumulates at a recurrent orbit closure and identifies $I^t_\infty$, we can \emph{record} the outcome of the first step of the renormalization scheme, by defining:
$$
E_1^t:= \begin{cases} I^t_\alpha, & \text{in \ Case\ } (1),\\
I_{\infty}^t  & \text{in \ Case\ } (2).
\end{cases}
$$
Thus, $E_1^t$ is the domain of recurrence detected by the induction in the first step of the scheme. In Case (3),  the induction detected that $I^t$ is a domain of transition, so there no domain of recurrence is recorded as output.

 \begin{rem}\label{rk:output} Remark that in both cases the restricted map $T_1$ defined by
 $$T_1 := T^{(n)}|_{E_1^t}: E_1^t \to E_1^b$$
 satisfies the assumptions of either i) or ii) of Theorem 1.2.
 \end{rem}

\smallskip
\noindent {\bf Stopping conditions}: the renormalization scheme is complete, in the following cases:
\begin{itemize}
\item[S1)]  if
$T$ is in Case (1) and $d_n=1$ (since in this case $d_{n+1}=0$);
\item[S2)] if $T$ is in Case (2) and the renormalization limit $l=0$;
\item[S3)] if $T$ is in Case (3), i.e.~the induction trivializes.
\end{itemize}
\smallskip

\noindent {\bf Further steps of the  scheme}:  If none of the \emph{stopping conditions} S1), S2), or S3) defined above hold, we \emph{iterate} the renormalization scheme as follows.   Consider the map obtained by restricting $T^{(n)}$ to the complement of $E_1$,
\begin{equation}\label{eq:newstepmap} T^{(n)}|_{I^t \backslash E_1^t}: I^t \backslash E_1^t \to I^b \backslash E_1^b.\end{equation} We can now repeat the basic step of the scheme taking this new g-GIET as initial $T$.   By induction, we can keep iterating the basic step of the scheme in the same way, unless one of the stopping condition holds.
 Note  that by construction, restricting $T^{(n)}$ to the complement of $E_1$,
we get a  g-GIET on strictly less intervals than $T^{(n)}$. Therefore, after a finite number of steps (at most $d$),
the renormalization scheme stops.

As output of the steps of the renormalization scheme we obtain $k\leq d$ recurrence domains $E^t_i$, $1\leq i\leq k$ and, correspondingly, the  $k$ g-GIETs
\begin{equation}\label{eq:koupput}
T_i:E_i^t \to T_i(E_i^t), \qquad 1 \leq i \leq k. \end{equation}
\begin{proof}[Proof of Theorems~\ref{thm:decomposition}  and~\ref{thm:structure} ]
We will prove at the same time that the renormalization scheme produces  the decomposition in Theorem~\ref{thm:decomposition},  and that the domains have the structure claimed by Theorem~\ref{thm:structure}.

 \vspace{1mm}
\noindent{\it Decomposition pieces.}  Given a GIET $T;I^t\to I^b$ (which is in particular a g-GIET), consider the  domains of recurrence $E_i^t\subset I^t$ and the associated  g-GIETs $T_i:E_i^t \to T_i(E_i^t)$ for $1 \leq i \leq k$  obtained by starting from $T$ and following the renormalization scheme described above. As already remarked (see Remark~\ref{rk:output}), all these g-GIETs satisfy the assumptions of Theorem~\ref{thm:structure}.

Now for each $1 \leq i \leq k$, consider the tower representation $R_i^t := $ Rep$(E_i^t)$ of the g-GIET $T_i$ over the original g-GIET $T$ and set $D^t := [0,1) \backslash \sqcup_{i=1}^k R_i^t$, so that by construction
$[0,1) = R_1^t \cup \dots \cup R_k^t \cup D$.

\vspace{1mm}
\noindent{\it Properties of the pieces of the decomposition.}  We claim that each $R_i^t$ defined above satisfies either (i) or (ii) of the decomposition Theorem~\ref{thm:decomposition}.
 Indeed, if $T_i$ has a (possibly singular) fixed point and all recurrent orbits are fixed, then $T|_{R_i^t}$ has a (possibly singular) periodic orbit and all recurrent orbits are periodic of the same period. Similarly, if $T_i$ has a unique quasiminimal and all recurrent orbits are non-trivially recurrent, then $T|_{R_i^t}$ contains a unique quasiminimal and all recurrent orbits are non-trivially recurrent.

The fact that $R_i^t$ are a domain of recurrence for $1, \leq i \leq k$ follows from the fact that corresponding $E_i^t$ are a domain of recurrence. Furthermore, $D^t := [0,1) \backslash \sqcup_{i=1}^k R_i^t$ is a domain of transition for $T$, since all recurrent orbits of $T$ are contained in some $E_i^t$, and hence also in some $R_i^t$, for $1 \leq i \leq k$.
Finally, the tower structure described by Theorem~\ref{thm:structure} is automatic by the definition of the $R^t_i$.

\vspace{1mm}
\noindent{\it Pairwise disjointness.}
It is left to prove that the tower representations $R_1^t, \dots, R_k^t$ are all disjoint: Assume by contradiction there exists $i > j$ such that $R_i^t \cap R_j^t \neq \emptyset$. Then, by Remark~\ref{rk:tower}, this implies that there exists a point $x \in E_i^t$ and $n \in \mathbb{N}$ such that $T^{n}(x) \in R_j^t$. By construction, both $E_i^t$, $E_i^b$ are disjoint with $R_j^t$. Hence, since $x \in E_i^t$ is not in $R_j^t$, and $R_j^t$ is a region of recurrence, it follows $T^{m}(x) \in R_j^t$ for all $m > n$, since otherwise the orbit through $x$ would be transient for $R_j^t$. But this is a contradiction since $x \in E_i^t$, in particular the orbit through $x$ must return to $E_i^b$, hence there exists $m > n$ such that $T^m(x) \notin R_j^t$. This concludes the proof of both theorems.
\end{proof}

\subsection{Bounds on periodic orbits and ergodic measures.}\label{sec:ergodic_bound}
Let us now prove Theorem~\ref{thm:boundonperiodicdomains} and Corollary~\ref{cor:ergodic_bound}.
\begin{proof}[Proof of Theorem~\ref{thm:boundonperiodicdomains}] If $d_i$ for $1 \leq i \leq k$ denotes the number of intervals of the domains $E_i^t$ defined in the proof of Theorem~\ref{thm:decomposition} and~\ref{thm:structure}, then by Proposition \ref{winningletterstrichotomy} if $E_i$ contains only periodic recurrent trajectories, then $d_i = 1$, whereas if $E_i$ contains a unique quasiminimal, then $d_i \geq 2$. Because the map described in \eqref{eq:newstepmap} has exactly a number of $d_i$ less intervals than the map considered in the previous induction step, it follows that if $p$ denotes the number of periodic domains of $T$, $q$ denotes the number of quasiminimal domains and $d$ denotes the number of intervals of the original GIET $T$, then

$$q < \left\lfloor \frac{d-p}{2} \right\rfloor.$$
\end{proof}
\begin{proof}[Proof of Corollary~\ref{cor:ergodic_bound}] Remark that each periodic domains supports only invariant measures which come from  periodic orbits, and hence are atomic. 
By the  Structure Theorem~\ref{thm:structure} for each of the $q$ quasiminimal domains of $T$, 
$d_i$, $1\leq i\leq q$ are the intervals of the IETs $T_i$ given by the Structure Theorem~\ref{thm:structure} for each of the $q$ quasiminimal domains, 
 Thus, for each $1\leq i\leq q$, 
the ergodic measures of the base g-GIET $T_i$
are given by the pullback of the ergodic measures of the semi-conjugated IET.
From a classical bound for the number of ergodic invariant measures of an  IET of $d_i$ intervals, which has at most $[d_i/2]$ invariant probability ergodic measures, we then get the bound. 
\end{proof}
\begin{rem}\label{rk:genus_bound}    We can make the bound of Corollary~\ref{cor:ergodic_bound} bound sharper: indeed, the sharp bound for number of ergodic invariant measures of an IET is known to be the \textit{genus} of a surface of which the IET is a Poincar{\'e} section; this is fully encoded (and can be reconstructed from)  the permutation of the IET (as proved by Veech \cite{Veech}, see for instance Proposition 9 in \cite{Yoccoz_Course}). 
\end{rem}



\subsection{Further Remarks and questions}
We conclude the paper with a few remarks and open questions.

\subsubsection{Non-uniqueness of the decomposition}\label{sec:notunique} As already remarked in the introduction, the decomposition in Theorem~\ref{thm:decomposition} is \emph{not} unique. This can be clearly seen from the
proof: when building the regions $E^t_i$ (and consequently the corresponding recurrence regions $R^t_i$, which are tower representations over $E^t_i$) which arise from outcome $(2)$ of the basic step of the renormalization process, we chose arbitrarily  an $n_0$ sufficiently large such that $T^{(n_0)}$ is RV-stable. Choosing any $n\geq n_0$, we get a g-GIET $T^{(n)}$ which is also RV-stable.
 If $R^t_i$  contains a Cantor-like quasiminimal, this can lead to a smaller recurrence region.

Note, on the other hand, that the $R^t_i$ which are  domains of trivial recurrence are unique, as well as the domains of non-trivial recurrence on which the restriction is minimal (i.e.~all orbits are dense).


\subsubsection{Recovering a flow decomposition result}\label{sec:surface_bands}
From our combinatorial results about the dynamical decomposition of a GIET, one can recover the decomposition theorem proved by Gutierrez in \cite{gutierrez}. Given a flow on a surface $S$, one can consider a g-GIET  obtained as the first return map on a transversal segment $I$ of the flow and perform the renormalization scheme to obtain the decomposition given by Theorem~\ref{thm:decomposition}. From this, to get a decomposition of the surface, it is enough to saturate the regions $R_i$ by leaves of the flow flow. More precisely,  instead of considering the tower representations $R_i$ (which are a subset of the transversal segment), we can construct a submanifold $S_{R_i}$ of the surface $S$ by flowing each point in $E^t_i$ until its first return\footnote{If $R_i$ is a tower representation of $T_i: E_i^t \to E_i^b$, where $\mathcal{A}_\infty^i$ is the set of letters of $T_i$, then for any $\alpha \in \mathcal{A}_\infty^i$ consider the union of arcs of forward trajectories which connect $I_{\alpha}^t$ with $I_{\alpha}^b$. We consider the union of all these arcs over all $\alpha \in \mathcal{A}_\infty^i$ and denote their closure by $S_{R_i}$.}
to $E^n_i$ and then taking the closure. One can see that $S_{R_i}$ is a compact connected surface with boundary, which is obtained by a \emph{Ribbon bands} type of construction\footnote{The \emph{bands} here are the region foliated by all trajectories which start from $I^t_\alpha$ and land in $I^b_\alpha$, $\alpha\in \mathcal{A}^i_{\infty}$.}.
The  boundary $\partial S_{R_i}$ is a union of saddle connections, i.e.~trajectories which join singularities of the flow.
If $S_{\Omega_i}$ is the quasiminimal on $S$ obtained saturating  $\Omega_i \subset R_i$  with flow trajectories,
then $S_{\Omega_i} \subset S_{R_i}$ and there is no arc of trajectory contained in $S_{R_i} \backslash \partial S_{R_i}$ connecting two points of $\partial(S_{R_i})$. Thus, this decomposition of the surface into subsurfaces with boundaries satisfies the conclusion of Gutierrez's decomposition theorem.

\subsubsection{Extended combinatorics and combinatorial Euler characteristic}\label{sec:extended_classes}
The \emph{extended} combinatorial data (see Remark~\ref{rk:extended_combinatoics} in \S~\ref{sssect:comb}) can be used to code 
the location of gaps. It would be interesting to characterize which extended combinatorics appear performing the renormalization scheme, by analyzing the structure of the analogous of \emph{Rauzy-classes} for extended combinatorics. 

To a given extended combinatorial data, one can associate a \emph{suspension}, which is a surface with boundary (as the band complex described in the previous \S~\ref{sec:surface_bands}) 
by generalizing the construction of zippered rectangles\footnote{
A  topological surface with boundaries which \emph{suspend} a g-GIET can be obtained by gluing $|\mathcal A|$ rectangles $S_\alpha$, which correspond to each label $\alpha \in \mathcal A$. Each of these rectangles has two horizontal sides which are glued to the interval with respect to the g-GIET map.} by Veech \cite{Veech} (see also \cite{Yoccoz_Course}). Both the \emph{genus} $g$ and the \emph{Euler characteristic} $\chi$ of the suspension can be recovered purely from the extended combinatorics\footnote{The genus can be recovered by extending Veech construction and  then $\chi:=	2-2g-b $, where 
$b$ is the number of connected components of gaps, which correspond to the number of boundary components of the suspension).}.  
One can  also show purely combinatorially that this Euler characteristic is preserved by Rauzy--Veech induction. This provides for example a direct combinatorial proof that, if $g_i$, for $1\leq i\leq q$, denotes the genus of each domain of recurrence (as defined in  \S~\ref{sec:ergodic_bound}), we have that  $g_1 + \dots + g_n \le g$. 

\subsubsection{Wandering intervals in quasiminimals} Consider one of the domains of recurrence $R_i$ given by our decomposition (if there exist any) which is a quasiminimal domain (see case (2) of Theorem~\ref{thm:decomposition}). We claim that from the combinatorial rotation number it is not possible to distinguish whether $T$ restricted to $R_i$ is \emph{minimal}, or whether $R_i$ contains a Cantor-like quasiminimal. By the Structure Theorem~\ref{thm:structure}, $R_i$ is  a tower representation over a g-GIET $T_i: E_i^t \to E_i^b$ for which all letters win infinitely often. The generalization of Poincar{\'e}-Yoccoz Theorem then implies that $T_i$ is semi-conjugated to a minimal IET on the same number of intervals. The semi-conjugacy is a topological  conjugacy if and only if $T_i$ has no wandering intervals. In this case, $T$ restricted to $R_i$ is minimal (i.e.~all orbits are dense in $R_i$). Otherwise, the quasiminimal is a Cantor set and the gaps are wandering intervals.

While the combinatorial rotation number cannot distinguish a wandering interval from an orbit, we observe that, algorithmically, if one e.g.~performs the RV algorithm numerically, one can recognize the presence of wandering intervals: indeed, a tower contains a wandering interval if and only if there are floors whose width does not go to zero as we continue iterating RV induction. Hence in the limit it is possible to identify the presence of wandering intervals.

\subsubsection{Pure quasiminimals} 
As mentioned in Section \S~\ref{sec:flowsintro}, Levitt introduced the terminology \emph{pure quasiminimal} to denote the quasiminimals which are topologically either attractors or repellors. One can characterize pure quasiminimals using the \textit{direction} of their wandering intervals: a quasiminimal of Cherry-type which has only \emph{forward} (or only \emph{backward}) wandering intervals is \emph{pure}. On the other hand, if there are both forward and backward wandering intervals (as in the AIETs built by Marmi-Moussa-Yoccoz, see Chapter \S~\ref{sec:flowsintro}), the quasiminimal cannot be pure\footnote{One can show that the presence of a backward wandering interval $I$ makes it impossible for the quasiminimal $\Omega$ to be a topological attractor (see \S \ref{sec:attractorsrepulsors} for the definition of attractor) or vice versa, since any open neighbourhood $U$ of the quasiminimal must intersect all backward iterates of $I$ and hence cannot under forward iterates of $T$ accumulate to $\Omega$.}.

Purity has an importance in the study of our  decompositions, for example for the following reason: as remarked in \S~\ref{sec:notunique}, the quasiminimal domains in the decomposition given by Theorem~\ref{thm:decomposition} are not in general unique.  In the case when the quasiminimal domain contains a pure quasiminimal, one can find a sequence $R^{(n), t}_i$ for $n\geq n_0$ of nested regions whose intersection is exactly the quasiminimal.

Another importance of purity lies in the fact that it allows us to \textit{regularize} a g-GIET: if $R_i$ is a quasiminimal domain of $T$ where the quasiminimal is pure and an attractor (resp. repellor) it follows from the work of Gutierrez in \cite{gutierrez} 
that $T$ restricted to $R_i$ is conjugated to a g-AIET with arbitrary contracting (resp. dilating) factors\footnote{Reinterpreting the result of Gutierrez, written in the language of foliations, in the language of IETs, this can be seen as follows. 
Using 
Theorem ~\ref{thm:PoincareYoccoz}, $T$ restricted to $R_i$ is semi-conjugated to a minimal IET $\tilde{T}$ with the same rotation number, where wandering intervals are \textit{blown down} to points. We can then choose arbitrary contracting (dilating) factors and again \textit{blow up} these points to wandering intervals in (as in \cite{CAMELIER_GUTIERREZ_1997} or \cite{bressaud_persistence_2010}) 
in order to obtain a g-AIET with the chosen contracting (dilating) factors. This works as long as all factors are contracting (dilating), since this directly induces directly finiteness of the wandering intervals. The g-AIET obtained in this way is then conjugated to the original g-GIET, since both share the same rotation number and the same wandering intervals (meaning they have the same order under iteration of the map).}. In particular, $T$ restricted to $R_i$ is conjugated to a $\mathcal{C}^\infty$ g-GIET.

As a tool to study purity, it may be helpful to use 
  \emph{extended} combinatorics (see Remark~\ref{rk:extended_combinatoics})  and their 
  their Rauzy-classes (see \S~\ref{sec:extended_classes}), since the extended combinatorial datum 
  contains information on the location and number of gaps, so that one can in principle read from it if there are gaps only on top or only on bottom, and hence if the g-IET is pure. 

\subsubsection{Stopping times for the renormalization scheme}
Our renormalization scheme consists of applying RV induction until one can identify the first (\emph{right-most}) region of recurrence (either a periodic or quasiminimal) and then \emph{restarting} RV induction by recursion on a g-GIET with one less interval, obtained by restricting the previous one on a suitable complementary region. This scheme, although algorithmic \emph{in flavour}, is devised to give a \emph{theoretical} proof of the existence of a decomposition with the desired properties.

While RV induction can be easily implemented \emph{numerically}, there are at least two aspects in our decomposition scheme which are not algorithmically implementable. While if the right-most quasiminimal is a periodic domain, it  can be detected in finite time (see below), if the induction accumulates at a recurrent orbit closure (i.e.~the g-GIET is in (2) of the basic step of the scheme, see \S~\ref{sec:proofs}), then it is not clear if it is possible to introduce a \emph{stopping condition} that allows to detect the quasiminimal in finite time. Even if this was possible, though, the combinatorial rotation number is infinite for a region which is a tower representation over a g-GIET which is (semi-)conjugated to a minimal IET, so one cannot compute it in a finite number of steps of RV induction.

For periodic domains,  a simple \emph{stopping condition} can be given (to \emph{stop} the search for quasiminimals during a step of the scheme and restart RV induction). A periodic domain indeed is detected when the induction detects two singular orbits, or when in case (2) one has a stable g-GIET with $d=1$. In both of these cases, we reach a step of the induction when $\alpha_t=\alpha_b=\alpha$, i.e.~the last intervals of $I^t$ and $I^b$ have the same letter (in case (1) they have the same lenghts, while in the latter case one is contained in the other). Here one can \emph{stop} the RV induction run for this step (since nested intervals imply the presence of a fixed point).

One may hope that there was a similar stopping condition for quasiminimal domains, for example if one could detect an inclusion between the \emph{union} of the top  intervals $I^t_\alpha$ with $\alpha\in\mathcal{A}_\infty$ and the union of the bottom intervals  $I^b_\alpha$ with $\alpha\in\mathcal{A}_\infty$. A simple counterexample shows that this is not the case: one can for example take an AIET with a wandering interval and add a continuity interval to the AIET (marking the endpoints of the wandering interval as \emph{fake} discontinuities of the g-GIET).  Then, the interval of continuity which is in fact a wandering interval is always loosing, but there is no inclusion between the remaining unions of top and bottom intervals.
On the other hand, if a quasiminimal is \emph{pure}, it may be possible to show that it can be detected with a finite-time stopping condition. We hope to come back to this question in further work.

\subsection*{Acknowledgements}
The authors would like to thank Selim Ghazouani and Michael Bromberg for many inspiring discussions, especially at the very early stages of this project.  S.~S.\ is currently supported by the SNSF grant Number $200021$\textunderscore$188617$. 
 C.~U.\ acknowledges the SNSF Consolidator Grant TMCG $2$\textunderscore$213663$ 
 as well as the SNSF Grant Number $200021$\textunderscore$188617$. 

\newpage
\bibliographystyle{alpha}  
\bibliography{references}

\end{document}